\documentclass[a4paper, 11pt, english]{article}
\usepackage{graphicx}
\usepackage{stmaryrd}
\usepackage[a4paper]{geometry}
\usepackage{amsmath,amsfonts,amssymb,amsthm,epsfig,epstopdf,url,array}
\numberwithin{equation}{section}
\usepackage{rotating}
\usepackage[colorlinks=true,citecolor=red,linkcolor=blue,pdfpagetransition=Blinds,backref=page]{hyperref}
\usepackage{cleveref}
\usepackage{nameref}
\usepackage{enumitem}

\usepackage{comment}
\Crefname{paragraph}{Section}{Sections}
\usepackage{fancyhdr}
\usepackage{cases}
\usepackage{mathrsfs}

\usepackage{pgfplots}

\usepackage{fullpage}

\newcommand\inter[1]{\left\llbracket #1\right\rrbracket}

\DeclareMathOperator*{\esssup}{ess\,sup}

\newcommand{\dive}[1]{\mathrm{div}}

\newcommand{\TT}{\mathbb T^d}
\newcommand{\dif}[1]{\partial^{#1}}

\DeclareMathOperator{\spn}{span}

\providecommand{\keywords}[1]{\noindent {\textit{Keywords:}} #1}

\usepackage[colorlinks=true,citecolor=red,linkcolor=blue,pdfpagetransition=Blinds,backref=page]{hyperref}
\usepackage{aliascnt}
\usepackage{cleveref}

\theoremstyle{plain}

\newtheorem{proposition}{Proposition}[section]
\newaliascnt{theo}{proposition}
\newtheorem{theo}[theo]{Theorem}
\aliascntresetthe{theo}

\newaliascnt{lemma}{proposition}
\newtheorem{lemma}[lemma]{Lemma}
\aliascntresetthe{lemma}

\newaliascnt{defprop}{proposition}

\aliascntresetthe{defprop}

\newaliascnt{cor}{proposition}
\newtheorem{cor}[cor]{Corollary}
\aliascntresetthe{cor}

\newaliascnt{rmk}{proposition}
\newtheorem{rmk}[rmk]{Remark}
\aliascntresetthe{rmk}

\theoremstyle{definition}

\newaliascnt{defi}{proposition}
\newtheorem{defi}[defi]{Definition}
\aliascntresetthe{defi}

\newaliascnt{app}{proposition}

\aliascntresetthe{app}

\newaliascnt{claim}{proposition}

\aliascntresetthe{claim}

\newaliascnt{ass}{proposition}

\aliascntresetthe{ass}

\crefname{proposition}{Proposition}{Propositions}
\Crefname{proposition}{Proposition}{Propositions}

\crefname{theo}{Theorem}{Theorems}
\Crefname{theo}{Theorem}{Theorems}

\crefname{lemma}{Lemma}{Lemmas}
\Crefname{lemma}{Lemma}{Lemmas}

\crefname{defprop}{Definition--Proposition}{Definition--Propositions}
\Crefname{defprop}{Definition--Proposition}{Definition--Propositions}

\crefname{cor}{Corollary}{Corollaries}
\Crefname{cor}{Corollary}{Corollaries}

\crefname{rmk}{Remark}{Remarks}
\Crefname{rmk}{Remark}{Remarks}

\crefname{defi}{Definition}{Definitions}
\Crefname{defi}{Definition}{Definitions}

\crefname{app}{Application}{Applications}
\Crefname{app}{Application}{Applications}

\crefname{claim}{Claim}{Claims}
\Crefname{claim}{Claim}{Claims}

\crefname{ass}{Assumption}{Assumptions}
\Crefname{ass}{Assumption}{Assumptions}

\def\dx{\,\textnormal{d}x}
\def\dt{\textnormal{d}t}
\def\d{\textnormal{d}}

\newcommand{\vertiii}[1]{{\left\vert\kern-0.25ex\left\vert\kern-0.25ex\left\vert #1 
    \right\vert\kern-0.25ex\right\vert\kern-0.25ex\right\vert}}

\makeatletter
\let\original@addcontentsline\addcontentsline
\newcommand{\dummy@addcontentsline}[3]{}
\newcommand{\DeactivateToc}{\let\addcontentsline\dummy@addcontentsline}
\newcommand{\ActivateToc}{\let\addcontentsline\original@addcontentsline}
\makeatother

\newcommand{\norm}[1]{\left\|#1\right\|}
\begin{document}

\title{Global controllability of the  Cahn-Hilliard equation\thanks{This work has received support from UNAM-DGAPA-PAPIIT grant IN117525 (Mexico). Corresponding author: Víctor Hernández-Santamaría. Email: \texttt{victor.santamaria@ciencias.unam.mx}}}
\author{V\'ictor Hern\'andez-Santamar\'ia\thanks{V. Hern\'andez-Santamar\'ia is supported by the Project CBF2023-2024-116 of SECIHTHI and by UNAM-DGAPA-PAPIIT grants IA103826 and IN102925 (Mexico).} \and Subrata Majumdar\thanks{Subrata Majumdar is supported by the UNAM Postdoctoral Program (POSDOC).} \and Luz de Teresa}

\maketitle

\begin{abstract}
{
This paper investigates global controllability properties of the
Cahn--Hilliard equation with a standard polynomial double-well
potential, posed on the $d$-dimensional flat torus $\mathbb T^d$. We first establish small-time global approximate controllability by
means of controls acting on finitely many Fourier modes. The argument
uses the nonlinearity constructively, through asymptotic and saturation
properties, to generate new Fourier directions. We then prove null controllability of the linearized system using a spatially localized control supported on an arbitrary measurable subset of positive Lebesgue measure. The proof relies on quantitative propagation of smallness estimates and the dissipative fourth-order
dynamics. For dimensions $d\in\{1,2,3\}$, we further derive local null
controllability of the nonlinear equation via a fixed-point argument. Combining the global approximate and local null controllability results, we obtain global null controllability for arbitrary initial data and every $T>0$. The global control is constructed by concatenating an initial waiting stage, a finite-dimensional Fourier control, and a spatially localized control. To the best of our knowledge, these are the first results establishing small-time global approximate controllability and global null controllability for the Cahn--Hilliard equation.}
\end{abstract}

\keywords{Global controllability, Geometric control theory, Agrachev-Sarychev method, Observability from measurable sets, Propagation of smallness, Lebeau-Robbiano approach, Source term method.} 

\medskip
\noindent
\textbf{2020 MSC:} {35K55, 93C10, 93B05, 93B07, 93C20}
\footnotesize
\tableofcontents
\normalsize

\section{Introduction}

{\subsection{Motivation}\label{sec:moti}

Phase separation is a physical phenomenon arising in multicomponent systems when
they are subjected to thermodynamic changes. A paradigmatic example is given
by a binary homogeneous mixture that is rapidly quenched below a critical
temperature. Under such conditions, the initially mixed state may become
unstable and the two components of the system start to reorganize through a
pattern-formation process, eventually leading to regions where one component is
predominant. The early stage of this dynamics is usually referred to as spinodal
decomposition: small spatial scales are amplified as a consequence of
thermodynamical instabilities, giving rise to fine microstructures. On longer
time scales, the system enters a coarsening regime, in which these
microstructures merge and larger regions emerge as the system evolves towards
equilibrium.

The Cahn--Hilliard equation is one of the fundamental models for describing this
type of phase separation in binary mixtures and other segregation-driven
processes. Originally introduced by Cahn and Hilliard \cite{CH58,Cah61}, it
provides a thermodynamically consistent description of spinodal decomposition
and phase ordering in metallic alloys. Beyond this classical setting, the
equation is also used to describe various pattern-formation phenomena, including
nucleation and growth \cite{CH59,BF93}. Over the years, its range of
applications has broadened substantially, with Cahn--Hilliard type models
appearing in contexts as diverse as copolymer morphology \cite{OK86}, image
inpainting \cite{BEG07}, tumor-growth modelling \cite{OTP10,GLKSS16},
biological applications \cite{DLCRW09,Dua16}, two-phase fluid flows
\cite{LS03,AGG12,LW18}, interacting populations \cite{CM81}, astronomical
structures \cite{Tre03}, and thin-film dynamics \cite{TK04}. This diversity of
applications has established the Cahn--Hilliard equation as a central tool for
studying segregation-like behaviour across disciplines. General overviews of its
modelling foundations and mathematical theory can be found in the surveys
\cite{Nov08} and \cite{Mir19}.

In its classical form, the Cahn--Hilliard equation describes the evolution of an
order parameter $u$, representing the relative concentration difference between
the two components of a binary mixture. In this interpretation, the physically
meaningful range is $u\in[-1,1]$. The values $u=1$ and $u=-1$ correspond to the
two pure phases, while the intermediate region $u\in(-1,1)$ describes a diffuse
interface, namely a transition layer of positive thickness where mixing between
the two components takes place.

In this work, we consider the phase-separation dynamics on the periodic domain
$\TT = \mathbb{R}^d / 2\pi\mathbb{Z}^d$ ($d\in\mathbb N^*$). This setting is
commonly adopted in the modelling of phase separation (see, e.g.,
\cite{Ell89,NST89,Mir17}), since it describes a bulk system without boundary
effects. It also allows one to use Fourier-based methods and to avoid artificial
boundary layers \cite{CE90,LS19,GWW14}.

The thermodynamical origin of the equation is reflected in its energy
formulation. The Cahn--Hilliard equation is associated with the
Ginzburg--Landau free energy
\begin{equation*}
\mathcal E(u)
=
\int_{\TT}
\left(
\frac{\alpha}{2}|\nabla u(x)|^2+F(u(x))
\right)\dx.
\end{equation*}
Here, $\alpha>0$ is related to the surface tension at the interface, and $F$ is a
double-well potential whose minima correspond to the stable states of the
mixture. The gradient term penalizes sharp spatial transitions and accounts for
the diffuse-interface character of the model, while the potential energy term
favours values of the order parameter close to the minima of $F$.

The chemical potential is defined as the first variational derivative of the
free energy, namely
\begin{equation*}
\mu
=
\frac{\delta \mathcal E}{\delta u}
=
-\alpha\Delta u+f(u),
\qquad f=F'.
\end{equation*}
The evolution is then obtained from a conservation law for the order
parameter,
\begin{equation*}
u_t+\operatorname{div} J=0,
\end{equation*}
where $J$ denotes the mass flux. Since the flux is driven by gradients of the
chemical potential, one postulates the constitutive relation
\begin{equation*}
J=-\kappa\nabla\mu,
\end{equation*}
with $\kappa>0$ denoting the mobility. Combining the last two identities yields
\begin{equation*}
u_t=\kappa\Delta\mu.
\end{equation*}
Thus, the classical Cahn--Hilliard system takes the form
\begin{equation}\label{eq:ch_intro}
\begin{cases}
u_t = \kappa \Delta \mu & \text{in } (0,T)\times\TT, \\
\mu = -\alpha \Delta u + f(u) & \text{in } (0,T)\times\TT,
\end{cases}
\end{equation}
where $T>0$ is a fixed time horizon.

The choice of the potential $F$ is closely related to the physical consistency
of the model. A thermodynamically relevant choice is the singular (or
logarithmic) Flory--Huggins potential
\begin{equation*}
F_{\log}(s)
=
\frac{\theta}{2}
\Big((1+s)\log(1+s)+(1-s)\log(1-s)\Big)
-\frac{\theta_0}{2}s^2,
\qquad s\in(-1,1),
\end{equation*}
where $0<\theta<\theta_0$. Here, $\theta$ represents the absolute temperature of
the system, while $\theta_0$ is the critical temperature below which phase
separation occurs. The singular character of $F_{\log}$ at $s=\pm1$ prevents the
order parameter from leaving the physically meaningful interval $[-1,1]$ and
therefore provides an accurate description of the mixture.

From the mathematical and numerical viewpoints, however, regular double-well
potentials are also widely used. A typical choice is
\begin{equation}\label{nl_intro}
F(s)= \tau (s^2-1)^2,
\qquad\text{so that}\qquad
f(s)=4\tau(s^3-s),
\end{equation}
where $\tau>0$ controls the strength of the phase-separation tendency. This
polynomial potential retains the double-well structure and the preference for
the two pure phases $s=\pm1$, while avoiding the singularities of the
logarithmic model. Although it does not strictly enforce the constraint
$u\in[-1,1]$, it provides a flexible and robust framework for analysis and
numerical approximation, and it is commonly regarded as an important regular
approximation of the singular case; see, e.g.,
\cite[Section 1 and Remark 2.2]{Mir17}.

The thermodynamical interpretation of the model is reflected in two basic
structural properties. First, the total mass is conserved:
\begin{equation*}
\frac{\d}{\dt}\int_{\TT} u(t,x)\dx = 0.
\end{equation*}
Second, along smooth solutions, the free energy is dissipated according to
\begin{equation*}
\frac{\d}{\dt}\mathcal E(u(t))
=
-\kappa\int_{\TT}|\nabla\mu(t,x)|^2\dx
\leq 0.
\end{equation*}
These two identities encode the conservative and dissipative character of the
phase-separation dynamics: while the total amount of material is preserved, the
system tends to reduce its free energy by forming interfaces and, at later
stages, by coarsening the resulting patterns.

The question addressed in this work is whether this phase-separation dynamics
can be influenced by suitable external actions in order to guide the system towards a prescribed final configuration. This places the problem in the framework
of PDE control theory, where one acts on the trajectories of a system, typically
through a forcing term, a boundary input, or another external mechanism, in order
to reach a prescribed goal. Classical references on this topic include the works \cite{MZ05,coron,TW09}.

To formulate this question, we consider a controlled version of the
Cahn--Hilliard system \eqref{eq:ch_intro}
\begin{equation}\label{ch_control}
\begin{cases}
u_t = \kappa \Delta \mu + {\eta(t,x)}
& \text{ in } (0,T)\times\TT,\\
\mu = -\alpha \Delta u + f(u)
& \text{ in } (0,T)\times\TT.
\end{cases}
\end{equation}
Here, $\eta\in\mathcal U_{\mathrm{ad}}$, where
$\mathcal U_{\mathrm{ad}}$ denotes a set of admissible controls
to be specified in each controllability result.
The controlled evolution is supplemented with the initial condition
\begin{equation}\label{u_0}
u(0,x)=u_0(x), \qquad x\in\TT.
\end{equation}
Of particular interest are spatially localized controls of the form $\eta(t,x)=\mathbf1_\omega(x)h(t,x)$, where $\omega\subseteq\TT$
and $\mathbf1_\omega$ denotes its indicator function. Such a control
may be interpreted as a localized source or extraction mechanism
acting on the composition variable. In
particular, in contrast with the free Cahn--Hilliard dynamics, the total mass is
not necessarily conserved once the control is introduced. This is consistent
with the fact that the purpose of the control is to modify the evolution
starting from $u_0$ in order to reach prescribed final states.

Within this framework, the two control notions considered in this work are
\emph{approximate controllability} and \emph{null controllability}. Approximate
controllability asks whether, for any prescribed target profile
$u_{\mathrm{tar}}$ and any $\epsilon>0$, one can choose a control {$\eta\in\mathcal U_{\mathrm{ad}}$} so that
the solution of \eqref{ch_control} satisfies
\begin{equation*}
\|u(T,\cdot)-u_{\mathrm{tar}}\|_{X}<\epsilon,
\end{equation*}
for a suitable functional space $X$. In other words, the question is whether the
evolution starting from $u_0$ can be steered arbitrarily close to a desired
configuration at the final time. In phase-separation models, this may be
interpreted as the possibility of influencing the system so that the
configuration at time $T$ comes close to a prescribed distribution of the two
components.

Null controllability concerns the question of steering the system exactly to the
zero state at the final time, that is, whether there exists a control {$\eta\in\mathcal U_{\mathrm{ad}}$} such
that
\begin{equation*}
u(T,x)\equiv 0 \quad \text{in } \TT.
\end{equation*}
In phase-separation models, the state $u\equiv0$ corresponds to a spatially
homogeneous configuration in which neither component is predominant.

{The global controllability of fourth-order systems, and in particular of the Cahn--Hilliard equation, remains largely open and challenging. The results obtained in this paper concern two distinct notions: small-time
global approximate controllability and global null controllability. To
the best of our knowledge, neither of these properties has previously
been established for the Cahn--Hilliard equation. Existing
contributions, such as \cite{DR98}, \cite{Guz20}, \cite{Kas20}, and
\cite{LY25}, typically rely on restrictive assumptions concerning the
spatial dimension, the structure of the nonlinearity, or the choice of
boundary conditions.} Among these works, only \cite{Guz20} addresses the Cahn--Hilliard
equation \eqref{ch_control} in one spatial dimension with the physically
relevant nonlinearity $f(u)=u^3-u$, whereas the remaining papers study control
properties of fourth-order parabolic equations with different and more
tractable nonlinearities (globally Lipschitz or power-type with good sign). We
also note that the result in \cite{Guz20} is local in the sense that it requires
small initial data, while the other works achieve global controllability under
specific assumptions on the nonlinear terms. Altogether, these limitations
highlight the intrinsic difficulties of controlling nonlinear fourth-order
dynamics globally, particularly in higher dimensions.

\subsection{Global approximate controllability}

Building on the discussion above, we now introduce a more precise formulation of
the approximate controllability problem. To simplify the exposition, we rewrite
the controlled Cahn--Hilliard equation as
\begin{equation}\label{eq_main}
\begin{cases}
u_t + \Delta^2 u + \Delta u = \Delta(u^3) + \eta & \text{in } (0,T)\times \TT,\\[1mm]
u(0) = u_0, &\text{in } \TT.
\end{cases}
\end{equation}
This formulation corresponds to choosing $\kappa=\alpha=1$ and $\tau=1/4$ in
\eqref{eq:ch_intro}--\eqref{nl_intro}. A similar analysis can be
performed for general values of the positive constants $\kappa$, $\alpha$ and
$\tau$, although the corresponding modifications remain to be detailed.

In \eqref{eq_main}, $\eta=\eta(t,x)$ corresponds to a control applied on the
whole domain $\TT$. With this function, our aim is to steer the solution from
any given initial data $u_0$ arbitrarily close to a target function $u_1$ in a
suitable norm. In more detail, for any nonnegative integer $k$, let us denote by
$H^k(\TT)$ the usual Sobolev space defined on the $d$-dimensional torus (see
\Cref{functional}) and consider the canonical basis of $\mathbb R^d$ given by
\begin{equation*}
	\mathcal K=\left\{\left(1,0,\ldots,0\right), \left(0,1,\ldots,0\right),\ldots,\left(0,\ldots,0,1\right)\right\}\subset \mathbb{Z}^d.
\end{equation*}
Let us denote the linear space
\begin{equation}\label{hnot}
	\mathscr{H}_0=\spn\{1, \sin (x\cdot k), \cos(x\cdot k) \}_{k\in \mathcal K}.
\end{equation}

For approximate controllability, we take
$\mathcal U_{\mathrm{ad}}=L^\infty(0,T;\mathscr H_0)$
and introduce the following definition.
\begin{defi}
The Cahn--Hilliard equation \eqref{eq_main} is said to be \emph{globally
$H^k$-approximately controllable} at time $T>0$ by an $\mathscr H_0$-valued
control if the following holds: for any $u_0,u_1\in H^k(\TT)$ and any
$\epsilon>0$, there exists a control
$
\eta\in L^\infty(0,T;\mathscr H_0)
$
such that the corresponding solution $u$ to \eqref{eq_main} satisfies
\begin{equation*}
\|u(T)-u_1\|_{H^k(\TT)} < \epsilon.
\end{equation*}
If this property holds for every time horizon $T>0$, we say that the
Cahn--Hilliard equation \eqref{eq_main} is \emph{small-time globally
$H^k$-approximately controllable}.
\end{defi}

Our first main result reads as follows.

\begin{theo}\label{small_time}
Let $k\in\mathbb N^*$ be such that $k>d/2$. Then, the Cahn--Hilliard equation
\eqref{eq_main} is small-time globally $H^k$-approximately controllable.
\end{theo}

{
Theorem \ref{small_time} establishes small-time global approximate
controllability for \eqref{eq_main}. Its proof uses the cubic
nonlinear term to generate new Fourier directions. More precisely, an
asymptotic property produces directions determined by the
nonlinearity, while a saturation argument yields successively larger
finite-dimensional spaces of such directions whose union is dense in
$H^k(\TT)$. This allows us to approximately reach any target in
$H^k(\TT)$. Throughout this construction, the control takes values in
the fixed space $\mathscr H_0$, which has dimension $2d+1$,
independently of the system parameters and of the initial and target
states. Moreover, we establish global approximate controllability
without any restriction on the control time $T>0$; indeed, $T$ can be
taken arbitrarily small.
}

The methodology used to prove Theorem \ref{small_time} is inspired by the
geometric control approach developed by Agrachev and Sarychev in
\cite{AS05,AS06}. Their strategy, originally introduced for the Navier--Stokes
and Euler equations on the torus, achieves approximate controllability by means
of finite-dimensional controls acting on the whole domain and relies on the use
of large controls over small time intervals.

This approach can be traced back to the finite-dimensional framework of
\cite{JK85}, where related geometric ideas based on short-time control
mechanisms were developed. The works \cite{AS05,AS06} may thus be viewed as
infinite-dimensional extensions of this control philosophy; see also
\cite{GHHM18} for another extension of these ideas to infinite-dimensional
systems.

Since then, the Agrachev--Sarychev strategy has been successfully adapted to a
broad class of nonlinear evolution equations, including the three-dimensional
Navier--Stokes system \cite{S06}, compressible and incompressible Euler
equations \cite{N11,N10}, the viscous Burgers equation \cite{S14,SA18}, and the
Boussinesq equations \cite{NR25b}.

There has also been substantial progress for nonlinear scalar equations on
periodic domains. In particular, the approach in \cite{Ner21} plays an important
role in our arguments: the overall structure of our proof is inspired by their
strategy, suitably adapted to the fourth-order setting of the Cahn--Hilliard
dynamics. Further advances motivated by \cite{Ner21} include controllability
results for the Benjamin--Bona--Mahony equation \cite{Jel23}, the
Kadomtsev--Petviashvili--I equation \cite{jellouli22}, the
Kuramoto--Sivashinsky equation \cite{G22}, the Korteweg--de~Vries equation
\cite{C23}, the Kawahara equation \cite{AM25}, the
Zakharov--Kuznetsov--Burgers equation \cite{C25}, and the Camassa--Holm
equation \cite{CDM25}. Furthermore, regarding small-time global controllability
in the framework of bilinear control, we refer to \cite{DN25} for the
Schrödinger equation, \cite{Pozzoli_2024} for the wave equation,
\cite{DPU25,majumdar2025} for semilinear parabolic equations, and \cite{DT25} for the
Burgers equation.
\subsection{Global null controllability}
Building on the approximate controllability result presented in the previous
section, we now turn to an application: global null controllability. The
argument relies essentially on this property, combined with a local null
controllability result that allows us to reach the null state exactly.

Let $d\in\{1,2,3\}$. Our second main result can be stated as follows.
\begin{theo}\label{global_null}
Let $T>0$ and $u_0\in L^2(\mathbb{T}^d)$. Then there exists a control
$\eta\in L^{\infty}(0,T;L^\infty(\mathbb{T}^d))$ such that equation
\eqref{eq_main} satisfies $u(T,\cdot)=0$ in $\mathbb{T}^d$. Moreover,
the control is of the form
\begin{equation}\label{exp_con}
\eta(t,\cdot)=
\begin{cases}
0, & t\in(0,\epsilon),\\
\eta_1(t,\cdot), & t\in(\epsilon,\delta),\\
\mathbf{1}_{\omega}\eta_2(t,\cdot), & t\in(\delta,T),
\end{cases}
\end{equation}
where $\epsilon$ and $\delta$ are arbitrary positive constants satisfying
$0<\epsilon<\delta<T$, $\eta_1(t,\cdot)\in\mathscr H_0$, and
$\eta_2\in L^\infty(\delta,T;L^\infty(\omega))$, where
$\omega\subset\mathbb{T}^d$ is any measurable set with positive Lebesgue
measure.
\end{theo}
{
Theorem \ref{global_null} establishes global null controllability for
arbitrary initial data, without any restriction on the control time
$T>0$. This property is obtained by combining the global approximate
controllability result stated in the previous section with a local null
controllability result. These two parts rely on different mechanisms.
The global approximate step uses the cubic nonlinearity to generate
Fourier directions, whereas the local step exploits the dissipative
fourth-order dynamics, together with quantitative propagation of
smallness estimates, to drive a sufficiently small state exactly to
zero by means of a spatially localized control. The overall strategy is
organized into three stages: an initial waiting time, a
finite-dimensional control acting on the whole torus, and finally a
localized control that drives the solution exactly to zero. A detailed
description of this construction is provided in \Cref{strategies}.
}

To address the null controllability strategy for the Cahn--Hilliard equation, we
rely on the extensive body of work developed over the last two decades on
spectral and observability estimates for parabolic equations on measurable sets
with positive Lebesgue measure, see, for instance,
\cite{PW13,AEWZ14,EMZ15,WZ17}. These results show that the classical
geometric requirement that the control region be open can be significantly
weakened, thanks to recent advances in spectral geometry and quantitative unique
continuation, see \cite{LM18,Lg18a,Lg18b}. More precisely, for linear
parabolic-type equations (such as the heat or Stokes equations \cite{CSZ20}), a
comprehensive theory is now available: observability holds from sets of the form
$E\times\omega$, where $E\subset(0,T)$ and $\omega\subset\Omega$ are
measurable sets of positive Lebesgue measure.

More recently, observability for the heat equation has been established even
from control sets of zero measure, provided they have positive fractional
Hausdorff content, see \cite{BM23,GLMO25}. This yields the sharpest geometric
conditions currently known. In line with these developments, we extend our
global null controllability results to controls supported on certain sets of
zero measure. For clarity, recall that the $s$-Hausdorff content of a set
$E\subset\mathbb{R}^d$ is
\begin{equation*}
\mathcal{C}^s_{\mathcal H}(E)
=
\inf\left\{
\sum_j r_j^{\,s}:E\subset\bigcup_j B(x_j,r_j)
\right\},
\end{equation*}
and its Hausdorff dimension is defined as
\begin{equation*}
\dim_{\mathcal H}(E)
=
\inf\left\{
s:\mathcal{C}^s_{\mathcal H}(E)=0
\right\}.
\end{equation*}
We may now state the following global null controllability result:
\begin{theo}\label{global_nulld2}
{Let $d\in\{1,2\}$. There exists $s=s(d)\in(0,1)$ such that, for every
$T>0$, $0<\epsilon<\delta<T$, $m>0$, $u_0\in L^2(\mathbb T^d)$, and every
closed set $\omega\subset\mathbb T^d$ satisfying
$\mathcal C_{\mathcal H}^{d-s}(\omega)>m$,} there exists a control
$\eta\in L^\infty(0,T;H^{-2}(\mathbb T^d))$ such that the solution of
\eqref{eq_main} satisfies $u(T,\cdot)=0$ in $\mathbb T^d$. The control
is of the form
\begin{equation*}
\eta(t,\cdot)=
\begin{cases}
0, & t\in(0,\epsilon),\\
\eta_1(t,\cdot), & t\in(\epsilon,\delta),\\
\eta_2(t,\cdot), & t\in(\delta,T),
\end{cases}
\end{equation*}
where $\eta_1\in L^\infty(\epsilon,\delta;\mathscr H_0)$ and
$\eta_2\in L^\infty(\delta,T;H^{-2}(\mathbb T^d))$ is supported in
$(\delta,T)\times\omega$.
\end{theo}

{
\begin{rmk}
The controls in Theorems~\ref{global_null}
and~\ref{global_nulld2} are obtained by concatenating controls acting
on disjoint time intervals. In Theorem~\ref{global_null}, this
concatenation is given by \eqref{exp_con}; the finite-dimensional
Fourier control belongs to $L^\infty(\mathbb T^d)$, and the localized
control belongs to the same space after extension by zero outside
$\omega$. Hence, the resulting control belongs to
$L^\infty(0,T;L^\infty(\mathbb T^d))$. In
Theorem~\ref{global_nulld2}, both the finite-dimensional Fourier
control and the localized control belong to $H^{-2}(\mathbb T^d)$.
Consequently, their concatenation naturally belongs to
$L^\infty(0,T;H^{-2}(\mathbb T^d))$. Thus, in each case, the global
control belongs to a common space containing the different classes of
controls used in the construction.
\end{rmk}
}

\subsection{Strategy of the proof of main results}\label{strategies}

We begin this section by presenting the overall strategy used to establish our
two main controllability results.
\begin{figure}[h!]
	\centering
	\includegraphics[width=0.95\textwidth, trim=0cm 9cm 0cm 0cm, clip]{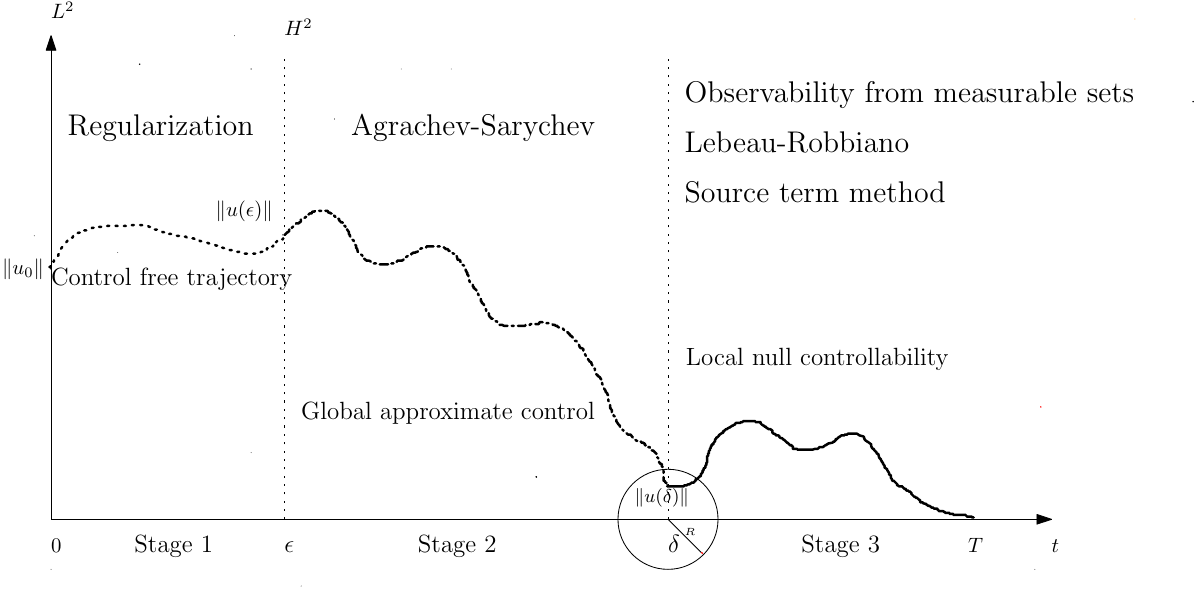}
	\caption{Global null controllability in small time}
	\label{fig:mylabel}
\end{figure}

\begin{itemize}
	\item The proof of Theorem \ref{small_time} relies on three key ingredients:
	stability estimates (\Cref{Lipschitz type}), an asymptotic property (\Cref{prop_asympt}), and a saturation property (\Cref{thm_sat}). Using these properties, we establish small-time global approximate controllability of \eqref{eq_main}. The proof technique is inspired by the approach in \cite{Ner21}. A brief idea is that, thanks to the asymptotic property, the system can be driven over sufficiently small time intervals from $u_0$ to an arbitrarily small neighborhood of any point in $u_0 + \mathscr H_0$. By iterating this procedure together with using the saturation property, we reduce the approximate controllability problem to steering the system toward targets lying in a finite-dimensional affine subspace translated by the initial data. The key step is the use of large controls over short intervals, exploiting a specific asymptotic property of the nonlinear dynamics. Finally, continuity of the solution in time, together with stability estimates (see \Cref{Lipschitz type}), allows us to conclude global approximate controllability at any prescribed time.

	\item To prove global null controllability, we use a strategy that combines global approximate controllability with local null controllability; see \Cref{fig:mylabel}. More precisely, we first prove small-time global approximate null controllability by which we steer the trajectory arbitrarily close to the origin, and then we apply the local null controllability result to reach the origin exactly.

	\item The main idea behind the proof of local null controllability is to reduce the nonlinear problem to the linearized one and to establish the null controllability of this linearized system. Our controllability result applies when the control region is a measurable set of positive Lebesgue measure, or even for certain control sets of zero measure. The overall structure of the argument is inspired by the works \cite{BM23} and \cite{GLMO25} which combine a propagation of smallness for free solutions of the linearized system and a Lebeau--Robbiano type strategy \cite{LR}. Furthermore, we establish an appropriate control cost estimate which enables the use of the source term method of \cite{Tucsnak_nonlinear}. In the present $L^\infty$ control setting, however, this is not a direct application of that method, since the linear result does not provide the bounded linear control operator available in the classical $L^2$ framework. We instead adapt the multivalued iterative construction recently used in \cite{HSLBP26}, which is inspired by Nadler's multivalued contraction principle \cite{Nad69}. This allows us to deduce the local null controllability of the original nonlinear equation.
\end{itemize}

\subsection{Outline}

The rest of the paper is organized as follows. \Cref{sec_prem} presents preliminary well-posedness results for the system, together with stability estimates. \Cref{sec_tech} contains the main technical ingredients, including an asymptotic property (\Cref{prop_asympt}) and a saturation property (\Cref{thm_sat}), which form the foundation of our main result. In \Cref{sec_global}, we prove the main global approximate controllability result Theorem \ref{small_time}. \Cref{sec_globalnull} is devoted to the proof of the global null controllability results Theorem \ref{global_null} and Theorem \ref{global_nulld2}. Finally, we provide all the proofs of well-posedness issues and some auxiliary results in Appendices \ref{app} and \ref{estp3}.

\section{Preliminaries}\label{sec_prem}
This section is devoted to preliminaries concerning the functional setting and several key well-posedness results.
\subsection{Functional framework}\label{functional}
Consider $L^2(\mathbb{T}^d):=L^2(\mathbb{T}^d;\mathbb{R})$ as the collection of real-valued
square-integrable functions on the torus. For $u\in L^2(\TT)$, its Fourier coefficients are given by
\begin{equation*}
u_{k} = \frac{1}{(2\pi)^d} \int_{\mathbb{T}^d} u(x)\, e^{-ix\cdot k}\,dx,
\qquad k \in \mathbb{Z}^d,
\end{equation*}
and we write the corresponding Fourier series expansion
\begin{equation*}
u(x) = \sum_{k \in \mathbb{Z}^d} u_{k}\, e^{ix\cdot k},
\end{equation*}
where the series converges to $u$ in the $L^2$ sense. For real-valued functions we also have
$u_{-k}=\overline{u_k}$ for all $k\in\mathbb{Z}^d$.

For $p\in[1,\infty]$, we denote by $L^p(\TT)$ the usual Lebesgue spaces on the
$d$-dimensional torus, endowed with the norms
\begin{equation*}
\|u\|_{L^p(\TT)}
:=\left(\int_{\TT}|u(x)|^p\,dx\right)^{1/p}, 
\;\;  1\le p<\infty, \qquad \text{and}\qquad \|u\|_{L^\infty(\TT)}
:=\esssup_{x\in\TT}|u(x)|.
\end{equation*}

For $s \in \mathbb{R}$, we define the spaces
\begin{equation*}
H^s(\mathbb{T}^d)
:=
\left\{
u = \sum_{k\in \mathbb{Z}^d} 
u_{k} \, e^{ix \cdot k} \in \mathcal D'(\mathbb{T}^d)
\; : \;
\sum_{k\in \mathbb{Z}^d} 
|u_{k}|^2 \, (1 + |k|^2)^s < \infty
\right\}.
\end{equation*}
The induced norm is $\|u\|_{H^s(\mathbb{T}^d)}^2 = \sum_{k\in \mathbb{Z}^d} |u_{k}|^2 \,(1 + |k|^2)^s$. 
\begin{rmk}
For brevity, and whenever the underlying domain is clear from the context, we
omit the explicit indication of the domain in the norms.  
For instance, we write $\|u\|_{L^p}$ or $\|u\|_{H^s}$ instead of
$\|u\|_{L^p(\TT)}$ or $\|u\|_{H^s(\TT)}$.  
Whenever the domain plays a role or ambiguity may arise, we will instead specify it explicitly.
\end{rmk}
The homogeneous operator $(-\Delta)^{s/2}$ acts as
\begin{equation*}
(-\Delta)^{s/2} u 
= \sum_{k\in \mathbb{Z}^d} |k|^{s} \, u_{k} \, e^{ix \cdot k},
\end{equation*}
with norm
\begin{equation*}
\|(-\Delta)^{s/2} u\|_{L^2}^2 
= \sum_{k\in \mathbb{Z}^d} |k|^{2s} \, |u_{k}|^2.
\end{equation*}
For $s>0$, the $H^s(\mathbb{T}^d)$-norm defined above satisfies the equivalence
\begin{equation*}
\|u\|_{H^s} 
\simeq 
\|u\|_{L^2} 
+ \|(-\Delta)^{s/2} u\|_{L^2}.
\end{equation*}
Let $\mathcal A$ be the operator defined as follows:
\begin{equation}\label{abstract}
	\begin{cases}
\mathcal A u =-\Delta^2 u  -\Delta u, \quad \forall u \in \mathcal D(\mathcal A),\\
\mathcal D(\mathcal A) = \left\{ u \in L^2(\TT) : \mathcal A u \in L^2(\TT) \right\}.
\end{cases}
\end{equation}
It is easy to see that, $\mathcal A$ generates a strongly continuous semigroup 
$\{\mathcal S(t) \}_{t \ge 0}$ on the space $L^2(\TT)$.
For any $s\geq 0$ and $T > 0$, we define the space
\begin{equation}\label{XT}
X_T^s = C([0, T]; H^s(\TT)) \cap L^2(0, T; H^{s+2}(\TT))
\end{equation}
endowed with its natural Banach norm
\begin{equation*}
\|u\|_{X_T^s} = \|u\|_{C([0,T];H^s(\TT))} + \|u\|_{L^2(0,T;H^{s+2}(\TT))}.
\end{equation*}
\subsection{Well-posedness results}
Let us consider the following extended system 
\begin{equation}\label{eq_extension}
	\begin{cases}
		u_t +\Delta^2 (u + \varphi) +\Delta(u+\varphi)=\Delta (u + \varphi)^3+\eta, & t>0, \,\, \, x \in \TT, \\
		u(0) = u_0, & x \in \TT,
	\end{cases}
\end{equation}
where $\varphi:=\varphi(x)$ is some suitable function defined on $\TT.$ We have the following local well-posedness results for the Cahn-Hilliard equation \eqref{eq_extension}.
\begin{proposition}[Local well-posedness in time]\label{wellposed2}
	Let $k\in\mathbb N$ satisfy $k>d/2$. For any
	$u_0 \in H^{k}(\mathbb{T}^d)$, $\varphi\in H^{k+2}(\TT)$ and 
	$\eta \in L^2_{\mathrm{loc}}(0,\infty;H^{k-2}(\TT))$, 
	there exists a maximal time $\mathcal T = \mathcal T(u_0,\varphi, \eta) > 0$ and a unique mild solution 
	$u$ of \eqref{eq_extension}. Namely, for any $T < \mathcal T(u_0,\varphi, \eta)$, $u \in X_T^k$.
	Moreover, if $\mathcal T(u_0,\varphi, \eta) < +\infty$, then 
	\begin{align*}
	\|u(t)\|_{H^{k}} \to +\infty \quad \text{as } t \to \mathcal T(u_0,\varphi,\eta)^{-}.
	\end{align*}
\end{proposition}
\begin{proposition}\label{max}
Let $u_0, \varphi, \eta,$ $\mathcal{T}(u_0, \varphi, \eta)$ and $u$ be as in \Cref{wellposed2}. Assume $\hat u_0\in H^k(\TT)$ and $T<\mathcal T(u_0,\varphi,\eta)$. Let us set $\Lambda=\norm{u}_{X_T^k}+ \|\eta\|_{L^2(0,T; H^{k-2}(\TT))}+\norm{\varphi}_{H^{k+2}}$. Then there exists $\delta>0$ depending on $T$ and $\Lambda$ such that if
 $\norm{u_0-\hat u_0}_{H^k}<\delta,$  equation \eqref{eq_extension} with initial condition $\hat u(0)=\hat u_0$ admits a unique mild solution $\hat u\in X_T^k$ satisfying
 \begin{equation}\label{eq:local_uniform_bound}
 \|u-\hat u\|_{C([0,T];H^k)}<1.
 \end{equation}
\end{proposition}

\noindent
\textbf{Notation} Let $\mathcal{R}$ be the mapping that takes a triple $(u_0,\varphi, \eta)$ to the solution $u$ of \eqref{eq_extension} and $\mathcal{R}_t(u_0,\varphi, \eta)$ is the restriction of $\mathcal{R}(u_0,\varphi, \eta)$ at time $t.$ That is, $\mathcal{R}_t$ takes $(u_0,\varphi, \eta)$ to $u(t),$ where $u(t,x)$ is the solution of \eqref{eq_extension}.

\medskip

Throughout the paper, $C>0$ denotes a generic positive constant which may change from line to line. The dependence of $C$ on parameters will be stated explicitly whenever relevant.
\begin{proposition}[Stability with respect to initial data]\label{Lipschitz type} 
Let $k\in\mathbb N$ satisfy $k>d/2$, $u_0\in H^k(\TT)$, and
$\eta\in L^2_{\mathrm{loc}}(0,\infty;H^{k-2}(\TT))$. Let
$u=\mathcal R(u_0,0,\eta)$ and $T<\mathcal T(u_0,0,\eta)$, and set
\begin{equation*}
\Lambda=\|u\|_{X_T^k}+\|\eta\|_{L^2(0,T;H^{k-2}(\TT))}.
\end{equation*}
Then, there exist $\delta>0$ and a constant $C=C(T,\Lambda)>0$ such that, for every $\hat u_0\in H^k(\TT)$ satisfying
$\|u_0-\hat u_0\|_{H^k}<\delta$,
	\begin{equation}\label{stab_est}
		\sup_{0 \le t \le T}
		\|\mathcal{R}_t(u_0, 0, \eta) - \mathcal{R}_t(\hat{u}_0, 0, {\eta })\|_{H^{k}}
		\;\le\;
		C\,
		\|u_0 - \hat u_0\|_{H^{k}}. 
	\end{equation}
\end{proposition}

\begin{lemma}[Parabolic regularization]\label{lem:parabolic_regularization}
Let $k\in\mathbb N$ satisfy $k>d/2$, let $u_0\in H^k(\TT)$, and let
$\eta\in L^2(0,T;H^k(\TT))$. Assume that
$u=\mathcal R(u_0,0,\eta)$ is well-defined on $[0,T]$. Then, for every
$\tau\in(0,T)$,
\begin{equation}\label{eq:parabolic_regularization}
u\in C([\tau,T];H^{k+2}(\TT))
\cap L^2(\tau,T;H^{k+4}(\TT)).
\end{equation}
In particular, $u(t)\in H^{k+2}(\TT)$ for every $t\in(0,T]$.
\end{lemma}

The proofs of the above results are presented in Appendix \ref{pr_of_1}, Appendix \ref{pr_of_2}, Appendix \ref{pr_of_3}, and Appendix \ref{pr_of_4}, respectively.

\begin{rmk}
Although the results above are stated for $k\in\mathbb N$, they remain
valid for any real $k>d/2$. This follows from the Fourier
characterization of Sobolev spaces and the algebra property of
$H^k(\TT)$. For simplicity and consistency throughout the paper, we
consider only integer values of $k$.
\end{rmk}


\section{Flow, asymptotic and saturation properties}\label{sec_tech}
This section presents several technical results that play a central role in proving our main theorems. {As usual, and for notational convenience, if $f=f(t,x)$ is a time-dependent function with values in a spatial function space $X$, we write $f(t)$ for the spatial function $x\mapsto f(t,x)\in X$.}

We begin with a standard result concerning the properties of the nonlinear semigroup.

\begin{lemma}Assume $u_0 \in H^{k}(\mathbb{T}^d)$, $\varphi\in H^{k+2}(\TT)$ and 
	$\eta \in L^2_{\mathrm{loc}}(0,\infty;H^{k-2}(\TT))$. The following properties hold:
	\begin{enumerate}
		\item Let $\delta \in (0, \mathcal T(u_0,\varphi,\eta))$. The flow property holds:  
		\begin{equation}\label{flow_property}
			\mathcal R_{\delta}(u_0,\varphi,\eta)=\mathcal R_{\delta}(u_0+\varphi,0,\eta)-\varphi.
		\end{equation}
		\item Let us define
		\begin{equation*}
			\eta(s) =
			\begin{cases}
				\eta_1(s) & s \in [0, t_1),\\
				\eta_2(s) & s \in [t_1, t_1 + t_2),\\
				\eta_3(s) & s \in [t_1 + t_2, t_1 + t_2 + t_3),
			\end{cases}
		\end{equation*}
		for some $t_1, t_2, t_3 > 0$ satisfying $t_1 + t_2 + t_3< \mathcal T(u_0, \eta)$. Then the solution $\mathcal R(u_0, 0, \eta)$ of \eqref{eq_extension} satisfies the concatenation property
		\begin{equation}\label{flow2}
			\mathcal R_{t_1 + t_2 + t_3}(u_0, 0, \eta)
			= \mathcal R_{t_3}\!\left(
			\mathcal R_{t_2}\!\left(
			\mathcal R_{t_1}(u_0, 0, \eta_1(\cdot)),
			0, \eta_2(\cdot + t_1)
			\right),
			0, \eta_3(\cdot + t_1 + t_2)
			\right).
		\end{equation}
	\end{enumerate}
\end{lemma}
\begin{proof}

		Let us denote $v=u+\varphi,$ where $u$ is the solution of \eqref{eq_extension}. It can be checked that $v$ satisfies the following equation
		\begin{equation*}
			\begin{cases}
				v_t +\Delta^2 v +\Delta v=\Delta (v^3)+\eta, \\
				v(0) = u_0+\varphi.
			\end{cases}
		\end{equation*}
		By uniqueness of the solution, we have $v=u+\varphi.$ This implies $\mathcal R_{\delta}(u_0+\varphi,0,\eta)=\mathcal R_{\delta}(u_0,\varphi, \eta)+\varphi.$

		Let $0 < s \le t$ and let $ \widetilde{ \mathcal R}(t,s,v,\eta)$ denote the solution of \eqref{eq_extension} at time $t$ with 
		$\varphi=0$ and initial condition
		$	\widetilde{\mathcal R}(s,s,v,\eta)=v.$
		In particular,
		$ \mathcal R_t(u_0,0,\eta)=\widetilde{\mathcal R}(t,0,u_0,\eta).
		$ By the uniqueness of solution, for any $\sigma\in [s,t]$ we have
		\begin{equation}\label{eq:flow_property}
			\widetilde{\mathcal R}(t,\sigma,\widetilde{\mathcal R}(\sigma,s,u_0,\eta),\eta)
			= \widetilde{\mathcal R}(t,s,u_0,\eta).
		\end{equation}
		Using \eqref{eq:flow_property}, we obtain
		\begin{equation*}
			\begin{aligned}
				&\widetilde{\mathcal R}(t_1+t_2+t_3,0,u_0,\eta)\\
				&= \widetilde{\mathcal R}\!\left(t_1+t_2+t_3,\, t_1+t_2,\,
				\widetilde{\mathcal R}\left(t_1+t_2,t_1,\widetilde{\mathcal R}\left(t_1,0,u_0,\eta_1(\cdot)\right),\, \eta_2(\cdot)\right),\,
				\eta_3(\cdot)\right)
				\\
				&= \widetilde{\mathcal R}\!\left(t_3,\,0,\,
				\widetilde{\mathcal R}\left(t_1+t_2,t_1,\widetilde{\mathcal R}\left(t_1,0,u_0,\eta_1(\cdot)\right),\, \eta_2(\cdot )\right),\,
				\eta_3(\cdot + t_1 +t_2)\right)
				\\
				&= \mathcal R_{t_3}\!\left(
				\widetilde{\mathcal R}\left(t_1+t_2,t_1,\widetilde{\mathcal R}\left(t_1,0,u_0,\eta_1(\cdot)\right), \eta_2(\cdot)\right),
				0,\,
				\eta_3(\cdot +t_1 + t_2)
				\right)
				\\
				&= \mathcal R_{t_3}\!\left(
				\mathcal R_{t_2}\left( \mathcal R_{t_1}\left(u_0,0,\eta_1(\cdot)\right),\, 0,\, \eta_2(\cdot +t_1\right)),
				0,\,
				\eta_3(\cdot + t_1 +t_2)
				\right).
			\end{aligned}
		\end{equation*}
	This completes the proof.
\end{proof}
\subsection{Asymptotic property}
This section is devoted to the study of the behavior of solutions to the perturbed equation \eqref{eq_extension} when  $\eta$ is independent of time. The following result guarantees 
the existence of solution on some finite time interval $[0, \delta]$ for some $\delta > 0$, when 
$\varphi$ and $\eta$ are sufficiently large. Let  us consider the system
\begin{equation}\label{eq_extension_delta}
	\begin{cases}
		u_t +\Delta^2 \left(u + \delta^{-\frac{1}{3}}\varphi\right) +\Delta\left(u+\delta^{-\frac{1}{3}}\varphi\right)=\Delta \left(u + \delta^{-\frac{1}{3}}\varphi\right)^3+{{\delta}^{-1}}\eta, & t>0, \,\, \, x \in \TT,\\
		u(0) = u_0, &  x \in \TT. 
	\end{cases}
\end{equation}
We now state an asymptotic property of the solution to \eqref{eq_extension_delta}, which is one of the key ingredient in proving the main global approximate control result Theorem \ref{small_time}.
\begin{proposition}\label{prop_asympt}
Let $k\in \mathbb{N}^*$ be such that  $k>{d}/{2}.$ Assume that $u_0\in H^{k+2}(\TT)$, $\varphi \in H^{3k+2}(\TT)$, and $\eta \in H^{k+2}(\TT)$.
	Then there exists $\delta_0>0$ such that for all $\delta \in (0,\delta_0)$, the solution $u(t)=\mathcal R_t(u_0, \delta^{-\frac{1}{3}}\varphi, {\delta}^{-1}\eta)$ of \eqref{eq_extension_delta} is well-defined in the time interval $[0,\delta].$ Moreover it satisfies the following property:
	\begin{equation}\label{asymp}
		\lim_{\delta\to 0^+}u(\delta)=
		u_0 + \eta+\Delta(\varphi^3)
		\quad \text{in } H^k(\TT).
	\end{equation}
	\end{proposition}

\begin{rmk}
	This asymptotic property \eqref{asymp}  allows us to steer the controlled trajectory of \eqref{eq_extension_delta} in small time to any neighborhood of target $u_1$ belonging to the affine space $u_0 + \mathscr{H}_1$, where $\mathscr H_1$ is the vector space generated by elements of the form \begin{equation}\label{asymp_h}\eta + \sum_{k=1}^{m} \Delta(\varphi_k^3),\end{equation}
	 for some integer $m \ge 1$ and vectors $\eta, \varphi_1, \ldots, \varphi_m \in \mathscr H_0$ (see \Cref{sec_sat} for the detailed expression of $\mathscr H_1$). By iterating this argument, we show that starting from $u_0$ we can also 
	approximately reach any point belonging to $u_0 + \mathscr H_2$ in small time, where the space $\mathscr H_2$ is defined as \eqref{asymp_h}, but now with vectors $\eta, \varphi_1, \ldots, \varphi_m \in \mathscr H_1$.
	In this way, we construct a non-decreasing sequence of subspaces $\{\mathscr H_j\}_{j\ge1}$ such that every point in $u_0 + \mathscr H_j$ is approximately reachable from $u_0$ by means of a control taking values only on $\mathscr H_0$. Using saturation property of $\mathscr H_0$ (see Definition \ref{def_sat} and \Cref{thm_sat}), we deduce that $\cup_{j=0}^{\infty}\mathscr H_j	$ is dense in $H^k(\mathbb{T}^d)$. This implies that one can approximately control \eqref{eq_extension_delta} to any point in $H^k(\mathbb{T}^d)$ in small time. Finally, controllability in arbitrary time $T>0$ is obtained by steering the system close to the target $u_1$ in small time, and then keeping the trajectory near $u_1$ for a sufficiently long period by applying an appropriate control, see the proof of Theorem \ref{small_time}.
	
\end{rmk}
\subsubsection{Proof of \Cref{prop_asympt}}
\begin{proof}
Let us assume that $u_0\in H^{k+2}(\TT)$, $\varphi \in H^{3k+2}(\TT)$ and $\eta \in H^{k+2}(\TT)$. 
Thanks to \Cref{wellposed2}, equation \eqref{eq_extension_delta} has a unique maximal solution defined on $[0,\mathcal T^{\delta}),$ where $\mathcal T^{\delta}=\mathcal{T}(u_0, \delta^{-\frac{1}{3}}\varphi, \delta^{-{1}} \eta).$ Thus $u \in C([0,\mathcal T^\delta); H^k(\TT)).$
We denote $u(t) = \mathcal R_t(u_0, \delta^{-\frac{1}{3}}\varphi, {\delta^{-1}}\eta)$ for all $t \in [0, \mathcal T_\delta)$.  Moreover if $\mathcal T^\delta<+\infty$ we have 
$	\|u(t)\|_{H^{k}} \longrightarrow +\infty 
	\text{ as } t\to \mathcal {T^\delta}^{-}.$
Our aim is to show that there exists $\delta_0 > 0$ such that $\mathcal T_\delta > \delta$ for all $\delta \in (0, \delta_0)$ and thus $u \in C([0, \delta]; H^k(\TT))$. We will also deduce that the limit \eqref{asymp} holds. 
In order to prove this, we estimate the $H^k$-norm of $u(t)$. 
Let us introduce the functions
\begin{align}\label{eq_vw}
	\notag&	w(t):=u_0+t(\eta+\Delta(\varphi^3)),\\
	&v(t):=u(\delta t)-w(t),
\end{align}
which are well-defined for $t<\delta^{-1}\mathcal T^\delta.$ Our goal is to prove that $u(\delta)\longrightarrow u_0+(\eta+\Delta(\varphi^3))$ in $H^k(\TT)$ as $\delta\to 0^+.$ Thus, it suffices to show that \begin{align}\label{v1}\norm{v(1)}_{H^k}\longrightarrow 0 \text{ as } \delta \to 0^+.
\end{align}
Before proving \eqref{v1}, we need to ensure the existence of $\delta_0>0$ small enough such that, for all $\delta\in (0,\delta_0),$ $v$ is well defined in $[0,1],$ which means we need to prove that $\delta^{-1}\mathcal T^\delta> 1.$

\noindent
To simplify the computations, we assume throughout the proof that $\delta\in (0,1).$
Let us take $t< \min\{2, \delta^{-1}\mathcal T^\delta\}.$ Observe that  $v$ is a solution of the following problem
\begin{equation}\label{eq_extension_delta2}
	\begin{cases}
		v_t +\delta \Delta^2 \left(v+w + \delta^{-\frac{1}{3}}\varphi\right) +\delta \Delta\left(v+w+\delta^{-\frac{1}{3}}\varphi\right)=\delta \Delta \left(v+w + \delta^{-\frac{1}{3}}\varphi\right)^3-\Delta(\varphi^3),  \\
		v(0) = 0.
	\end{cases}
\end{equation}
As we take $t<2,$ from \eqref{eq_vw}, we have $\norm{w}_{H^k}\leq C\left(\norm{u_0}_{H^k}+\norm{\eta}_{H^k}+\norm{\varphi}^3_{H^{k+2}}\right)$ and also $\norm{w}_{H^{k+2}}\leq C\left(\norm{u_0}_{H^{k+2}}+\norm{\eta}_{H^{k+2}}+\norm{\varphi}^3_{H^{k+4}}\right)$  for some constant $C>0$. Multiplying equation \eqref{eq_extension_delta2} by $ v$ and integrating over $\TT$, we obtain after performing integration by parts
\begin{align}\label{est4}
	\notag	&\frac{1}{2} \frac{d}{dt}  \|v\|_{L^2}^2 
	+\delta\norm{\Delta v}_{L^2}^2
	\leq  \delta \| \Delta v\|_{L^2}\| \Delta w\|_{L^2} +\delta^{\frac{2}{3}} \| \Delta \varphi\|_{L^2 }\| \Delta v\|_{L^2 }\\
	\notag	&+ \delta \| \Delta v\|_{L^2 } \| v\|_{L^2 } +\delta \|  v\|_{L^2 }\| \Delta w\|_{L^2 } +\delta^{\frac{2}{3}} \| \Delta \varphi\|_{L^2 }\| v\|_{L^2 } \\
	&\qquad+\underbrace{\int_{\TT}v \left(\delta \Delta \left(v+w + \delta^{-\frac{1}{3}}\varphi\right)^3-\Delta(\varphi^3)\right)}_{\mathcal Q}.\end{align}
We now estimate $\mathcal Q:=\mathcal Q_1+\mathcal Q_2+\mathcal Q_3,$ where the explicit expressions for $\mathcal Q_i$ are the following
\begin{align*}
	\mathcal Q_1:=\delta\int_{\TT}v  \Delta \left(v+w\right)^3,  \mathcal Q_2:=3\delta^{\frac{2}{3}} \int_{\TT}v  \Delta\left((v+w)^2 \varphi\right),
	\mathcal Q_3:=3\delta^{\frac{1}{3}} \int_{\TT}v  \Delta\left((v+w) \varphi^2\right).
\end{align*}
For sufficiently small $\epsilon>0$, applying Young's inequality, we deduce the existence of a constant $C_\epsilon>0$ independent of $\delta$ such that we estimate $\mathcal Q_i$, $i=1,2$, as follows. 
\begin{align}\label{est_q}
\notag	\mathcal{Q}_1&\leq  \delta \norm{\Delta v}_{L^2 } \norm{(v+w)^3}_{L^2 }\leq \epsilon\delta  \norm{\Delta v}_{L^2 }^2+C_\epsilon\delta \norm{(v+w)^3}_{L^{2} }^2\\
\notag	&\leq\epsilon\delta  \norm{\Delta v}_{L^2 }^2+C_\epsilon\delta \left(\norm{ v}_{H^{k} }^6+\norm{\ w}_{H^{k} }^6\right),\\
\notag	\mathcal{Q}_2&\leq  3\delta^{\frac{2}{3}}\norm{\varphi}_{L^\infty } \norm{\Delta v}_{L^2 } \norm{(v+w)^2}_{L^2 }\\
\notag&\leq \epsilon\delta  \norm{\Delta v}_{L^2 }^2+C_\epsilon\delta^{\frac{1}{3}} \norm{(v+w)^2}_{H^{k} }^2\norm{\varphi}_{L^\infty }^2\\
	&\leq\epsilon\delta  \norm{\Delta v}_{L^2 }^2+C_\epsilon\delta^{\frac{1}{3}} \left(\norm{v}_{H^{k} }^4+\norm{w}_{H^{k} }^4\right)\norm{\varphi}_{L^\infty }^2.
\end{align}
Again, performing integration by parts and then applying Young's inequality, we deduce
\begin{align}\label{QQ3}
\notag	
\mathcal Q_3&=3\delta^{\frac{1}{3}} \int_{\TT}v  \Delta\left((v+w) \varphi^2\right)=-3\delta^{\frac{1}{3}} \int_{\TT}\nabla v  \left(\varphi^2\nabla(v+w)+(v+w) \nabla\varphi^2\right)\\
\notag	&=-3\delta^{\frac{1}{3}} \int_{\TT}|\nabla v|^2 \varphi^2-3\delta^{\frac{1}{3}} \int_{\TT}\nabla v \varphi^2 \nabla w -3\delta^{\frac{1}{3}} \int_{\TT}\nabla v  \left((v+w) \nabla\varphi^2\right)\\
\notag	&\leq 3\delta^{\frac{1}{3}} \int_{\TT} v  \nabla\cdot(\varphi^2 \nabla w)+ \frac{3}{2}\delta^{\frac{1}{3}} \int_{\TT}|v|^2\Delta(\varphi^2)  +3\delta^{\frac{1}{3}} \int_{\TT}  v\nabla\cdot(w\nabla \varphi^2)\\
\notag	&\leq 3\delta^{\frac{1}{3}}\norm{v}_{L^2 }\norm{\varphi^2 \Delta w+2\varphi \nabla \varphi \nabla w}_{L^2 }+\frac{3}{2}\delta^{\frac{1}{3}}\norm{v}^2_{L^2 }\norm{\Delta\varphi^2}_{L^\infty }\\
\notag	&\qquad \qquad \qquad \quad+3\delta^{\frac{1}{3}}\norm{v}_{L^2 }\norm{w \Delta \varphi^2+2\varphi \nabla \varphi \nabla w}_{L^2 }\\
\notag	&\leq C\delta^{\frac{1}{3}} \norm{v}^2_{L^2 }\left(1+\norm{\varphi}^2_{H^{k+2} }\right)+C\delta^{\frac{1}{3}}\bigg(\norm{\varphi}^4_{H^{k} }\norm{w}^2_{H^2 }+\norm{w}^2_{H^{1} }\norm{\varphi}^4_{H^{k+1} }\\&\qquad\qquad  \qquad \quad +\norm{\varphi}^4_{H^{k+1} }
	+\norm{\varphi}^4_{H^{k+2} }\norm{w}^2_{L^2}
	\bigg).\end{align}
Using Young’s inequality in \eqref{est4}, inserting \eqref{est_q} and \eqref{QQ3} into it, and noting that $\delta\in(0,1)$, we deduce the existence of a constant $C>0$, independent of $\delta$ such that
\begin{align}\label{estl21}
	\notag	 \frac{d}{dt}  \|v\|_{L^2}^2 
	+\delta\norm{\Delta v}_{L^2}^2
&\leq	C \delta \|\Delta w\|_{L^2 }^2 + C\delta^{\frac{1}{3}} \|\Delta \varphi\|_{L^2 }^2+C\delta^{\frac{1}{3}} \|v\|_{L^2 }^2+C\delta^{\frac{1}{3}}\left(1+\norm{v}^6_{H^k}+\norm{v}^4_{H^k}\right)\\
	 &\leq C\delta^{\frac{1}{3}}\norm{v}^2_{L^2}+C\delta^{\frac{1}{3}}\left(\norm{v}^6_{H^k}+1\right).
	\end{align}
Similarly multiplying equation \eqref{eq_extension_delta2} by $(-\Delta)^k v$ and integrating over $\TT$, we have
\begin{align}\label{est5}
	\notag	\frac{1}{2} \frac{d}{dt}  &\|(-\Delta)^{\frac{k}{2}} v\|_{L^2 }^2 
	+ \delta\|(-\Delta)^{\frac{k+2}{2}} v\|_{L^2 }^2
	\leq  \delta \| (-\Delta)^{\frac{k+2}{2}} v\|_{L^2 }\|  (-\Delta)^{\frac{k+2}{2}}w\|_{L^2 }\\
	\notag	& +\delta^{\frac{2}{3}} \| (-\Delta)^{\frac{k+2}{2}} \varphi\|_{L^2 }\| (-\Delta)^{\frac{k+2}{2}} v\|_{L^2 }
	+ \delta \| (-\Delta)^{\frac{k+2}{2}} v\|_{L^2 } \| (-\Delta)^{\frac{k}{2}} v\|_{L^2 } \\
	\notag&+\delta \|  (-\Delta)^{\frac{k}{2}} v\|_{L^2 }\| (-\Delta)^{\frac{k+2}{2}} w\|_{L^2 }
	+\delta^{\frac{2}{3}} \|(-\Delta)^{\frac{k+2}{2}} \varphi\|_{L^2 }\| (-\Delta)^{\frac{k}{2}} v\|_{L^2 } \\
	&\qquad \qquad \qquad+\underbrace{\int_{\TT}(-\Delta)^{k} v \left(\delta \Delta \left(v+w + \delta^{-\frac{1}{3}}\varphi\right)^3-\Delta(\varphi^3)\right)}_{\mathcal P},
\end{align}
where the corresponding nonlinear terms $\mathcal P$ are detailed below.
\begin{align*}
	\mathcal P&:=	\delta\int_{\TT}(-\Delta)^{k}v  \Delta \left(v+w\right)^3+3\delta^{\frac{2}{3}} \int_{\TT} (-\Delta)^{k} v  \Delta\left((v+w)^2 \varphi\right)+{3\delta^{\frac{1}{3}} \int_{\TT} (-\Delta)^{k} v \Delta  \left((v+w) \varphi^2\right)}\\
	&=:\mathcal P_1+\mathcal P_2+\mathcal P_3.
\end{align*}
In a similar manner as for $\mathcal Q$, we estimate $\mathcal P_i,$ $i=1,2,$
\begin{align}\label{est_p1}
	\notag	\mathcal P_1&\leq \delta \| (-\Delta)^{\frac{k+2}{2}} v\|_{L^2 } \| (-\Delta)^{\frac{k}{2}} (v+w)^3\|_{L^2 }\\
	\notag	&\leq \delta\epsilon \|(-\Delta)^{\frac{k+2}{2}} v\|_{L^2}^2+C_\epsilon \delta\| v\|_{H^k }^6+C_\epsilon\delta\| w\|_{H^k }^6\\
	\notag	\mathcal P_2&\leq 3\delta^{\frac{2}{3}} \| (-\Delta)^{\frac{k+2}{2}} v\|_{L^2 } \| (-\Delta)^{\frac{k}{2}} \left((v+w)^2\varphi\right)\|_{L^2 }\\
		&\leq \delta\epsilon \|(-\Delta)^{\frac{k+2}{2}} v\|_{L^2}^2+C_\epsilon \delta^{\frac{1}{3}}\| v\|_{H^k }^4\norm{\varphi}_{H^k }^2+C_\epsilon \delta^{\frac{1}{3}}\| w\|_{H^k }^4\norm{\varphi}_{H^k }^2,
	\end{align}
	where $C_\epsilon>0$ is independent of $\delta$. Observe that in the term $\mathcal P_3$ we have a coefficient with order $\delta^{\frac{1}{3}}$. Since  $\delta<1$ is small enough, this term can not be estimated using Young's inequality to absorb it into the term $ \delta\|(-\Delta)^{\frac{k+2}{2}} v\|_{L^2 }^2$ appearing on the left of the inequality \eqref{est5}. Indeed doing so, would produce a factor of the form $\delta^{-r}$ for some $r>0$ at the right which becomes large and prevents us from taking the limit $\delta\to 0^+.$ Therefore, we need to apply integration by parts successively to reduce the differential order on $v$ transferring derivatives onto the function
 $\varphi$ and then estimate the resulting terms with $\norm{v}_{H^k }.$ The details of this argument are provided in Lemma \ref{estimateofp3} of Appendix \ref{estp3} which leads to the following estimate
\begin{align}\label{p3est}
	\mathcal P_3&\leq \delta^{\frac{1}{3}} C\norm{\varphi}^2_{H^{3k+2} }\left(\|v\|_{H^k }^2+\|w\|_{H^{k+2} }^2\right).
\end{align}
Applying Young’s inequality to \eqref{est5}, combining it with \eqref{est_p1}--\eqref{p3est}, we obtain a constant $C>0$, independent of $\delta \in (0,1)$, such that
\begin{align}\label{est_pq}
	\notag  &\frac{d}{dt} \left(  \|(-\Delta)^{\frac{k}{2}} v\|_{L^2 }^2 \right)
	+ \delta\|(-\Delta)^{\frac{k+2}{2}} v\|_{L^2 }^2\\
		&
	\leq  C \bigg(\delta \|  (-\Delta)^{\frac{k+2}{2}}w\|^2_{L^2 }
	 + \delta^{\frac{1}{3}} \|  (-\Delta)^{\frac{k+2}{2}}\varphi\|^2_{L^2 }+ \delta\|(-\Delta)^{\frac{k}{2}} v\|_{L^2 }^2\bigg)+ C\delta^{\frac{1}{3}}(1+\norm{v}^6_{H^k}).
\end{align}
Adding \eqref{estl21} and \eqref{est_pq}, using the fact that $\delta\in (0,1)$, we have
\begin{align*}
	 \frac{d}{dt} \left( \|v\|_{L^2 }^2 + \|(-\Delta)^{\frac{k}{2}} v\|_{L^2 }^2 \right)
	&+ \delta\|(-\Delta)^{\frac{k+2}{2}} v\|_{L^2 }^2+\delta\norm{\Delta v}_{L^2 }^2\\
	&
	\leq C\delta^{\frac{1}{3}}\left(\norm{v}^2_{L^2}+\norm{v}^2_{H^k}\right)+C\delta^{\frac{1}{3}}\left(1+ \norm{v}^6_{H^k }\right).
\end{align*}
where the constant $C>0$ absorbs all the fixed quantities as $\|\varphi\|$ and $\|w\|$. 
 The above inequality further simplifies to
\begin{align*}
	 \frac{d}{dt}\left( \|v\|_{L^2 }^2 + \|(-\Delta)^{\frac{k}{2}} v\|_{L^2 }^2 \right)& 
	 \leq C \delta^{\frac{1}{3}}\left(1+\|v(t)\|^6_{H^k}\right).
\end{align*}
By Gr\"onwall's inequality, we have
\begin{align}\label{grn}
	\|v(t)\|^2_{H^k{ }}\leq C \delta^{\frac{1}{3}}t+C \delta^{\frac{1}{3}}\int_{0}^{t}\|v(s)\|^6_{H^k{ }} \, ds, \,\, \text{ for }t<  \min\{2, \delta^{-1}\mathcal T^\delta\}.
	\end{align}
Next, we will show the existence of $\delta_0\in (0,1)$ small enough such that, for all $\delta\in (0,\delta_0),$  $\delta^{-1}\mathcal T^\delta> 1.$ 
Let us denote $$\gamma^{\delta}:=\sup\left\{t<\delta^{-1}\mathcal T^{\delta}: \sup\limits_{0\leq s\leq t}\norm{v(s)}_{H^k }<1\right\}.$$ Inequality \eqref{grn} ensures that $\gamma^\delta>0.$ If $\gamma^\delta=+\infty,$ then $\delta^{-1}\mathcal T^\delta> 1$ follows trivially. Let us consider the case $\gamma^\delta<+\infty.$ Note that, to prove $\delta^{-1}\mathcal T^{\delta}>1,$ it is enough to show that, there exists $\delta_0>0$ small enough such that, for all $\delta\in (0,\delta_0),$  $\gamma^\delta> 1.$ We prove it by contradiction. If possible, let us assume that, for every $\delta_0>0$ there exists $\delta\in (0,\delta_0)$ such that $\gamma^{\delta}\leq 1.$
Thanks to \eqref{grn}, we have
\begin{align*}
	1=\|v(\gamma^\delta)\|^2_{H^k{ }}\leq C \delta^{\frac{1}{3}} \gamma^\delta+C \delta^{\frac{1}{3}}\int_{0}^{\gamma^\delta}\|v(s)\|^6_{H^k{ }} \, ds.
\end{align*}
Choose $\delta_0$ sufficiently small such that for all $\delta \in (0,\delta_0)$ we have $C \delta^{\frac{1}{3}} <\frac{1}{2}.$
This gives that \begin{align*}C \delta^{\frac{1}{3}} \gamma^\delta+C \delta^{\frac{1}{3}}\int_{0}^{\gamma^\delta}\|v(s)\|^6_{H^k{ }} \, ds<1,
\end{align*}
which contradicts the fact $\|v(\gamma^\delta)\|_{H^k{ }}=1$ above.
Thus we can conclude that, there exists a $\delta_0>0$ such that when $0<\delta<\delta_0,$ $\gamma^\delta> 1$ and thus $\delta^{-1}\mathcal T^{\delta}> 1.$

\noindent
Therefore, solution $v$ of \eqref{eq_extension_delta2} is well-defined in $[0,1]$ and finally from \eqref{grn} we have 
\begin{align*}
		\|v(1)\|^2_{H^k}\leq C_1 \delta^{\frac{1}{3}} \longrightarrow 0 \text{ as } \delta \to 0^+.
\end{align*}
 Hence, we have proved that solution $u(t)$ of \eqref{eq_extension_delta} is well-defined in $[0,\delta]$ and the following limit holds $$ u(\delta)= \mathcal R_{\delta}(u_0, \delta^{-\frac{1}{3}}\varphi, {\delta}^{-1}\eta)\to u_0+(\eta+\Delta(\varphi^3)) \text{ in } H^k(\TT) \text{ as } \delta \to 0^+.$$
\end{proof}

\subsection{Saturation property}\label{sec_sat}
This section is devoted to the so-called saturation property, which serves as a key ingredient together with \Cref{prop_asympt} for proving Theorem \ref{small_time}.
Let us recall the finite dimensional subspace $\mathscr{H}_0$ of $H^{k}(\mathbb T^d)$ defined in \eqref{hnot}. We now define a non-decreasing sequence of finite dimensional subspaces as follows
\begin{equation}\label{construction_h}
	\mathscr H:=\mathscr H_0, \qquad
	 \mathscr{H}_j:=\mathcal F(\mathscr H_{j-1}), \ j\in\mathbb N^*, \qquad \mathscr{H}_\infty:=\cup_{j=1}^{\infty}\mathscr{H}_{j-1}
\end{equation}
where
\begin{equation}\label{ansatz_PI}
	\mathcal F(G):=\spn\left\{f_0+\sum_{i=1}^{m}\Delta(f_i^3): f_0,f_i \in G, \, \textnormal{for some } m\in\mathbb N^*\right\}
\end{equation}
for some set $G$. 

\begin{defi}\label{def_sat}
	We say that $\mathscr H$ is saturating if $\mathscr H_{\infty}$ is dense in $H^{k}(\mathbb T^d)$.
\end{defi}
For any subset $I\subset \mathbb Z^d$, let us denote the set of trigonometric functions 
\begin{equation}\label{notation_A_I}
	A_I:=\{\sin (x \cdot p),\cos (x \cdot p) : p \in I\}.
\end{equation}
For $N\in\mathbb N$, we define the set of indices
\begin{equation*}
	I_N=\left\{p=(p_1,p_2,\ldots,p_d)\in\mathbb Z^d: \sum_{i=1}^{d}|p_i|\leq N+1\right\}.
\end{equation*}
Observe that
\begin{equation*}
	\mathscr{H}_0=\spn A_{I_0}.
\end{equation*}
The goal of this section is to prove the following result. 
\begin{proposition}\label{thm_sat}
	The space $\mathscr{H}_0$ is saturating.
\end{proposition}
The proof relies on a characterization of certain elements in the subspace $\mathscr{H}_{j}$ defined in \eqref{construction_h}. This characterization is stated in \Cref{prop:charac}. In what follows, we will make use of the following auxiliary lemmas, whose proofs can be obtained by direct computations.

\begin{lemma}\label{lem_aux}
	Let $p,q,r$ be real-valued scalar functions. Then the following identity holds
	\begin{equation*}
		\begin{split}
			pqr&=\left[\left(\frac{p}{2}+\frac{q}{3}+\frac{r}{4}\right)^3-\left(\frac{p}{2}+\frac{q}{3}-\frac{r}{4}\right)^3-\left(\frac{p}{2}-\frac{q}{3}+\frac{r}{4}\right)^3+\left(\frac{p}{2}-\frac{q}{3}-\frac{r}{4}\right)^3\right] \\
			&=:f_1(p,q,r)^3+f_2(p,q,r)^3+f_3(p,q,r)^3+f_4(p,q,r)^3 \\
			&=:Q(p,q,r).
		\end{split}
	\end{equation*}
\end{lemma}

\begin{lemma}\label{lem_anti}
	For any real-valued scalar functions $p,q,r$, we have $Q(p,q,-r)=-Q(p,q,r)$.
	
\end{lemma}

Hereinafter, for any nonnegative real numbers $a<b$, we define  the set $\inter{a,b}:=[a,b]\cap \mathbb N$. 

\begin{lemma}\label{lem_iden_Q}
	For any subset $I\subset \mathbb Z^d$, if $p,q,r\in A_I$, then $f_i(p,q,r)\in\spn A_{I}$ for $i\in\inter{1,4}$. Additionally, we have 
	\begin{equation*}
		\Delta Q(p,q,r)=\sum_{i=1}^{4}\Delta\left(f_i(p,q,r)^3\right).
	\end{equation*}
\end{lemma}
We are in position to prove the following result.
\begin{proposition}\label{prop:charac}
	For $N\in\mathbb N^*$, the following holds
	\begin{equation*}
		\spn A_{I_N}\subset \mathscr H_{N}.
	\end{equation*}
\end{proposition}
\begin{proof}
	We proceed by induction. 
	
	\medskip
	
	\textbf{-- Case $N=1$}. From the notation introduced in \eqref{notation_A_I}, note that
	\begin{equation*}
		A_{I_1}=\left\{\sin(x\cdot p),\cos(x\cdot p) : \sum_{i=1}^{d}|p_i|\leq 2\right\}.
	\end{equation*}
	By construction (see \eqref{construction_h} and \eqref{ansatz_PI}), we already have $\spn A_{I_0}\subset \mathscr H_1$. Therefore, to prove that $\spn{A_{I_1}}\subset \mathscr H_{1}$, it
	suffices to focus on the functions whose indices satisfy $\sum_{i=1}^{d}|p_i|\leq N+1= 2$. For these elements, we will show that each of them can be generated (up to a constant) by means of the ansatz (cf. \eqref{ansatz_PI})
	\begin{equation}\label{ansatz_0}
		f_0+\sum_{i=1}^{m}\Delta(f_i^3) \textnormal{ with } f_0,f_i \in \mathscr H_0 \textnormal{ and some } m\in\mathbb N^*.
	\end{equation}
	Using the symmetry properties of the sine and cosine functions, it suffices to consider the cases with $p = e_i\pm e_j$, where $e_i, e_j$ are vectors in the canonical basis of $\mathbb{R}^d$, since all other functions can be generated from these by symmetry.
	
	\begin{itemize}
		
		\item \textbf{The single angles}. Let us generate $\sin 2\theta$ for $\theta=x_i$, $i\in\inter{1,d}$. Using Lemma \ref{lem_aux}, we have that
		\begin{equation*}
			\sin \theta\cos \theta=Q\left(\sin \theta,\cos \theta,1\right).
		\end{equation*}
		Applying the operator $-\Delta$ on both sides of the above expression, we have
		\begin{align*}
			2 \sin 2\theta=&-\Delta Q(\sin \theta,\cos \theta,1).
		\end{align*}
		By Lemma \ref{lem_iden_Q} and recalling the ansatz \eqref{ansatz_0}, we deduce that $\sin 2\theta$ for $\theta=x_i$ belongs to $\mathscr H_1$. To generate $\cos 2\theta$ for $\theta=x_i$, $i\in\inter{1,d}$, we argue as follows. We have that 
		\begin{equation*}
			\cos^2 \theta=Q(\cos \theta,\cos \theta,1).
		\end{equation*}
		Using that $\cos 2\theta=2\cos^2 \theta-1$ and the above expression, we have that
		\begin{equation*}
			\frac{1}{2}\cos 2\theta=Q(\cos \theta,\cos \theta,1)-\frac{1}{2}.
		\end{equation*}
		Therefore, 
		\begin{align*}
			2\cos 2\theta=-\Delta(Q(\cos \theta,\cos \theta,1)),
		\end{align*}
		and, by Lemma \ref{lem_iden_Q}, $\cos(2\theta)$ for $\theta=x_i$ belongs to the vector space $\mathscr H_1$.

		\item \textbf{The sum of angles}. Using Lemma \ref{lem_aux}, we can write
		\begin{equation*}
	\sin x_i\cos x_j=Q(\sin x_i,\cos x_j,1) \quad\text{and}\quad \cos x_i \sin x_j = Q(\cos x_i,\sin x_j,1),
		\end{equation*}
	which lead to
		\begin{equation*}
			\sin(x\cdot p)=	\sin(x\cdot (e_i\pm e_j))=	\sin(x_i\pm x_j)=Q(\sin x_i,\cos x_j,1) + Q(\cos x_i,\sin x_j,\pm 1),
		\end{equation*}
		where we have used Lemma \ref{lem_anti}. Applying the operator $-\Delta$ in both sides of the above expression we have
		\begin{equation*}
			2\sin(x_i\pm x_j)= -\Delta(Q(\sin x_i,\cos x_j,1))-\Delta(Q(\cos x_i,\sin x_j,\pm 1))
		\end{equation*}
		so $\sin(x_i\pm x_j)\in \mathscr H_1$ by recalling Lemma \ref{lem_iden_Q} and \eqref{ansatz_0}. On the other hand, we have that
		\begin{equation*}
			\cos x_i\cos x_j=Q(\cos x_i,\cos x_j,1) \quad\text{and}\quad \sin x_i \sin x_j = Q(\sin x_i,\sin x_j,1),
		\end{equation*}
		which in turn yield
		\begin{equation*}
			\cos(x\cdot p)=\cos(x_i\pm x_j)=Q(\cos x_i,\cos x_j,1) + Q(\sin x_i,\sin x_j,\mp 1).
		\end{equation*}
	Similarly applying $-\Delta$ on both sides of the above expression, we have that 
		\begin{equation*}
			2\cos(x_i\pm x_j)=-\Delta Q(\cos x_i,\cos x_j,1) - \Delta Q(\sin x_i,\sin x_j,\mp 1),
		\end{equation*}
		and thus $\cos( x\cdot p)\in \mathscr H_1$. 
	\end{itemize}
	With this, we have proved that $\spn A_{I_1} \subset \mathscr H_1$.
	
	\medskip
	
	\textbf{-- The inductive step}. Assume that $\spn{A}_{I_N} \subset \mathscr H_{N}$. Then, by the construction described in \eqref{construction_h} and \eqref{ansatz_PI},  we have that $\spn A_{I_N}\subset \mathscr{H}_{N+1}$. Therefore, our task is reduced to showing that for any $p\in \mathbb Z^d$ such that $\sum_{i=1}^{d}|p_i|=N+2$, the corresponding elements of the set $A_{I_{N+1}}$ can be generated by means of the ansatz 
	\begin{equation}\label{ansatz_N}
		f_0+\sum_{i=1}^{m}\Delta(f_i^3) \textnormal{ with } f_0,f_i \in \mathscr H_N \textnormal{ and some } m\in\mathbb N^*.
	\end{equation}
\begin{itemize}
		\item \textbf{The single angle}. We analyze the cases when $p\in \mathbb{Z}^d$ is equal to $\displaystyle (0,\ldots,\underbrace{N+2}_{i-\text{th}},\ldots,0)$. Let $\theta$ be equal to $x_i$, $i\in\inter{1,d}$. We note that 
		\begin{equation*}
		\sin( x \cdot p)=	\sin((N+2)\theta)=\sin((N+1)\theta+\theta)=\sin((N+1)\theta)\cos \theta + \cos((N+1)\theta)\sin \theta
		\end{equation*}
		By the induction hypothesis $\sin((N+1)\theta)$ and $\cos((N+1)\theta)$ belong to $\mathscr H_N$, hence using Lemma \ref{lem_aux}, we can write
		\begin{equation*}
			\sin((N+2)\theta)=Q(\sin((N+1)\theta),\cos \theta,1)+Q(\cos((N+1)\theta),\sin \theta,1)
		\end{equation*}
		and by applying the differential operator $-\Delta$ we have
		\begin{equation*}
			(N+2)^2\sin((N+2)\theta)=-\Delta Q(\sin((N+1)\theta),\cos \theta,1)-\Delta Q(\cos((N+1)\theta),\sin \theta,1)
		\end{equation*}
		which shows that $\sin((N+2)\theta)\in \mathscr H_{N+1}$ for $\theta=x_i$, $i\in\inter{1,d}$, by recalling Lemma \ref{lem_iden_Q} and \eqref{ansatz_N}. On the other hand, we have that 
		\begin{equation*}
			\cos((N+2)\theta)=\cos((N+1)\theta+\theta)=\cos((N+1)\theta)\cos \theta-\sin((N+1)\theta)\sin \theta.
		\end{equation*}
		Hence, using Lemma \ref{lem_aux} and Lemma \ref{lem_anti}, we can write
		\begin{equation*}
			\cos((N+2)\theta)=Q(\cos((N+1)\theta),\cos \theta,1)+Q(\sin((N+1)\theta),\sin \theta,-1),
		\end{equation*}
		whence
		\begin{equation*}
			(N+2)^2\cos((N+2)\theta)=-\Delta Q(\cos((N+1)\theta),\cos \theta,1)-\Delta Q(\sin((N+1)\theta),\sin \theta,-1),
		\end{equation*}
		and thus $\cos((N+2)\theta)\in \mathscr{H}_{N+1}$.
		
		\item \textbf{The sum of angles}. Let  $p\in \mathbb Z^d$ such that $\sum_{i=1}^{d}|p_i|=N+2$, and $|p_i|<N+2,$ for all $1\leq i\leq d $ be arbitrary. Then $p$ can be written as $p=l\pm q,$ where $l,q\in \mathbb Z^d $ such that $\sum_{i=1}^{d}|l_i|=N+1$, and $\sum_{i=1}^{d}|q_i|=1$ . Using trigonometric identities, 
		\begin{align*}
	&	\sin(x\cdot (l\pm q))=\sin(x\cdot l)\cos(x\cdot q)\pm \cos(x\cdot l )\sin(x\cdot  q)\\
	&	\cos(x\cdot (l\pm q))=\cos(x\cdot l)\cos(x\cdot q)\mp \sin(x\cdot l )\sin(x\cdot  q),
		\end{align*}
		along with
		Lemma \ref{lem_aux} and Lemma \ref{lem_anti}, we have
		\begin{align*}
				\sin(x\cdot p)&=\sin(x\cdot l)\cos(x \cdot q) \pm \cos(x\cdot l)\sin(x\cdot q) \\
			&=Q\left(\sin(x\cdot l),\cos(x \cdot q),1\right)+Q\left(\cos(x\cdot l),\sin(x \cdot q),\pm 1\right)
		\end{align*}
		Note that by the construction \eqref{construction_h}--\eqref{ansatz_PI}, the functions $\pm1$, $\sin(x\cdot l),\cos(x \cdot q), \cos(x\cdot l),\sin(x \cdot q)$ belong to $\mathscr{H}_{N}$. Hence, by taking $-\Delta$ on both sides of the previous identity and Lemma \ref{lem_iden_Q} give
		\begin{align*}
			|p|^2\sin(x\cdot p)&=-\Delta Q\left(\sin(x\cdot l),\cos(x\cdot q),1\right)-\Delta Q\left(\cos(x\cdot l),\sin(x\cdot q),\pm 1\right)
		\end{align*}
		and hence $\sin(x\cdot p)$ with $\sum_{i=1}^{d}|p_i|=N+2$, and $|p_i|<N+2,$ belong to $\mathscr H_{N+1}$. Analogously, 
		\begin{align*}
			\cos(x\cdot p )&=\cos(x\cdot l)\cos(x\cdot q)\mp\sin(x \cdot l)\sin(x\cdot q) \\
			&=Q(\cos(x\cdot l),\cos(x\cdot q),1)+Q(\sin(x\cdot l),\sin(x \cdot q),\mp 1),
		\end{align*}
		from which we deduce that 
		\begin{align*}
			|p|^2\cos(x\cdot p)&=-\Delta Q\left(\cos(x\cdot l),\cos(x\cdot q),1\right)-\Delta Q\left(\sin(x\cdot l),\sin(x \cdot q),\mp 1\right)
		\end{align*}
		and thus $\cos(x\cdot p)$ belongs to $\mathscr H_{N+1}$. 
		
	\end{itemize}
\end{proof}

\section{Proof of global approximate controllability}\label{sec_global}
In this section, we will prove the main global approximate controllability result Theorem \ref{small_time}. Our argument follows the approach used in \cite[Theorem 3.3]{Ner21}. The strategy is to first prove small-time approximate controllability to any point of 
the affine space $u_0 + \mathscr{H}_N$ by combining \Cref{prop_asympt} with an induction argument 
in $N\in \mathbb{N}$. Using the saturation property \Cref{thm_sat}, we then deduce small-time approximate controllability to every point in $H^k(\TT)$. Finally, global approximate controllability in any time $T$ is obtained 
by steering the system close to the target $u_1$ in a short time and then keeping it 
near $u_1$ for a sufficiently long interval. Let us begin by proving global approximate controllability to $u_0+\mathscr{H}_N$ in very small time.
\begin{proposition}\label{prop_small_time}
Equation \eqref{eq_main} is globally $H^k$-approximately controllable in the following  sense: for any $u_0,u_1\in H^k(\TT)$ and $\sigma,\epsilon>0$, there exist $\delta\in(0,\sigma)$ and a control $f\in L^\infty(0,\delta;\mathscr H_0)$ such that the solution of \eqref{eq_main} satisfies
	\begin{equation*}
		\| \mathcal R_\delta(u_0, 0, f) - u_1 \|_{H^{k}} < \epsilon.
	\end{equation*} 
	\end{proposition}
\begin{proof}
	Let us take $u_0 \in H^{k+2}(\TT)$. 
	We first establish the approximate controllability in small time from $u_0$ to $u_0 + \mathscr{H}_N$ by induction, along with the asymptotic property from \Cref{prop_asympt}. More precisely, our goal is to show the following for all $N\in \mathbb{N}$
	\begin{equation}\tag{$\mathcal P_N$}\label{propPN}
\begin{aligned}
&\text{For any } u_0\in H^{k+2}(\TT),\ u_1\in u_0+\mathscr H_N,\ \sigma,\epsilon>0, \text{there exists}\, \delta\in(0,\sigma) \text{ and a}\\
&\text{control } f\in L^\infty(0,\delta;\mathscr H_0),\ \text{such that the solution of \eqref{eq_main} is well-defined in $[0,\delta]$ and } \\ 
& \hspace{5 cm} \| \mathcal R_\delta(u_0,0,f)-u_1\|_{H^k}<\epsilon.
\end{aligned}
\end{equation}

	\noindent\textbf{Step 1. Controllability from $u_0$ to $u_0 + \mathscr{H}_0$.}  
	For $u_1 \in u_0 + \mathscr{H}_0$, there exists $\eta \in \mathscr{H}_0$ such that $u_1 = u_0 + \eta$.  
	Thanks to \Cref{prop_asympt},
	\begin{equation*}
	\mathcal R_\delta(u_0, 0, \delta^{-1} \eta) \to u_0 + \eta, \quad \text{as } \delta \to 0^+.
	\end{equation*}
	Hence, for any $\epsilon>0$ there exists $\delta \in (0, \sigma)$ such that the solution of \eqref{eq_main} is well-defined in $[0,\delta]$ and
	\begin{equation*}
		\| \mathcal R_\delta(u_0, 0, f_1) - u_1 \|_{H^{k}} < \epsilon,
	\end{equation*}
	where the constant (in time) control $f_1$ is of the form $f_1(t, \cdot) = \tfrac{1}{\delta}\eta(\cdot)$.  
	Thus, equation \eqref{eq_main} is small-time approximately controllable from $u_0$ to $u_0 + \mathscr{H}_0$ by means of an $\mathscr{H}_0$-valued control. Therefore $\mathcal P_0$ holds. 
	
	\medskip
	\noindent\textbf{Step 2. Controllability from $u_0$ to $u_0 + \mathscr{H}_N$.} 
	Assume that $\mathcal{P}_{N-1}$ holds. That means \eqref{eq_main} is approximately controllable from $u_0$ to $u_0 + \mathscr{H}_{N-1}$ using an $\mathscr{H}_0$ valued control. Let $\epsilon>0$ and $\sigma>0.$ Assume that $u_1 \in u_0 + \mathscr{H}_N$. Then there exist $\eta, \varphi_1, \dots, \varphi_m \in \mathscr{H}_{N-1}$ such that $u_1$ can be written as
	\begin{equation}\label{eq:eta_decomp1}
	u_1 = u_0 + \eta + \sum_{i=1}^m \Delta (\varphi_i^3),
	\end{equation}
	for some $m \in \mathbb{N}^*$. Indeed, according to the definition of $\mathscr H_N,$ it immediately follows that $u_1$ has the form 
	\begin{equation*}
		u_1 =  u_0+c \widetilde \eta+\sum_{i=1}^{m}c_i \Delta(\widetilde \xi_i)^3,
	\end{equation*}
	where $m \in \mathbb{N}^*$, $c, c_i\in \mathbb R,$  and $\widetilde \eta, \widetilde \xi_i \in \mathscr H_{N-1},$ $i \in \llbracket 0,m\rrbracket.$
	Since the nonlinearity is cubic and $\mathscr{H}_N$ is a subspace, it is 
	straightforward to rewrite the expression above in the form \eqref{eq:eta_decomp1}, 
	irrespective of the sign of the constants $c_i$.
	

	We show that $u_1$, defined in \eqref{eq:eta_decomp1}, can be
	approximately reached in arbitrarily small time. We proceed by
	induction on $m\in\mathbb{N}^*$. For the base case $m=1$, we have
	$u_1=u_0+\eta+\Delta(\varphi_1^3)$.
Using \Cref{prop_asympt} 
we have 
\begin{align*}
	\mathcal R_\delta(u_0, \delta^{-\frac{1}{3}}\varphi_1, 0)\to u_0 + \Delta(\varphi_1^3), \quad \text{ in } H^k(\TT) \text{ as } \delta \to 0^+.
\end{align*}
From the flow property \eqref{flow_property},  $\mathcal	R_\delta(u_0 + \delta^{-1/3}\varphi_1, 0, 0) - \delta^{-1/3}\varphi_1=\mathcal R_\delta(u_0, \delta^{-\frac{1}{3}}\varphi_1, 0)$, we obtain
\begin{equation*}
	\mathcal	R_\delta(u_0 + \delta^{-1/3}\varphi_1, 0, 0) - \delta^{-1/3}\varphi_1 \to u_0 +\Delta(\varphi_1^3)
	\quad \text{in } H^{k}({\TT}), \text{ as } \delta \to 0^+.
\end{equation*}
	Hence, there exists $\delta_2\in(0,\sigma/3)$ such that the
	corresponding solution is well-defined on $[0,\delta_2]$ and
	\begin{equation}\label{apprx2}
		\left\|
		\mathcal R_{\delta_2}
		(u_0+\delta_2^{-1/3}\varphi_1,0,0)
		-\bigl(u_0+\delta_2^{-1/3}\varphi_1
		+\Delta(\varphi_1^3)\bigr)
		\right\|_{H^k}<\frac{\epsilon}{3}.
	\end{equation}
	We now keep $\delta_2$ fixed. Applying \Cref{max} and the
	stability estimate \eqref{stab_est} to the reference solution
	$\mathcal R(u_0+\delta_2^{-1/3}\varphi_1,0,0)$ on
	$[0,\delta_2]$, we obtain constants $\rho>0$ and $C>0$ such
	that, whenever
	\begin{equation}\label{eq_v0}
		\|v_0-(u_0+\delta_2^{-1/3}\varphi_1)\|_{H^k}<\rho,
	\end{equation}
	the solution $\mathcal R(v_0,0,0)$ is well-defined on
	$[0,\delta_2]$ and satisfies
	\begin{equation}\label{eq_conv0}
		\left\|
		\mathcal R_{\delta_2}(v_0,0,0)
		-\mathcal R_{\delta_2}
		(u_0+\delta_2^{-1/3}\varphi_1,0,0)
		\right\|_{H^k}
		\leq C
		\|v_0-(u_0+\delta_2^{-1/3}\varphi_1)\|_{H^k}.
	\end{equation}
	Here, $\rho$ and $C$ depend on $\delta_2$.
	Let us choose
	\begin{equation*}
		0<\alpha<
		\min\left\{\rho,\frac{\epsilon}{3C}\right\}.
	\end{equation*}
	Since $u_0+\delta_2^{-1/3}\varphi_1
	\in u_0+\mathscr H_{N-1}$, the induction hypothesis implies
	that there exist $\delta_1\in(0,\sigma/3)$ and a control
	$f_2\in L^\infty((0,\delta_1);\mathscr H_0)$ such that
	$\mathcal R(u_0,0,f_2)$ is well-defined on $[0,\delta_1]$ and
	\begin{equation}\label{aprx1}
		\left\|
		\mathcal R_{\delta_1}(u_0,0,f_2)
		-(u_0+\delta_2^{-1/3}\varphi_1)
		\right\|_{H^k}<\alpha.
	\end{equation}
	By the choice of $\alpha$, putting $v_0=\mathcal R_{\delta_1}(u_0,0,f_2)$ in \eqref{eq_v0}, the solution
	$\mathcal R(\mathcal R_{\delta_1}(u_0,0,f_2),0,0)$ is
	well-defined on $[0,\delta_2]$ and 
	\begin{align}\label{apprx31}
		&
		\left\|
		\mathcal R_{\delta_2}
		(\mathcal R_{\delta_1}(u_0,0,f_2),0,0)
		-\mathcal R_{\delta_2}
		(u_0+\delta_2^{-1/3}\varphi_1,0,0)
		\right\|_{H^k}
	<C\alpha
		<\frac{\epsilon}{3}.
	\end{align}
	Define the piecewise control
	$f$ by
	\begin{equation*}
		f(t)=
		\begin{cases}
			f_2(t),&t\in(0,\delta_1),\\
			0,&t\in(\delta_1,\delta_1+\delta_2).
		\end{cases}
	\end{equation*}
	Thanks to the flow property \eqref{flow2}, the solution
	$\mathcal R(u_0,0,f)$ is well-defined on
	$[0,\delta_1+\delta_2]$. Moreover, combining
	\eqref{apprx2} and \eqref{apprx31}, 
	we deduce that
	\begin{align}\label{apprx3}
		\notag
		&\left\|
		\mathcal R_{\delta_1+\delta_2}(u_0,0,f)
		-\bigl(u_0+\delta_2^{-1/3}\varphi_1
		+\Delta(\varphi_1^3)\bigr)
		\right\|_{H^k}\\
		\notag
		&\quad\leq
		\left\|
		\mathcal R_{\delta_2}
		(\mathcal R_{\delta_1}(u_0,0,f_2),0,0)
		-\mathcal R_{\delta_2}
		(u_0+\delta_2^{-1/3}\varphi_1,0,0)
		\right\|_{H^k}\\
		\notag
		&\qquad+
		\left\|
		\mathcal R_{\delta_2}
		(u_0+\delta_2^{-1/3}\varphi_1,0,0)
		-\bigl(u_0+\delta_2^{-1/3}\varphi_1
		+\Delta(\varphi_1^3)\bigr)
		\right\|_{H^k}\\
		&\quad
		<\frac{2\epsilon}{3}.
	\end{align}
		Let us denote the actual state reached at time
	$\delta_1+\delta_2$ by
	\begin{equation*}
		\hat u_0
		:=\mathcal R_{\delta_1+\delta_2}(u_0,0,f)
		=\mathcal R_{\delta_2}
		(\mathcal R_{\delta_1}(u_0,0,f_2),0,0).
	\end{equation*}
	By Lemma~\ref{lem:parabolic_regularization}, applied to the free evolution on
	$[0,\delta_2]$, we have $\hat u_0\in H^{k+2}(\TT)$.
	Define
	\begin{equation*}
		\hat u_1
		:=\hat u_0+\eta-\delta_2^{-1/3}\varphi_1
		\in\hat u_0+\mathscr H_{N-1}.
	\end{equation*}
	Since $u_1=u_0+\eta+\Delta(\varphi_1^3)$, estimate
	\eqref{apprx3} gives
	\begin{equation*}
		\|\hat u_1-u_1\|_{H^k}
		=
		\left\|
		\hat u_0-\bigl(u_0+\delta_2^{-1/3}\varphi_1
		+\Delta(\varphi_1^3)\bigr)
		\right\|_{H^k}
		<\frac{2\epsilon}{3}.
	\end{equation*}
	By the induction hypothesis, applied with initial datum
	$\hat u_0$ and target $\hat u_1$, there exist
	$\delta_3\in(0,\sigma/3)$ and a control
	$f_3\in L^\infty((0,\delta_3);\mathscr H_0)$ such that
	$\mathcal R(\hat u_0,0,f_3)$ is well-defined on
	$[0,\delta_3]$ and
	\begin{equation}\label{apprx4}
		\|\mathcal R_{\delta_3}(\hat u_0,0,f_3)
		-\hat u_1\|_{H^k}<\frac{\epsilon}{3}.
	\end{equation}
		Let us define the piecewise control $\hat f$ by
	\begin{equation*}
		\hat f(t)=
		\begin{cases}
			f(t),&t\in(0,\delta_1+\delta_2),\\
			f_3(t-\delta_1-\delta_2),
			&t\in(\delta_1+\delta_2,\delta_1+\delta_2+\delta_3).
		\end{cases}
	\end{equation*}
	Set $\delta:=\delta_1+\delta_2+\delta_3<\sigma$.
	By the flow property \eqref{flow2}, the solution
	$\mathcal R(u_0,0,\hat f)$ is well-defined on $[0,\delta]$,
	and
	\begin{align*}
		\|\mathcal R_\delta(u_0,0,\hat f)-u_1\|_{H^k}
		&=\|\mathcal R_{\delta_3}(\hat u_0,0,f_3)-u_1\|_{H^k}\\
		&\leq
		\|\mathcal R_{\delta_3}(\hat u_0,0,f_3)-\hat u_1\|_{H^k}
		+\|\hat u_1-u_1\|_{H^k}\\
		&<\frac{\epsilon}{3}+\frac{2\epsilon}{3}
		=\epsilon.
	\end{align*}
	This proves the case $m=1$.
	
	Next, assume that the assertion holds for $m-1$. Let us recall
	\begin{equation*}
		u_1=u_0+\eta+\sum_{i=1}^m\Delta(\varphi_i^3),
		\qquad
		\eta,\varphi_1,\ldots,\varphi_m\in\mathscr H_{N-1}.
	\end{equation*}
	By the induction hypothesis, there exist
	$\delta_1\in(0,\sigma/2)$ and a control
	$f_4\in L^\infty((0,\delta_1);\mathscr H_0)$ such that
	$\mathcal R(u_0,0,f_4)$ is well-defined on $[0,\delta_1]$ and
	\begin{equation*}
		\left\|
		\mathcal R_{\delta_1}(u_0,0,f_4)
		-\left(u_0+\eta+\sum_{i=1}^{m-1}\Delta(\varphi_i^3)\right)
		\right\|_{H^k}<\frac{\epsilon}{2}.
	\end{equation*}
	Let us denote the actual state reached at time $\delta_1$ by
$	\widetilde u_0:=\mathcal R_{\delta_1}(u_0,0,f_4).$
	Since $\mathscr H_0\subset C^\infty(\TT)$ is finite-dimensional,
	we have $f_4\in L^2(0,\delta_1;H^k(\TT))$.
	Thus Lemma~\ref{lem:parabolic_regularization} gives
	$\widetilde u_0\in H^{k+2}(\TT)$.
	Define $\widetilde u_1:=\widetilde u_0+\Delta(\varphi_m^3).$
	Then
	\begin{equation*}
		\|\widetilde u_1-u_1\|_{H^k}
		=
		\left\|
		\widetilde u_0
		-\left(u_0+\eta+\sum_{i=1}^{m-1}\Delta(\varphi_i^3)\right)
		\right\|_{H^k}
		<\frac{\epsilon}{2}.
	\end{equation*}
	Using the result for $m=1$ with initial datum
	$\widetilde u_0$, with $\eta=0$, and with $\varphi_m$ in
	place of $\varphi_1$, we obtain $\delta_2\in(0,\sigma/2)$
	and a control
	$f_5\in L^\infty((0,\delta_2);\mathscr H_0)$ such that
	$\mathcal R(\widetilde u_0,0,f_5)$ is well-defined on
	$[0,\delta_2]$ and
	\begin{equation*}
		\|\mathcal R_{\delta_2}(\widetilde u_0,0,f_5)
		-\widetilde u_1\|_{H^k}<\frac{\epsilon}{2}.
	\end{equation*}
	Let $\delta:=\delta_1+\delta_2<\sigma$, and define
	\begin{equation*}
		\widetilde f(t)=
		\begin{cases}
			f_4(t),&t\in(0,\delta_1),\\
			f_5(t-\delta_1),&t\in(\delta_1,\delta_1+\delta_2).
		\end{cases}
	\end{equation*}
	Thanks to the flow property \eqref{flow2}, the solution
	$\mathcal R(u_0,0,\widetilde f)$ is well-defined on
	$[0,\delta]$ and satisfies
	\begin{align*}
		\|\mathcal R_\delta(u_0,0,\widetilde f)-u_1\|_{H^k}
		&=\|\mathcal R_{\delta_2}(\widetilde u_0,0,f_5)
		-u_1\|_{H^k}\\
		&\leq
		\|\mathcal R_{\delta_2}(\widetilde u_0,0,f_5)
		-\widetilde u_1\|_{H^k}
		+\|\widetilde u_1-u_1\|_{H^k}\\
		&<\epsilon.
	\end{align*}
	This completes the proof for $m$.

\textbf{Step 3: Global approximate controllability at small time in $H^{k}({\TT})$.}  
Let $u_0,u_1\in H^k(\TT)$ and $\epsilon,\sigma>0$.
	By local well-posedness (Proposition~\ref{wellposed2}), we can choose
	$\tau\in(0,\sigma/2)$ sufficiently small such that the
	uncontrolled solution $\mathcal R(u_0,0,0)$ is well-defined
	on $[0,\tau]$. Set
	\begin{equation*}
		\widetilde u_0:=\mathcal R_\tau(u_0,0,0).
	\end{equation*}
	By Lemma~\ref{lem:parabolic_regularization}, we have
	$\widetilde u_0\in H^{k+2}(\TT)$.
	By the saturating property, there exist $N\in\mathbb N$
	and $\varphi_N\in\mathscr H_N$ such that
	\begin{equation}\label{est_apprx11}
		\|u_1-(\widetilde u_0+\varphi_N)\|_{H^k}
		<\frac{\epsilon}{2}.
	\end{equation}
	Using Steps 1 and 2 with initial datum $\widetilde u_0$,
	target $\widetilde u_0+\varphi_N$, and time bound $\sigma/2$,
	we obtain $\delta_1\in(0,\sigma/2)$ and a control
	$f_1\in L^\infty((0,\delta_1);\mathscr H_0)$ such that
	$\mathcal R(\widetilde u_0,0,f_1)$ is well-defined on
	$[0,\delta_1]$ and
	\begin{equation}\label{est_apprx1}
		\|\mathcal R_{\delta_1}(\widetilde u_0,0,f_1)
		-(\widetilde u_0+\varphi_N)\|_{H^k}
		<\frac{\epsilon}{2}.
	\end{equation}
	Let $\delta:=\tau+\delta_1<\sigma$, and define
	\begin{equation*}
		f(t)=
		\begin{cases}
			0,&t\in(0,\tau),\\
			f_1(t-\tau),&t\in(\tau,\tau+\delta_1).
		\end{cases}
	\end{equation*}
	By the flow property \eqref{flow2}, the solution
	$\mathcal R(u_0,0,f)$ is well-defined on $[0,\delta]$, and
	$	\mathcal R_\delta(u_0,0,f)
		=\mathcal R_{\delta_1}(\widetilde u_0,0,f_1).$ 
	Combining \eqref{est_apprx11} and \eqref{est_apprx1},
	we obtain
	\begin{align*}
		\|\mathcal R_\delta(u_0,0,f)-u_1\|_{H^k}
		&\leq
		\|\mathcal R_{\delta_1}(\widetilde u_0,0,f_1)
		-(\widetilde u_0+\varphi_N)\|_{H^k}\\
		&\quad+
		\|\widetilde u_0+\varphi_N-u_1\|_{H^k}\\
		&<\epsilon.
	\end{align*}
	Thus system \eqref{eq_main} is globally
	$H^k$-approximately controllable in arbitrarily small time
	with an $\mathscr H_0$-valued control.
\end{proof}

\begin{proof}[{Proof of Theorem \ref{small_time}}]
Let $T>0$ and $\epsilon>0$. Given $u_0,u_1\in H^k(\TT)$, we aim to
show the existence of a control
$\eta\in L^\infty(0,T;\mathscr H_0)$ such that the solution of
\eqref{eq_main} satisfies
\begin{equation*}
    \|\mathcal R_T(u_0,0,\eta)-u_1\|_{H^k}<\epsilon.
\end{equation*}

{
Using the small-time approximate controllability result in
\Cref{prop_small_time}, we can steer the trajectory starting from
$u_0$ close to $u_1$ in a short time with an
$\mathscr H_0$-valued control. The core strategy is then to employ an
$\mathscr H_0$-valued control to maintain the trajectory near $u_1$
for the remaining time interval.

First, using the time-continuity of the uncontrolled solution starting
from $u_1$, together with \Cref{wellposed2}, \Cref{max}, and the
stability estimate in \Cref{Lipschitz type}, we show that an
uncontrolled solution starting sufficiently close to $u_1$ remains
near $u_1$ for a fixed period of time.

More precisely, by \Cref{wellposed2}, the uncontrolled solution
starting from $u_1$ is well-defined on some interval $[0,\tau_0]$.
We can further choose $\tau\in(0,\tau_0)$ such that
$\|\mathcal R_t(u_1,0,0)-u_1\|_{H^k}<\epsilon/2$ for every
$t\in[0,\tau]$. By \Cref{max}, there exists $r_0>0$ such that every
uncontrolled solution starting from $v\in B_{H^k}(u_1,r_0)$ is also
well-defined on $[0,\tau]$. Let $C_\tau>0$ be the constant in
\Cref{Lipschitz type} on this interval. Choosing
$r\in(0,\min\{r_0,\epsilon/2,\epsilon/(2C_\tau)\})$, the triangle
inequality and \Cref{Lipschitz type} give
\begin{equation}\label{cont_free_estimate}
    \|\mathcal R_t(v,0,0)-u_1\|_{H^k}<\epsilon
    \qquad\forall\,t\in[0,\tau],
\end{equation}
for every $v\in H^k(\TT)$ satisfying
$\|v-u_1\|_{H^k}\leq r$.

By applying the approximate controllability result in small time
(see \Cref{prop_small_time}), we find $0<\delta<T$ and a control
$f\in L^\infty(0,\delta;\mathscr H_0)$ such that
$\|\mathcal R_\delta(u_0,0,f)-u_1\|_{H^k}<r/2$. Denoting
$v_1=\mathcal R_\delta(u_0,0,f)$, we have
$v_1\in B_{H^k}(u_1,r)$. Hence, by
\eqref{cont_free_estimate},
$\|\mathcal R_t(v_1,0,0)-u_1\|_{H^k}<\epsilon$ for every
$t\in[0,\tau]$.

Next, for any $t\in[0,\tau]$, we define the control $\eta_1$ over the
time interval $[0,\delta+t]$ by
\begin{equation*}
\eta_1(s)=
\begin{cases}
    f(s), &s\in[0,\delta],\\
    0, &s\in(\delta,\delta+t].
\end{cases}
\end{equation*}
It is straightforward to check, using the concatenation property
\eqref{flow2}, that
$\|\mathcal R_{\delta+t}(u_0,0,\eta_1)-u_1\|_{H^k}<\epsilon$ for
every $t\in[0,\tau]$. Recall that $\delta\in(0,T)$. Thus, if
$\delta+\tau=T$, the proof is complete. If $\delta+\tau>T$, we choose
$t=T-\delta\in(0,\tau)$ in the preceding argument, which also
completes the proof.

Otherwise, we reapply the small-time approximate controllability
property to steer the trajectory back into the ball
$B_{H^k}(u_1,r)$. More precisely, since $\delta+\tau<T$, we can
choose $\delta'\in(0,T-\delta-\tau)$ and a control
$f_1\in L^\infty(0,\delta';\mathscr H_0)$ such that the state reached
at time $\delta+\tau+\delta'$ belongs to $B_{H^k}(u_1,r)$.

Extending $\eta_1$, we define
$\eta\in L^\infty(0,\delta+\delta'+2\tau;\mathscr H_0)$ by
\begin{equation*}
\eta(s)=
\begin{cases}
    \eta_1(s), &s\in[0,\delta+\tau],\\
    f_1(s-\delta-\tau),
        &s\in(\delta+\tau,\delta+\tau+\delta'],\\
    0,
        &s\in(\delta+\tau+\delta',
              \delta+\tau+\delta'+\tau].
\end{cases}
\end{equation*}
By \eqref{flow2} and \eqref{cont_free_estimate}, the corresponding
solution satisfies
\begin{equation*}
    \|\mathcal R_{\delta+\tau+\delta'+t}
    (u_0,0,\eta)-u_1\|_{H^k}<\epsilon
    \qquad\forall\,t\in[0,\tau].
\end{equation*}

If $\delta+\tau+\delta'+\tau\geq T$, we choose
$t=T-(\delta+\tau+\delta')\in(0,\tau]$ in the preceding estimate,
and the proof is complete. Otherwise, we repeat the same process,
choosing at each step the new controlled interval shorter than the
time remaining before $T$.

Note that the iteration is finite. Indeed, after $n$ iterations, the
accumulated time is greater than $n\tau$. If it remains smaller
than $T$, the procedure can be repeated at most
$\lfloor T/\tau\rfloor+1$ times, and the sequence of accumulated
times cannot have an accumulation point below $T$. Thus, after
finitely many iterations, we reach time $T$ with the desired proximity
to $u_1$.

The resulting control belongs to
$L^\infty(0,T;\mathscr H_0)$, since only finitely many
$\mathscr H_0$-valued controls are concatenated.
Therefore, the system is globally $H^k$-approximately controllable at
time $T$ with an $\mathscr H_0$-valued control.
}
\end{proof}

\section{Global null controllability of the Cahn-Hilliard equation}\label{sec_globalnull}
This section is devoted to the proof of Theorem \ref{global_null} and Theorem \ref{global_nulld2}. The proofs of these results rely on the global approximate controllability established in Theorem \ref{small_time}, together with a local null controllability result. To state this local null controllability result we first consider the following system with localized interior control
\begin{equation}\label{eq_nonlin}
	\begin{cases}
		u_t +\Delta^2 u  +\Delta u=\Delta(u^3)+f\mathbf{1}_{\omega} & \text{in }  (0, T) \times \mathbb{T}^d , \\
		u(0) = u_0 & \text{in } \mathbb{T}^d.
	\end{cases}
\end{equation} 
We have the following local null controllability result.
\begin{proposition}[Local null controllability]\label{res_non}
	Let $T>0$ and $\omega\subset \mathbb{T}^d$ be any measurable set with positive Lebesgue measure. The Cahn-Hilliard equation \eqref{eq_nonlin} is locally null controllable. That is, there exist $R_1, R_2>0$ such that the following holds:
	\begin{enumerate}
		\item\label{point-1} If $d\leq 2,$ for any $u_0\in L^2(\mathbb{T}^d)$ satisfying $\norm{u_0}_{L^2}\leq R_1$, there exists a control $f \in L^\infty\left(0,T;L^\infty(\omega)\right)$ such that system \eqref{eq_nonlin} satisfies $ u(T,\cdot) = 0$ in $\mathbb{T}^d.$
		
		\item\label{point-2} If $d=3,$ for any $u_0\in H^1(\mathbb{T}^d)$ satisfying $\norm{u_0}_{H^1}\leq R_2$, there exists a control $f \in L^\infty\left(0,T;L^\infty(\omega)\right)$ such that system \eqref{eq_nonlin} satisfies $ u(T,\cdot) = 0$ in $\mathbb{T}^d.$
	\end{enumerate}  
\end{proposition}
We begin by establishing the null controllability of the associated linearized system.
\subsection{Controllability of the linearized problem}
Let us first consider linear system with non-homogeneous source term
\begin{equation}\label{eq_source}
	\begin{cases}
		u_t +\Delta^2 u  +\Delta u=f\mathbf{1}_{\omega}+h & \text{in }  (0, T) \times \mathbb{T}^d , \\
		u(0) = u_0 & \text{in } \mathbb{T}^d.
	\end{cases}
\end{equation}
The following well-posedness result holds.
\begin{proposition}\label{well_source}
	Let $T>0.$ For any $u_0\in L^2(\mathbb{T}^d),$ $h\in L^2(0,T;H^{-2}(\mathbb{T}^d))$, $f\in L^2(0,T;L^2(\mathbb{T}^d))$ equation \eqref{eq_source} possesses a unique solution in $X_T^0=C([0,T];L^2(\mathbb{T}^d))\cap L^2(0,T;H^2(\mathbb{T}^d))$. Moreover, we have the following estimate
	\begin{align*}
		\norm{u}_{X_T^0}\leq Ce^{CT}\bigg(\norm{u_0}_{L^2(\mathbb{T}^d)}+\norm{f}_{L^2(0,T;L^2(\mathbb{T}^d))}+\norm{h}_{L^2(0,T;H^{-2}(\mathbb{T}^d))}\bigg).
	\end{align*}
\end{proposition}
\begin{proof}
This is a direct consequence of the linear estimate in Lemma~\ref{est_semigroup} with $k=0$, applied to the source $h+f\mathbf 1_\omega\in L^2(0,T;H^{-2}(\TT))$.
\end{proof}
We prove the null controllability of the linear control system with localized interior control
\begin{equation}\label{eq_lin}
	\begin{cases}
		u_t +\Delta^2 u  +\Delta u=f\mathbf{1}_{\omega} & \text{in }  (0, T) \times \mathbb{T}^d , \\
		u(0) = u_0 & \text{in } \mathbb{T}^d.
	\end{cases}
\end{equation}
Our goal is to establish the observability inequality below, which yields the required null controllability of the linearized system \eqref{eq_lin}. As discussed in the introduction, our argument is inspired by Burq and Moyano \cite{BM23}. In what follows, we fix an exponent $s=s(d)\in(0,1)$ provided by \cite[Theorem 1]{BM23}. This is the exponent used in the condition $\mathcal C_{\mathcal H}^{d-s}(\omega)>m$ in \Cref{global_nulld2}.
\begin{proposition}\label{observ} Let $m>0$, $T>0$ and $u_0\in L^2(\TT).$ Then the following holds:
	\begin{enumerate}\label{point_1}
		\item for any $\omega_1 \subset \TT$ with {$|\omega_1| > m$}, there exists a constant $C>0$ only depending on $d$ and $m$ such that the solution $u$ of \eqref{eq_lin} with $f=0$ satisfies
		\begin{align*}
			\norm{u(T,\cdot)}_{L^2(\TT)}\leq Ce^{CT+\frac{C}{T}}\int_{0}^{T}\|u(t)\|_{L^1(\omega_1)} dt.
		\end{align*}
	\item\label{point_2} for any measurable set $\omega_2 \subset \TT$ with $\mathcal{C}^{d-s}_{\mathcal H}(\omega_2)>m$, there exists a constant $C > 0$, depending only on $d$, $m$, and $s$, such that the solution $u$ of \eqref{eq_lin} with $f=0$ satisfies
	\begin{align*}
		\norm{u(T,\cdot)}_{L^2(\TT)}\leq Ce^{CT+\frac{C}{T}}\int_{0}^{T}\sup_{x\in \omega_2}|u(t,x)| dt.
		\end{align*}
	\end{enumerate}
\end{proposition}

The proof is given in \Cref{proof:obs}.

By combining \Cref{observ} with a duality argument in the spirit of the classical Hilbert Uniqueness  method (see \cite[Theorem 2.48]{coron}), we derive the following null controllability results (see \cite[Section 5]{BM23} for more details). 
\begin{proposition}\label{control}
	Let $T > 0$ and $m>0$. Then for any 
	$u_0 \in L^2(\mathbb{T}^d)$, and $\omega \subset \mathbb{T}^d $ with $|\omega| > m$ there exists 
	$f\in L^{\infty}\left(0,T;L^\infty(\TT)\right)$, supported in $(0,T)\times \omega$, 
	such that the  solution $u$ of \eqref{eq_lin} satisfies $u(T,\cdot) = 0.$
	Moreover,
	there exists a constant $C>0$ only depending on $d$ and $m$ such that, the control
	$f \in L^{\infty}(0,T;L^\infty(\TT))$ satisfies
	\begin{equation}\label{controlcost}
		\|f\|_{L^{\infty}((0,T)\times \TT )} 
		\le C\, e^{C/T} \|u_0\|_{L^2(\TT)}.
	\end{equation}
\end{proposition}
Let us denote $\mathcal M$ by the space of  Borel measure on $\TT,$  where for any $\phi\in C(\TT),$ $\mu(\phi)=\int_{\TT}\phi \, d\mu$ and $\mathcal M$ is endowed with the norm
\begin{align*}
	\norm{\mu}:=\sup_{\phi\in C(\TT), \norm{\phi}_{L^\infty(\TT)}\leq 1}\left|\int_{\TT} \phi \, d\mu\right|.
\end{align*}
\begin{proposition}\label{control1}
	For every $T > 0$, $m>0$, $u_0 \in L^2(\mathbb{T}^d)$, and every
	 closed set $\omega_2\subset \TT$ satisfying $\mathcal{C}^{d-s}_{\mathcal H}(\omega_2)>m$, there exists
	$f\in L^{\infty}(0,T;\mathcal M)$, supported in $(0,T)\times \omega_2$,
	such that the  solution $u$ of \eqref{eq_lin} satisfies $u(T,\cdot) = 0.$
	Moreover,
	there exists a constant $C $ only depending on $d$, $m$ and $s$ such that, the control
	$f \in L^{\infty}(0,T;\mathcal M)$ satisfies
	\begin{equation}\label{controlcost1}
		\|f\|_{L^{\infty}(0,T;\mathcal M)} 
		\le C\, e^{C/T} \|u_0\|_{L^2(\TT )}.
	\end{equation}
\end{proposition}
\subsubsection{Key estimates leading to the observability inequality}
To prove \Cref{observ}, we first recall the Lebeau--Robbiano spectral inequality on $\TT$. We write
$0 \leq \lambda_1 \le \lambda_2 \le \cdots$
for the eigenvalues of \(-\Delta\) in $\TT$, and
\(\{e_j: j \ge 1\}\) for the set of \(L^2(\TT)\)-normalized eigenfunctions, that is, $
-\Delta e_j = \lambda_j e_j  \text{ in } \TT.
$
For \(\lambda > 0\), we define for any $f \in L^2(\TT)$
\begin{equation*}
\mathcal E_\lambda f = \sum_{\lambda_j \le \lambda} (f, e_j)_{L^2}\, e_j,
\qquad
\mathcal  E_\lambda^{\perp} f = \sum_{\lambda_j > \lambda} (f, e_j)_{L^2}\, e_j,
\end{equation*}
where
$(f, e_j)_{L^2} = \int_{\TT} f\, e_j\, dx, \, j \ge 1.$
We have the following spectral estimates due to \cite[Theorem 1]{BM23}.
\begin{proposition}\label{spectral}
For any $m > 0$, there are constants $C, D > 0$ such that, for any $\omega_1, \omega_2 \subset \TT$ satisfying $|\omega_1| > m$ and $\mathcal{C}^{d-s}_{\mathcal H}(\omega_2)>m$, any $\lambda > 0$, and every $\varphi \in L^2(\TT)$, the following estimates hold
	\begin{align}
\label{spectral_est1}	&	\|\mathcal{E}_\lambda\varphi\|_{L^\infty(\TT)} \le C e^{D \sqrt{\lambda}} 
		\|\mathcal{E}_\lambda\varphi \|_{L^1(\omega_1)},\\
	\label{spectral_est2}	&\|\mathcal{E}_\lambda\varphi\|_{L^\infty(\TT)} \le C e^{D \sqrt{\lambda}} 
		\sup_{x\in \omega_2}|\mathcal{E}_\lambda\varphi|.
	 \end{align}
\end{proposition}
Inspired by \cite[Theorem 6]{AEWZ14} and \cite[Theorem 8]{BM23}, and using the above spectral estimates \eqref{spectral_est1}-\eqref{spectral_est2} together with the exponential decay property of the associated semigroup, our next goal is to derive interpolation inequalities for solutions of the Cahn–Hilliard system. It may also be interpreted as a quantitative form of the strong unique continuation property for the Cahn–Hilliard equation.

\begin{proposition}
	Let $\theta\in(0,1)$, $m>0$ and $T>0$. Assume that $|\omega_1|> m$ and $\mathcal{C}^{d-s}_{\mathcal H}(\omega_2)>m$.
	Then there exist constants $N,C>0$  such that for all $0\le r<t\leq T$ we have
	\begin{align}
		\label{lone}	&	\|u(t,\cdot)\|_{L^2(\TT)} 
		\le 
		N e^{\frac{N}{(t-r)}} 
		\|u(t,\cdot)\|_{L^1(\omega_1)}^{\theta} 
		\|u(r,\cdot)\|_{L^2(\TT)}^{1-\theta},\\
		\label{linfty}	&	\|u(t,\cdot)\|_{L^2(\TT)} 
		\le 
		N e^{\frac{N}{(t-r)}} 
		\sup_{x\in \omega_2}|u(t,x)|^{\theta} 
		\|u(r,\cdot)\|_{L^2(\TT)}^{1-\theta},
	\end{align}
	where $u$ is a solution of \eqref{eq_lin} with $f=0$.
\end{proposition}

\begin{proof}
	Let $\theta\in(0,1)$ and $0\leq r<t\leq T$. 
	On $\TT$, the eigenvalues of $-\Delta$ are
	$\lambda_k=|k|^2$, $k\in\mathbb Z^d$. Hence
	$\lambda_k=0$ for $k=0$, while $\lambda_k\geq1$ for
	$k\neq0$. Consequently,
	$
	\lambda_k^2-\lambda_k=\lambda_k(\lambda_k-1)\geq0.
	$
	Therefore,
	$
	e^{-(\lambda_k^2-\lambda_k)(t-r)}\leq1,
	$
	and Parseval's identity yields
	 the solution $u$ of \eqref{eq_lin} with $f=0$
	satisfies
	\begin{equation}\label{eq_dis1}
		\|u(t,\cdot)\|_{L^2(\TT)}
		\leq \|u(r,\cdot)\|_{L^2(\TT)},
		\qquad 0\leq r<t.
	\end{equation}
		If $\|u(r,\cdot)\|_{L^2(\TT)}=0$, both estimates follow
	immediately. We therefore assume that
	$\|u(r,\cdot)\|_{L^2(\TT)}>0$.
	We first prove the estimates when $0<t-r\leq1$.
	Let $\lambda\geq2$. It can be checked that
	\begin{equation}\label{dis2}
		\|\mathcal{E}_\lambda^{\perp}u(t,\cdot)\|_{L^2(\TT)}
		\leq e^{-\frac{\lambda^2}{2}(t-r)}
		\|u(r,\cdot)\|_{L^2(\TT)}.
	\end{equation}
	Indeed, the solution $u$ of \eqref{eq_lin} with $f=0$
	can be written as
	$\displaystyle u=\sum_k u_k e^{-(\lambda_k^2-\lambda_k)t}e_k$.
	Whence,
	\begin{align*}
		\mathcal{E}_\lambda^{\perp}u(t,x)
		&=\sum_{\lambda_k>\lambda}
		u_k e^{-(\lambda_k^2-\lambda_k)t}e_k
		=\sum_{\lambda_k>\lambda}
		e^{-(\lambda_k^2-\lambda_k)(t-r)}
		\left(u_k e^{-(\lambda_k^2-\lambda_k)r}e_k\right).
	\end{align*}
	Since $\lambda_k>\lambda\geq2$, we have
	$-(\lambda_k^2-\lambda_k)\leq-\lambda^2/2$, which gives
	\eqref{dis2}. Next, it follows from \Cref{spectral} that
	\begin{align*}
		\|u(t,\cdot)\|_{L^2(\TT)}
		&\leq
		\|\mathcal{E}_\lambda u(t,\cdot)\|_{L^2(\TT)}
		+\|\mathcal{E}_\lambda^{\perp}u(t,\cdot)\|_{L^2(\TT)}\\
		&\leq
		Ce^{D\sqrt{\lambda}}
		\|\mathcal{E}_\lambda u(t,\cdot)\|_{L^1(\omega_1)}
		+e^{-\frac{\lambda^2}{2}(t-r)}
		\|u(r,\cdot)\|_{L^2(\TT)}\\
		&\leq
		Ce^{D\sqrt{\lambda}}
		\Big(
		\|u(t,\cdot)\|_{L^1(\omega_1)}
		+\|\mathcal{E}_\lambda^{\perp}u(t,\cdot)\|_{L^2(\TT)}
		\Big)
		+e^{-\frac{\lambda^2}{2}(t-r)}
		\|u(r,\cdot)\|_{L^2(\TT)}\\
		&\leq
		Ne^{N\sqrt{\lambda}}
		\Big(
		\|u(t,\cdot)\|_{L^1(\omega_1)}
		+e^{-\frac{\lambda^2}{2}(t-r)}
		\|u(r,\cdot)\|_{L^2(\TT)}
		\Big),
	\end{align*}
	for some $N>1$. Consequently, since $\lambda\geq2$,
	\begin{equation}\label{est71}
		\|u(t,\cdot)\|_{L^2(\TT)}
		\leq Ne^{N\sqrt{\lambda}}
		\Big(
		\|u(t,\cdot)\|_{L^1(\omega_1)}
		+e^{-\lambda(t-r)}
		\|u(r,\cdot)\|_{L^2(\TT)}
		\Big).
	\end{equation}
	Applying Young's inequality, we deduce
	\begin{equation*}
		N\sqrt{\lambda}
		\leq (1-\theta)\lambda(t-r)
		+\frac{N^2}{4(1-\theta)(t-r)}.
	\end{equation*}
	Thus, it follows from \eqref{est71} that, for every
	$\lambda\geq2$,
	\begin{align}\label{est81}
		\|u(t,\cdot)\|_{L^2(\TT)}
		&\leq
		Ne^{\frac{N^2}{4(1-\theta)(t-r)}}
		\Big(
		e^{(1-\theta)\lambda(t-r)}
		\|u(t,\cdot)\|_{L^1(\omega_1)}
		+e^{-\theta\lambda(t-r)}
		\|u(r,\cdot)\|_{L^2(\TT)}
		\Big).
	\end{align}
	If $\|u(t,\cdot)\|_{L^1(\omega_1)}=0$, letting
	$\lambda\to+\infty$ in \eqref{est81} gives
	$\|u(t,\cdot)\|_{L^2(\TT)}=0$, and \eqref{lone} follows.
	Otherwise, we distinguish two cases.
	If $0<\|u(t,\cdot)\|_{L^1(\omega_1)}
		\leq e^{-2(t-r)}\|u(r,\cdot)\|_{L^2(\TT)},$
	we choose
	\begin{equation*}
		\lambda
		=\frac{1}{t-r}
		\log\left(
		\frac{\|u(r,\cdot)\|_{L^2(\TT)}}
		{\|u(t,\cdot)\|_{L^1(\omega_1)}}
		\right)\geq2.
	\end{equation*}
	For this choice, the two terms in parentheses in
	\eqref{est81} coincide, and we obtain
	\begin{equation*}
		\|u(t,\cdot)\|_{L^2(\TT)}
		\leq
		2Ne^{\frac{N^2}{4(1-\theta)(t-r)}}
		\|u(t,\cdot)\|_{L^1(\omega_1)}^\theta
		\|u(r,\cdot)\|_{L^2(\TT)}^{1-\theta}.
	\end{equation*}
	On the other hand, if
	$	\|u(t,\cdot)\|_{L^1(\omega_1)}
		>e^{-2(t-r)}\|u(r,\cdot)\|_{L^2(\TT)},$
	using \eqref{eq_dis1} and $t-r\leq1$, we obtain
	\begin{align*}
		\|u(t,\cdot)\|_{L^2(\TT)}
		\leq \|u(r,\cdot)\|_{L^2(\TT)}
	&	\leq
		e^{2\theta(t-r)}
		\|u(t,\cdot)\|_{L^1(\omega_1)}^\theta
		\|u(r,\cdot)\|_{L^2(\TT)}^{1-\theta}\\
		&\leq
		e^{2\theta}
		\|u(t,\cdot)\|_{L^1(\omega_1)}^\theta
		\|u(r,\cdot)\|_{L^2(\TT)}^{1-\theta}.
	\end{align*}
	Combining these two cases and enlarging $N$, with a
	possible dependence on $\theta$, proves \eqref{lone}
	when $0<t-r\leq1$.
	
Finally, suppose that $t-r>1$. Applying the estimates already
proved for time intervals of length at most one on the interval
$[t-1,t]$, we obtain
\begin{align*}
	\|u(t,\cdot)\|_{L^2(\TT)}
	&\leq N_0e^{N_0}
	\|u(t,\cdot)\|_{L^1(\omega_1)}^\theta
	\|u(t-1,\cdot)\|_{L^2(\TT)}^{1-\theta}\\
	&\leq N_0e^{N_0}
	\|u(t,\cdot)\|_{L^1(\omega_1)}^\theta
	\|u(r,\cdot)\|_{L^2(\TT)}^{1-\theta},
\end{align*}
where we used \eqref{eq_dis1}. Since
$e^{N_1/(t-r)}\geq1$, after increasing the constant $N$,
we obtain \eqref{lone}. 
	
	In order to prove \eqref{linfty}, we use a similar strategy.
	Choose an integer $l>d/2$. Using estimate
	\eqref{spectral_est2} and the embedding
	$H^l(\TT)\hookrightarrow L^\infty(\TT)$, we have
	\begin{align}\label{linfty_est1}
		\notag
		\|u(t,\cdot)\|_{L^2(\TT)}
		&\leq
		\|\mathcal{E}_\lambda u(t,\cdot)\|_{L^2(\TT)}
		+\|\mathcal{E}_\lambda^{\perp}u(t,\cdot)\|_{L^2(\TT)}\\
		\notag
		&\leq
		Ce^{D\sqrt{\lambda}}
		\sup_{x\in\omega_2}|\mathcal{E}_\lambda u(t,x)|
		+\|\mathcal{E}_\lambda^{\perp}u(t,\cdot)\|_{L^2(\TT)}\\
		\notag
		&\leq
		Ce^{D\sqrt{\lambda}}
		\Big(
		\sup_{x\in\omega_2}|u(t,x)|
		+\|\mathcal{E}_\lambda^{\perp}u(t,\cdot)\|_{H^l(\TT)}
		\Big)
		+\|\mathcal{E}_\lambda^{\perp}u(t,\cdot)\|_{L^2(\TT)}\\
		&\leq
		Ne^{N\sqrt{\lambda}}
		\Big(
		\sup_{x\in\omega_2}|u(t,x)|
		+\|\mathcal{E}_\lambda^{\perp}u(t,\cdot)\|_{H^l(\TT)}
		\Big),
	\end{align}
	where $N$ is some positive constant.
	
	\noindent
Observe that, the second term in the above inequality can be estimated
	as follows:
	\begin{align}\label{linfty_est2}
		\notag
		\|\mathcal{E}_\lambda^{\perp}u(t,\cdot)\|_{H^l(\TT)}^2
		&\leq
		C\sum_{\lambda_k>\lambda}
		\left(
		e^{-2(\lambda_k^2-\lambda_k)(t-r)}
		(1+\lambda_k)^l
		\right)
		e^{-2(\lambda_k^2-\lambda_k)r}|u_k|^2\\
		\notag
		&\leq
		C\left(
		\sup_{\lambda_k>\lambda}
		e^{-2\delta(\lambda_k^2-\lambda_k)(t-r)}
		(1+\lambda_k)^l
		\right) 
		\sum_{\lambda_k>\lambda}
		e^{-2(1-\delta)(\lambda_k^2-\lambda_k)(t-r)}
		e^{-2(\lambda_k^2-\lambda_k)r}|u_k|^2\\
		&\leq
		C2^l
		\left(
		\sup_{\lambda_k>\lambda}
		e^{-2\delta\lambda_k(t-r)}\lambda_k^l
		\right)
		e^{-2(1-\delta)\lambda(t-r)}
		\|u(r,\cdot)\|_{L^2(\TT)}^2,
	\end{align}
	for any $\delta\in(0,1)$. Using
	$x^l/l!\leq e^x$ for $x\geq0$, we infer from
	\eqref{linfty_est2} that
	\begin{align}\label{linfty_est3}
		\notag
		\|\mathcal{E}_\lambda^{\perp}u(t,\cdot)\|_{H^l(\TT)}^2
		&\leq
		C l!\left(\frac{1}{\delta}\right)^l
		\left(\frac{1}{t-r}\right)^l
		e^{-2(1-\delta)\lambda(t-r)}
		\|u(r,\cdot)\|_{L^2(\TT)}^2\\
		&\leq
		C_{l,\delta}e^{\frac{1}{t-r}}
		e^{-2(1-\delta)\lambda(t-r)}
		\|u(r,\cdot)\|_{L^2(\TT)}^2,
	\end{align}
	where $C_{l,\delta}>0$ is independent of $\lambda$,
	$r$ and $t$.
	
	Taking square roots in \eqref{linfty_est3}, choosing
	$\delta=1/2$, and combining the resulting estimate with
	\eqref{linfty_est1}, we obtain, after enlarging the constants,
	\begin{equation*}
		\|u(t,\cdot)\|_{L^2(\TT)}
		\leq
		Ne^{\frac{1}{t-r}}e^{N\sqrt{\lambda}}
		\Big(
		\sup_{x\in\omega_2}|u(t,x)|
		+e^{-\frac{\lambda}{2}(t-r)}
		\|u(r,\cdot)\|_{L^2(\TT)}
		\Big).
	\end{equation*}
		Proceeding with a similar argument as in the previous case,
		one can conclude the proof of \eqref{linfty}. We omit the details for the sake of brevity.
			\end{proof}
		
As a consequence of the above interpolation inequalities, we obtain the following result.
\begin{cor}\label{time_estimate}
	Let $m > 0$. Assume that $|\omega_1| > m$ and $\mathcal{C}^{d-s}_{\mathcal H}(\omega_2)>m$.
	Then, for any $D \geq 1$, there exist constants $A, C > 0$ such that for all
	$0 < t_1 < t_2 \le T$ and for the solution  $u$ of \eqref{eq_lin} with $f=0$, the following estimates hold
	\begin{align*}
		e^{-\frac{A}{t_2 - t_1}}\big\| u(t_2,\cdot)  \big\|_{L^2(\TT)}
		-e^{- \frac{D A}{t_2 - t_1}} \big\|  u(t_1,\cdot) \big\|_{L^2(\TT)}
		&\le
		C \int_{t_1}^{t_2}\,
		\big\| u(r,\cdot) \big\|_{L^1(\omega_1)} \, dr.\\
		e^{-\frac{A}{t_2 - t_1}}\big\| u(t_2,\cdot)  \big\|_{L^2(\TT)}
		-e^{- \frac{D A}{t_2 - t_1}} \big\|  u(t_1,\cdot) \big\|_{L^2(\TT)}
		&\le
		C \int_{t_1}^{t_2}\,
		\sup_{x\in \omega_2}|u(r,x)| \, dr.
	\end{align*}
\end{cor}
\begin{proof}
	The proof is obtained by adapting the argument of \cite[Corollary 4.1]{BM23}.
\end{proof}
We are now in a position to prove the observability inequality stated in \Cref{observ}. The main idea follows the classical Lebeau--Robbiano method \cite{LR}, which derives observability inequalities from spectral estimates. However, since our setting differs from the original framework, we must adapt the method to the case of measurable observation regions. The detailed analysis in similar contexts can be found in the proofs of \cite[Theorem 1.3]{GLMO25} and \cite[Theorem 3]{BM23}. For the sake of completeness, we present the  proof here.

\subsubsection{Proof of \Cref{observ}}\label{proof:obs}
\begin{proof}
	Thanks to Corollary \ref{time_estimate}, there exist constants 
	$A > 0$, $C > 0$ such that for all $0 < t_1 < t_2 \le T$ and for all 
	$u_0 \in L^2(\TT)$, the corresponding solution $u$ of \eqref{eq_lin} with $f=0$ satisfies
	\begin{equation}\label{est9}
		e^{-\frac{A}{t_2 - t_1}} \|u(t_2,\cdot)\|_{L^2(\TT)}
		- e^{-\frac{2A}{t_2 - t_1}} \|u(t_1,\cdot)\|_{L^2(\TT)}
		\le C \int_{t_1}^{t_2} 
		\|u(r,\cdot)\|_{L^1(\omega_1)} dr.
	\end{equation}
	Let us define
	\begin{equation*}
		l = \frac{T}{2}, 
		\qquad l_1 = \frac{3T}{4}, 
		\qquad l_{n+1} - l = \frac{(l_1 - l)}{2^{n}}, \quad n \ge 1.
	\end{equation*}
	We apply \eqref{est9} with $t_1 = l_{n+1}$ and $t_2 = l_n$ and obtain
	\begin{equation}\label{obs_est1}
		e^{-\frac{A}{l_n - l_{n+1}}} \|u(l_n,\cdot)\|_{L^2(\TT)}
		- e^{-\frac{2A}{l_n - l_{n+1}}} \|u(l_{n+1},\cdot)\|_{L^2(\TT)}
		\le C \int_{l_{n+1}}^{l_n} 
		\|u(r,\cdot)\|_{L^1(\omega_1)} dr.
	\end{equation}
	Note that
	$\frac{2}{(l_n - l_{n+1})} = \frac{1}{\,l_{n+1} - l_{n+2}\,}.$
	Thus inequality \eqref{obs_est1} reads as
	\begin{equation}\label{est10}
		e^{-\frac{A}{l_n - l_{n+1}}} \|u(l_n,\cdot)\|_{L^2(\TT)}
		- e^{-\frac{A}{l_{n+1} - l_{n+2}}} \|u(l_{n+1},\cdot)\|_{L^2(\TT)}
		\le C \int_{l_{n+1}}^{l_n} 
		\|u(r,\cdot)\|_{L^1(\omega_1)} \, dr.
	\end{equation}
	Summing the telescopic series \eqref{est10} and noting the fact 
	$\displaystyle \lim_{n \to +\infty} e^{-A/(l_{n+1} - l_{n+2})} = 0$, 
	we deduce
	\begin{equation*}
		e^{-\frac{A}{l_1 - l_2}} \|u(l_1,\cdot)\|_{L^2(\TT)} 
		\le C \int_{l}^{l_1} 
		\|u(r,\cdot)\|_{L^1(\omega_1)} \, dr.
	\end{equation*}
	Using the definitions of $l$, $l_1$, and $l_2$, it follows that
	\begin{equation}\label{eq:final}
		\|u(l_1,\cdot)\|_{L^2(\TT)} 
		\le C e^{\frac{C}{T}} 
		\int_0^T 
		\|u(r,\cdot)\|_{L^1(\omega_1)} \, dr.
	\end{equation}
	Combining \eqref{eq_dis1} with \eqref{eq:final}, we can complete the proof of \Cref{observ}--\Cref{point-1}. The proof of \Cref{observ}--\Cref{point-2} can be done in a similar manner.
	\end{proof}
	
\subsection{Source term method with $L^\infty$ control and local null controllability}
In this section, we apply the source term method introduced by Liu, Takahashi, and Tucsnak \cite{Tucsnak_nonlinear} to establish the local null controllability result \Cref{res_non} for the nonlinear system \eqref{eq_nonlin}. The proof relies on the null controllability results \Cref{control}–\Cref{control1} for the associated linearized system, together with suitable estimates of the control cost. By \Cref{control}, we have an $L^\infty$ estimate for the control cost. We then fix $M > 0$ and redefine the control cost as $M e^{M/T}.$ 

Let $q \in (1, \sqrt{2})$ and $p > \dfrac{q^2}{2 - q^2}$. 
We define the weights
\begin{align*}
	\rho_0(t) 
\notag	&:=  \exp\left(-\frac{Mp}{(q - 1)(T - t)}\right),\\
	\rho_{\mathcal S}(t) 
	&:= 
	\exp{\left(-\frac{(1+p)q^2 M}{ (q - 1)(T - t)}\right)}.
\end{align*}
The assumption 
$p > \frac{q^2}{2 - q^2}$ implies 
$2p > (1 + p) q^2$ which further
implies that
\begin{equation*}
	\frac{\rho_0^2}{\rho_{\mathcal S}} \in C([0,T]).
\end{equation*}
We then define the associated spaces for the source term, the state, and the control as follows:
\begin{align*}
	\mathcal{S} 
	&:= \Bigl\{
	S \in L^2\big(0,T; H^{-2}({\TT})\big) 
	\;\Big|\;
	\tfrac{S}{\rho_{\mathcal S}} \in L^2\big(0,T; H^{-2}({\TT})\big)
	\Bigr\}, \\
	\mathcal{Z}
	&:= \Bigl\{
	z \in C([0,T];L^2({\TT}))\cap L^2\big(0,T; H^2({\TT})\big) 
	\;\Big|\;
	\tfrac{z}{\rho_0} \in C([0,T];L^2({\TT}))\cap L^2\big(0,T; H^2({\TT})\big)
	\Bigr\},  \\
	\mathcal{H}
	&:= \Bigl\{
	f \in L^\infty\big((0,T)\times \omega )\big) 
	\;\Big|\;
	\tfrac{f}{\rho_0} \in L^\infty\big(0,T; L^\infty(\omega) \big)
	\Bigr\}. 
\end{align*}
Consider the equation
\begin{equation}\label{eq_lin1}
	\begin{cases}
		u_t +\Delta^2 u  +\Delta u=S+f\mathbf{1}_{\omega} & \text{in }  (0, T) \times \TT , \\
		u(0) = u_0 & \text{in } \TT.
	\end{cases}
\end{equation}
The following result is an adaptation of the work \cite{Tucsnak_nonlinear} in the $L^\infty$ control setting, as done in \cite{LB20}. For the sake of completeness, we include the proof.
\begin{proposition}\label{lin}
	For every $S \in \mathcal{S}$ and $u_0 \in L^2(\TT)$, there exists 
	$f \in \mathcal{H}$ such that the solution $u$ of \eqref{eq_lin1}
	satisfies $u \in \mathcal{Z}$. Furthermore, there exists a constant $C > 0$, 
	independent of $S$ and $u_0$, such that
	\begin{equation}\label{est12}
		\|u / \rho_0\|_{C([0,T]; L^2(\TT))}  +	\|u / \rho_0\|_{L^2(0,T; H^2(\TT))} +
		\|f\|_{\mathcal{H}} 
		\le C \left(
		\|u_0\|_{L^2} 
		+ \|S\|_{\mathcal{S}}
		\right).
	\end{equation}
	In particular, since $\rho_0$ is a continuous function satisfying $\rho_0(T) = 0$, 
	the estimate~\eqref{est12} implies that
	\begin{equation*}
	u(T,\cdot) = 0.
	\end{equation*}
\end{proposition}
\begin{proof}
	For $k \ge 0$, we define $T_k = T(1 - \frac{1}{q^{k}})$. Let us denote $a_0 = u_0$ and, for
	$k \ge 0$, we define $a_{k+1} = \hat u(T_{k+1}^-, \cdot),$ where $\hat u$ is the solution to
\begin{equation*}
	\begin{cases}
	\hat	u_t +\Delta^2 \hat u  +\Delta \hat u=S & \text{in }  (T_k, T_{k+1}) \times \TT , \\
		\hat u(T_k^+) = 0 & \text{in } \TT.
	\end{cases}
\end{equation*}
	From \Cref{well_source}, we have
	\begin{equation}\label{ak}
		\norm{a_{k+1}}_{L^2} \le 
		\|\hat u\|_{C([T_k,T_{k+1}];L^2(\TT)}
		\le C \|S\|_{L^2(T_k,T_{k+1};H^{-2}(\TT))}.
	\end{equation}
	On the other hand, for $k \ge 0$, we also consider the control systems
\begin{equation*}
	\begin{cases}
		\tilde u_t +\Delta^2 \tilde u  +\Delta \tilde u=\mathbf{1}_{\omega} f & \text{in }  (T_k, T_{k+1}) \times \TT , \\
		\tilde u(T_k^+) = a_k & \text{in } \TT.
	\end{cases}
\end{equation*}
	Using \Cref{control}, we get the existence of  control
	$f_k \in L^\infty(T_k,T_{k+1};L^\infty(\TT))$
	such that $\tilde u(T_{k+1}^-,\cdot) = 0$ and, thanks to the control cost estimate \eqref{controlcost} (recalling that we redefine the constant as $ M e^{M/T}$),
	\begin{equation}\label{fk2}
		\|f_k\|_{L^\infty(T_k,T_{k+1};L^\infty(\TT))} 
		\le M e^{M/(T_{k+1}-T_k)} \|a_k\|_{L^2}.
	\end{equation}
	In particular, for $k = 0$, we have
	\begin{equation*}
		\|f_0\|_{L^\infty(T_0,T_1;L^\infty(\TT))} 
		\le M e^{qM/T(q-1)} \|u_0\|_{L^2}.
	\end{equation*}
	And, since $\rho_0$ is decreasing,
	\begin{equation}\label{fzero}
		\left\|\frac{f_0}{\rho_0}\right\|_{L^\infty(T_0,T_1;L^\infty(\TT))}  \le 
		\rho_0^{-1}\left(\frac{T(q-1)}{q}\right) M e^{qM/T(q-1)} \|u_0\|_{L^2}.
	\end{equation}
	For $k \ge 0$, since $\rho_S$ is decreasing, combining \eqref{ak} and \eqref{fk2} yields
	\begin{equation*}
		\|f_{k+1}\|_{L^\infty(T_{k+1},T_{k+2};L^\infty(\TT))}
		\le C M e^{M/(T_{k+2}-T_{k+1})} \rho_{\mathcal S}(T_k)
		\left\|\frac{S}{\rho_{ \mathcal S}}\right\|_{L^2(T_k,T_{k+1};H^{-2}(\TT))}.
	\end{equation*}
	In particular, by using $e^{M/(T_{k+2}-T_{k+1})}\rho_S(T_k) = \rho_0(T_{k+2})$, we get
	\begin{equation}\label{fk1}
		\|f_{k+1}\|_{L^\infty(T_{k+1},T_{k+2};L^\infty(\TT))}
		\le C \rho_0(T_{k+2}) 
		\left\|\frac{S}{\rho_{\mathcal S}}\right\|_{L^2(T_k,T_{k+1};H^{-2}(\TT))}.
	\end{equation}
	Then, from \eqref{fk1}, by using the fact that $\rho_0$ is decreasing,
	\begin{equation}\label{fk}
		\left\|\frac{f_{k+1}}{\rho_0}\right\|_{L^\infty(T_{k+1},T_{k+2};L^\infty(\TT))}  \le 
		C 	\left\|\frac{S}{\rho_{\mathcal S}}\right\|_{L^2(T_k,T_{k+1};H^{-2}(\TT))}.
	\end{equation}
	We can glue the controls $f_k$ together for $k \ge 0$ by defining
	\begin{equation*}
	f := \sum_{k \ge 0} \mathbf{1}_{(T_k,T_{k+1})}f_k.
	\end{equation*}
	We have the estimate from \eqref{fzero} and \eqref{fk}, and using the fact that the essential supremum over disjoint intervals equals the supremum
	of the individual norms, we get
	\begin{align*}
	\notag	\left\|\frac{f}{\rho_0}\right\|^2_{L^\infty(0,T;L^\infty(\TT))}&=\sup_{k\in \mathbb{N}}\norm{\frac{f_{k}}{\rho_0}}^2_{L^\infty(T_{k},T_{k+1};L^\infty(\TT))} \\
	\notag	&
		\le C^2 \rho_0^{-2}(T_1) M^2 e^{2qM/T(q-1)} \|u_0\|^2_{L^2}+\sum_{k\in\mathbb N} \left\|\frac{S}{\rho_{\mathcal S}}\right\|^2_{L^2(T_k,T_{k+1};H^{-2}(\TT))}\\
		&\le C^2 \|S/\rho_{\mathcal S}\|^2_{L^2(0,T;H^{-2}(\TT))}
		+ C^2 \rho_0^{-2}(T_1) M^2 e^{2qM/T(q-1)} \|u_0\|^2_{L^2}.
	\end{align*}
The full state is given by $u=\hat u+\tilde u$ on each interval $(T_k,T_{k+1})$. We then proceed to estimate the state. On every interval $(T_k, T_{k+1})$,
we apply the energy estimate given in \Cref{well_source} and deduce 
	\begin{align*}
		\|\hat u\|_{C([T_k,T_{k+1}];L^2(\TT))}+	\|\hat u\|_{L^2(T_k,T_{k+1};H^2(\TT))} &\le C \|S\|_{L^2(T_k,T_{k+1}; H^{-2}(\TT))},\\
		\|\tilde u\|_{C([T_k,T_{k+1}];L^2(\TT))} +	\|\tilde u\|_{L^2(T_k,T_{k+1};H^2(\TT))} 
		&\le C \|a_k\|_{L^2} 
		+ C \|f_k\|_{L^\infty(T_k,T_{k+1};L^\infty(\TT))}.
	\end{align*}
	
	Set $A_k:=\left\|\frac{S}{\rho_{\mathcal S}}\right\|_{L^2(T_k,T_{k+1};H^{-2}(\TT))}$. The identity $e^{M/(T_{k+2}-T_{k+1})}\rho_{\mathcal S}(T_k)=\rho_0(T_{k+2})$ and the fact that $\rho_0$ is decreasing imply
	\begin{equation*}
	\rho_{\mathcal S}(T_k)\leq \rho_0(T_{k+1}),
	\qquad
	\rho_{\mathcal S}(T_{k-1})\leq \rho_0(T_{k+1}), \quad k\geq1.
	\end{equation*}
	Combining these inequalities with \eqref{ak} and \eqref{fk1}, the estimates above give
	\begin{align*}
	\|\hat u/\rho_0\|_{C([T_k,T_{k+1}];L^2(\TT))}+\|\hat u/\rho_0\|_{L^2(T_k,T_{k+1};H^2(\TT))}
	&\le C A_k, \qquad k\geq0,\\
	\|\tilde u/\rho_0\|_{C([T_k,T_{k+1}];L^2(\TT))}+\|\tilde u/\rho_0\|_{L^2(T_k,T_{k+1};H^2(\TT))}
	&\le C A_{k-1}, \qquad k\geq1.
	\end{align*}
	For $k=0$, using $a_0=u_0$ and \eqref{fzero} instead, we obtain
	\begin{equation*}
	\|\tilde u/\rho_0\|_{C([T_0,T_1];L^2(\TT))}+\|\tilde u/\rho_0\|_{L^2(T_0,T_1;H^2(\TT))}
	\leq C\|u_0\|_{L^2}.
	\end{equation*}
	Since $u=\hat u+\tilde u$ on each interval, taking the supremum over $k$ yields
	\begin{align*}
	\|u/\rho_0\|_{C([0,T];L^2(\TT))}
	&\le C\left(\|u_0\|_{L^2}+\sup_{k\in\mathbb N}A_k\right)
	\le C\left(\|u_0\|_{L^2}+\|S\|_{\mathcal S}\right).
	\end{align*}
	On the other hand, squaring the estimates on each interval and summing over $k$ gives
	\begin{align*}
	\|u/\rho_0\|_{L^2(0,T;H^2(\TT))}^2
	&\le C\left(\|u_0\|_{L^2}^2+\sum_{k\in\mathbb N}A_k^2\right).
	\end{align*}
	Since the intervals $(T_k,T_{k+1})$ are disjoint and cover $(0,T)$, the definition of $A_k$ gives
	$\sum_{k\in\mathbb N}A_k^2=\|S\|_{\mathcal S}^2$. Hence,
	\begin{equation*}
	\|u/\rho_0\|_{L^2(0,T;H^2(\TT))}
	\leq C\left(\|u_0\|_{L^2}+\|S\|_{\mathcal S}\right).
	\end{equation*}
	Therefore, for an appropriate choice of constant $C > 0$, $u$ and $f$ satisfy \eqref{eq_lin1} and \eqref{est12}. 
	This concludes the proof of \Cref{lin}.
	\end{proof}
Let us restate the previous result (\Cref{lin})  in a more regular setting so that we may apply it to the nonlinear case in dimension 
$3$. Let $u_0\in H^1(\TT)$. Thanks to \Cref{control}, we have the existence of a control $f\in L^\infty\left(0,T;L^\infty(\TT)\right)$ supported in $(0,T)\times \omega$, 
such that the  solution $u$ of \eqref{eq_lin} satisfies $u(T,\cdot) = 0$
and the control satisfies the cost estimate \eqref{controlcost}. We can now employ a similar technique as done in \Cref{lin}, to get controllability of the linear problem \eqref{eq_lin1} in a weighted space.
For that we denote 
\begin{align*}&\mathcal Y=\Bigl\{
	z \in C([0,T];H^1({\TT}))\cap L^2\big(0,T; H^3({\TT})\big) 
	\;\Big|\;
	\tfrac{z}{\rho_0} \in C([0,T];H^1({\TT}))\cap L^2\big(0,T; H^3({\TT})\big)
	\Bigr\}.\\
	& \mathcal W=\Bigl\{
	S \in L^2\big(0,T; H^{-1}({\TT})\big) 
	\;\Big|\;
	\tfrac{S}{\rho_{\mathcal S}} \in L^2\big(0,T; H^{-1}({\TT})\big)
	\Bigr\}. \end{align*}
\begin{proposition}\label{lin1}
	For every $S \in \mathcal{W}$ and $u_0 \in H^1(\TT)$, there exists 
	$f \in \mathcal{H}$ such that the solution $u$ of \eqref{eq_lin1}
	satisfies $u \in \mathcal{Y}$. Furthermore, there exists a constant $C > 0$, 
	independent of $S$ and $u_0$, such that
	\begin{equation}\label{est14}
		\|u / \rho_0\|_{C([0,T]; H^1(\TT))}  +	\|u / \rho_0\|_{L^2(0,T; H^3(\TT))}  
		+ \|f\|_{\mathcal{H}} 
		\le C \left(
		\|u_0\|_{H^1} 
		+ \|S\|_{\mathcal{W}}
		\right).
	\end{equation}
	In particular, since $\rho_0$ is a continuous function satisfying $\rho_0(T) = 0$, 
the estimate~\eqref{est14} implies that
\begin{equation*}
	u(T,\cdot) = 0.
\end{equation*}
\end{proposition}
\begin{proof}
The proof is a direct adaptation of \cite[Proposition 2.8]{Tucsnak_nonlinear}, using the regularity estimates in Lemma~\ref{est_semigroup} with $k=1$.
\end{proof}
\subsubsection{Proof of \Cref{res_non}--\Cref{point-1}}
\begin{proof}
Here we restrict ourselves $d\in\{1, 2\}$. To construct a controlled solution of \eqref{eq_nonlin}, we use an iterative argument inspired by Nadler's multivalued contraction principle \cite[Theorem 5]{Nad69}. For $u_0\in L^2(\TT)$ and $S\in\mathcal S$, let
\begin{equation*}
\Gamma_{\mathcal Z}(u_0,S):=\left\{(u,f)\in\mathcal Z\times\mathcal H:\ (u,f)\text{ solves \eqref{eq_lin1} and }u(T)=0\right\}.
\end{equation*}
By \Cref{lin}, this set is nonempty and contains at least one pair satisfying \eqref{est12}. Let $S_1,S_2\in\mathcal S$ and fix $(u_1,f_1)\in\Gamma_{\mathcal Z}(u_0,S_1)$. Applying \Cref{lin} with zero initial datum and source $S_2-S_1$, we obtain $(w,g)\in\Gamma_{\mathcal Z}(0,S_2-S_1)$ satisfying $\|w\|_{\mathcal Z}+\|g\|_{\mathcal H}\leq C\|S_2-S_1\|_{\mathcal S}$. Consequently, $(u_2,f_2):=(u_1+w,f_1+g)$ belongs to $\Gamma_{\mathcal Z}(u_0,S_2)$ and
\begin{equation}\label{eq:compatible_Z}
	\|u_2-u_1\|_{\mathcal Z}
	+\|f_2-f_1\|_{\mathcal H}
	\leq C\|S_2-S_1\|_{\mathcal S}.
\end{equation}
This estimate plays the role of the bounded linear control operator used in classical applications of the source term method in the $L^2$-setting; see, for instance, \cite[Proposition 3.4]{Tucsnak_nonlinear}.

By the Gagliardo--Nirenberg inequality
$\| u \|_{L^\infty} \leq C \| u \|_{L^2}^{1/2} \| u \|_{H^2}^{1/2}$ for any $u\in H^2(\TT)$, we have
\begin{equation}\label{eq:GN_Z}
\|u\|^6_{L^6}\leq C \| u \|_{L^\infty}^4\| u \|_{L^2}^{2}
\le C\,\|u\|_{L^2}^{4}
\|u\|_{H^2}^{2}.
\end{equation}
Let $\mathcal Q(u):=\Delta(u^3)$. For any $z_1,z_2,z_3\in X_T^0$, by H\"older's inequality and \eqref{eq:GN_Z} we have
\begin{equation}\label{eq:trilinear_Z}
\|z_1z_2z_3\|_{L^2(0,T;L^2(\TT))}
\leq \prod_{j=1}^3\|z_j\|_{L^6(0,T;L^6(\TT))}
\leq C\prod_{j=1}^3\|z_j\|_{X_T^0}.
\end{equation}
Since $\Delta:L^2(\TT)\to H^{-2}(\TT)$ is continuous, $\rho_0\leq1$, and $\rho_0^2/\rho_{\mathcal S}$ is bounded, applying \eqref{eq:trilinear_Z} yields
\begin{align}\label{eq:Q_Z}
\|\mathcal Q(u)\|_{\mathcal S}
&=\left\|\frac{\Delta(u^3)}{\rho_{\mathcal S}}\right\|_{L^2(0,T;H^{-2}(\TT))} \leq C\left\|\frac{\rho_0^3}{\rho_{\mathcal S}}
\left(\frac{u}{\rho_0}\right)^3\right\|_{L^2(0,T;L^2(\TT))}
\leq C\|u\|_{\mathcal Z}^3.
\end{align}
Using the identity $a^3-b^3=(a-b)(a^2+ab+b^2)$, we similarly obtain
\begin{equation}
\|\mathcal Q(u)-\mathcal Q(v)\|_{\mathcal S}
 \leq C\left(\|u\|_{\mathcal Z}^2+\|v\|_{\mathcal Z}^2\right)\|u-v\|_{\mathcal Z}.
\label{eq:Q_Z_diff}
\end{equation}

Choose $(u^0,f^0)\in\Gamma_{\mathcal Z}(u_0,0)$ satisfying
\eqref{est12}, and set $u^{-1}=0$. Since
$\mathcal Q(u^{-1})=0$, this initializes the iteration. Assume that
$(u^n,f^n)\in\Gamma_{\mathcal Z}(u_0,\mathcal Q(u^{n-1}))$ has already
been chosen. Applying \eqref{eq:compatible_Z} with
$S_1=\mathcal Q(u^{n-1})$ and $S_2=\mathcal Q(u^n)$, we choose
$(u^{n+1},f^{n+1})\in\Gamma_{\mathcal Z}(u_0,\mathcal Q(u^n))$ such that
\begin{equation}\label{eq:increment_Z}
\|u^{n+1}-u^n\|_{\mathcal Z}
+\|f^{n+1}-f^n\|_{\mathcal H}
\leq C\|\mathcal Q(u^n)-\mathcal Q(u^{n-1})\|_{\mathcal S}.
\end{equation}
By the definition of $\Gamma_{\mathcal Z}$, the pair
$(u^{n+1},f^{n+1})$ satisfies
\begin{equation}\label{eq:iteration_Z}
u^{n+1}_t+\Delta^2u^{n+1}+\Delta u^{n+1}
=\mathcal Q(u^n)+\mathbf 1_\omega f^{n+1},
\quad
u^{n+1}(0)=u_0,\quad u^{n+1}(T)=0.
\end{equation}
Let $r>0$ be sufficiently small. As long as $\|u^{n-1}\|_{\mathcal Z},\|u^n\|_{\mathcal Z}\leq r$, \eqref{eq:compatible_Z} and \eqref{eq:Q_Z_diff} give
\begin{align}\label{eq:contraction_Z}
&\|u^{n+1}-u^n\|_{\mathcal Z}+\|f^{n+1}-f^n\|_{\mathcal H}\leq Cr^2\|u^n-u^{n-1}\|_{\mathcal Z}.
\end{align}

We choose $r$ so that $Cr^2<1/2$. By \eqref{est12}, $\|u^0\|_{\mathcal Z}+\|f^0\|_{\mathcal H}\leq C\|u_0\|_{L^2}$,
while \eqref{eq:increment_Z} and \eqref{eq:Q_Z} yield, (since $u^{-1}=0$ and hence $Q(u^{-1})=0$),
\begin{equation}\label{eq:d1_bound}
	d_1:=\|u^1-u^0\|_{\mathcal Z}+\|f^1-f^0\|_{\mathcal H}\leq C\|u^0\|_{\mathcal Z}^3.
\end{equation}
More generally, for $n\geq 1$ we set $d_n:=\|u^n-u^{n-1}\|_{\mathcal Z}+\|f^n-f^{n-1}\|_{\mathcal H}$.
We now choose $R_1>0$ sufficiently small so that $\|u_0\|_{L^2}\leq R_1$ implies
\begin{equation}\label{eq:small_data}
	\|u^0\|_{\mathcal Z}+\frac{d_1}{1-q}\leq r,
\end{equation}
with $q:=Cr^2<1/2$.	Under \eqref{eq:small_data}, for every $n\geq0$, one can show by induction that 
$	\|u^n\|_{\mathcal Z}\leq r
		\text{ and }
		d_{n+1}\leq q^n d_1.$ 
Hence $\sum_{n\geq1}d_n<\infty$, and $(u^n,f^n)$ is a Cauchy sequence in the Banach space $\mathcal Z\times\mathcal H$. Let $(u,f)$ denote its limit. It remains to verify that $(u,f)$ satisfies \eqref{eq_nonlin} and the control condition $u(T)=0$.
	
The convergence of the nonlinear terms follows directly from \eqref{eq:Q_Z_diff}. Indeed,
\begin{equation*}
\|\mathcal Q(u^n)-\mathcal Q(u)\|_{\mathcal S}
\leq C\left(\|u^n\|_{\mathcal Z}^2+\|u\|_{\mathcal Z}^2\right)
\|u^n-u\|_{\mathcal Z}\longrightarrow0.
\end{equation*}
Together with the convergence of $(u^n,f^n)$ in $\mathcal Z\times\mathcal H$, a standard argument allows us to pass to the limit in the variational formulation of \eqref{eq:iteration_Z}. Hence $(u,f)$ solves \eqref{eq_nonlin}. Moreover, $u^n\to u$ in $C([0,T];L^2(\TT))$, and therefore $u^n(T)=0$ implies $u(T)=0$. This proves \Cref{res_non}--\Cref{point-1}.
\end{proof}
\subsubsection{Proof of \Cref{res_non}--\Cref{point-2}}
\begin{proof}
Let us restrict ourselves $d=3.$ For $u_0\in H^1(\TT)$ and $S\in\mathcal W$, we define analogously
\begin{equation*}
\Gamma_{\mathcal Y}(u_0,S):=\left\{(u,f)\in\mathcal Y\times\mathcal H:\ (u,f)\text{ solves \eqref{eq_lin1} and }u(T)=0\right\}.
\end{equation*}
By \Cref{lin1}, $\Gamma_{\mathcal Y}(u_0,S)$ is nonempty. As in \eqref{eq:compatible_Z}, given $(u_1,f_1)\in\Gamma_{\mathcal Y}(u_0,S_1)$, we apply \Cref{lin1} with zero initial datum and source $S_2-S_1$ to obtain $(u_2,f_2)\in\Gamma_{\mathcal Y}(u_0,S_2)$ such that
\begin{equation}\label{eq:compatible_Y}
\|u_2-u_1\|_{\mathcal Y}+\|f_2-f_1\|_{\mathcal H}
\leq C\|S_2-S_1\|_{\mathcal W}.
\end{equation}

For any $z_1,z_2,z_3\in X_T^1$, the embeddings $H^1(\TT)\hookrightarrow L^6(\TT)$ and $H^2(\TT)\hookrightarrow W^{1,6}(\TT)$ give
\begin{equation}\label{eq:trilinear_Y}
\|z_1z_2z_3\|_{L^2(0,T;H^1(\TT))}
\leq C\sum_{j=1}^3
\|z_j\|_{L^2(0,T;H^3(\TT))}
\prod_{i\neq j}\|z_i\|_{C([0,T];H^1(\TT))}
\leq C\prod_{j=1}^3\|z_j\|_{X_T^1}.
\end{equation}
Since $\Delta:H^1(\TT)\to H^{-1}(\TT)$ is continuous, $\rho_0\leq1$, and $\rho_0^2/\rho_{\mathcal S}$ is bounded, applying \eqref{eq:trilinear_Y} to $u/\rho_0$ gives
\begin{equation}\label{eq:Q_Y}
\|\mathcal Q(u)\|_{\mathcal W}
 \leq C\|u\|_{\mathcal Y}^3,
\end{equation}
and, using $u^3-v^3=(u-v)(u^2+uv+v^2)$ and \eqref{eq:trilinear_Y},
\begin{equation}
\|\mathcal Q(u)-\mathcal Q(v)\|_{\mathcal W}
 \leq C\left(\|u\|_{\mathcal Y}^2+\|v\|_{\mathcal Y}^2\right)\|u-v\|_{\mathcal Y}.
\label{eq:Q_Y_diff}
\end{equation}
We repeat the construction in \eqref{eq:iteration_Z}, replacing $(\mathcal Z,\mathcal S,\Gamma_{\mathcal Z})$ by $(\mathcal Y,\mathcal W,\Gamma_{\mathcal Y})$ and \Cref{lin} by \Cref{lin1}. Choose $(u^0,f^0)\in\Gamma_{\mathcal Y}(u_0,0)$ satisfying \eqref{est14}, and set $u^{-1}=0$. Given $(u^n,f^n)\in\Gamma_{\mathcal Y}(u_0,\mathcal Q(u^{n-1}))$, we apply \eqref{eq:compatible_Y} with $S_1=\mathcal Q(u^{n-1})$ and $S_2=\mathcal Q(u^n)$ to choose $(u^{n+1},f^{n+1})\in\Gamma_{\mathcal Y}(u_0,\mathcal Q(u^n))$.
For $r>0$ sufficiently small, \eqref{eq:compatible_Y} and \eqref{eq:Q_Y_diff} give
\begin{align*}
&\|u^{n+1}-u^n\|_{\mathcal Y}+\|f^{n+1}-f^n\|_{\mathcal H}
\leq Cr^2\|u^n-u^{n-1}\|_{\mathcal Y},
\end{align*}
whenever $\|u^{n-1}\|_{\mathcal Y},\|u^n\|_{\mathcal Y}\leq r$. Choosing $Cr^2<1/2$ and then taking $\|u_0\|_{H^1}\leq R_2$ sufficiently small, \eqref{est14} and the argument following \eqref{eq:contraction_Z} show that the iterates remain in the ball of radius $r$ and converge in $\mathcal Y\times\mathcal H$ to a pair $(u,f)$. By \eqref{eq:Q_Y_diff}, $\mathcal Q(u^n)\to\mathcal Q(u)$ in $\mathcal W$. Passing to the limit as in the proof of \Cref{res_non}--\Cref{point-1}, we conclude that $(u,f)$ solves \eqref{eq_nonlin} and $u(T)=0$. This proves \Cref{res_non}--\Cref{point-2}.
\end{proof}
\subsection{Proof of Theorem \ref{global_null}}
\begin{proof}
{
	Fix $T > 0$ and assume $u_0 \in L^2(\mathbb{T}^d)$. Thanks to \cite[Chapter III, Theorem 4.2]{RT88}, one can choose a sufficiently small $\epsilon \in (0, T)$ such that the solution $u$ to \eqref{eq_main} with $\eta = 0$ on $(0, \epsilon)$ satisfies $u(\epsilon, \cdot) \in H^2(\mathbb{T}^d)$. 
	We divide the time interval $[0, T]$ into three sub-intervals $[0, \epsilon]$, $[\epsilon, \delta]$, and $[\delta, T]$, where $0 < \epsilon < \delta < T$. We will apply global approximate controllability on $(\epsilon, \delta)$ and local null controllability on $(\delta, T)$.
	
	Let $R > 0$ be the radius of the ball of initial data required for the local null controllability result in Proposition~\ref{res_non} on the interval $(\delta, T)$, where the state space $X$ is defined by
	\[
	X = \begin{cases} 
		L^2(\mathbb{T}^d), & d \in \{1, 2\}, \\ 
		H^1(\mathbb{T}^d), & d = 3.
	\end{cases}
	\]
	 Applying Theorem~\ref{small_time} on the interval $[\epsilon, \delta]$ with initial datum $u(\epsilon)$ and target $0$, we can select a control $\eta_1 \in L^{\infty}(\epsilon, \delta; \mathscr{H}_0)$ such that the state $u(\delta)$ satisfies
	\[
	\|u(\delta)\|_{H^k(\mathbb{T}^d)} < R,
	\]
	where $k=1$ for $d=1$ and $k=2$ for $d=2, 3.$
	By the Sobolev embedding $H^k(\mathbb{T}^d) \hookrightarrow X$, this directly implies $\|u(\delta)\|_X < R$.
	Since $u(\delta) \in B_X(0, R)$, we can apply Proposition~\ref{res_non} on $(\delta, T)$ with a localized control $1_\omega \eta_2$ to steer the state to zero, thereby achieving $u(T, \cdot) =0$. This completes the proof.}
\end{proof}
\subsection{Proof of Theorem \ref{global_nulld2}}
This proof follows a similar approach as Theorem \ref{global_null}. Thus we just need to take care of the third step, that is, the corresponding local controllability result:
\begin{proposition}\label{local2}
Let $d\in \{1,2\}$, let $s=s(d)\in(0,1)$ be the exponent given by \Cref{spectral}, and let $T>0$ and $m>0$. Then there exists $R>0$ such that for any $u_0\in L^2(\mathbb{T}^d)$ satisfying $\norm{u_0}_{L^2}\leq R$ and every closed set $\omega\subset \TT$ satisfying $\mathcal{C}^{d-s}_{\mathcal H}(\omega)>m$, there is $f \in L^\infty\left(0,T; H^{-2}(\TT) \right)$, supported in $(0,T)\times\omega$, such that system \eqref{eq_nonlin} satisfies $u(T,\cdot) = 0$ in $\mathbb{T}^d$.
\end{proposition}
\begin{proof}
Thanks to \Cref{control1}, 	we have the existence of a control $f\in L^\infty(0,T;\mathcal M)$ supported in $(0,T)\times \omega$, 
such that the  solution $u$ of \eqref{eq_lin} satisfies $u(T,\cdot) = 0$
and the control satisfies the cost estimate \eqref{controlcost1}. Here we will show that for almost every $t \in (0,T),$ the control belongs to a negative-order Sobolev space in the space variable, more precisely $f\in  L^\infty(0,T;H^{-2}(\TT)).$ 
Note that for $k>\tfrac d2$, the embedding $H^k(\mathbb{T}^d)\hookrightarrow C(\mathbb{T}^d)$ holds. Thus there exists $C>0$ depending only on $d$ and $k$ such that $\|\phi\|_{L^\infty} \le C \|\phi\|_{H^k},$ for all $\phi \in H^k(\TT).$
 Let $\mu \in \mathcal M$ be a finite Borel measure on $\mathbb{T}^d$, and define
\begin{equation*}
	L_\mu(\phi)=\int_{\mathbb{T}^d} \phi\, d\mu , \qquad \phi\in H^k(\TT).
\end{equation*}
A straightforward computation shows that
\begin{align*}
\sup_{\phi\in H^k(\TT) \atop \norm{\phi}_{H^k}\leq 1}\left|\int_{\TT} \phi d\mu\right|\leq C \sup_{\phi\in C(\TT) \atop \norm{\phi}_{L^\infty}\leq 1}\left|\int_{\TT} \phi d\mu\right|=C\norm{\mu}.
\end{align*}
Thus $L_\mu$ defines a continuous functional on $H^k(\mathbb{T}^d)$, and
\begin{equation}\label{norm}
	\|L_\mu\|_{H^{-k}(\TT)} \le C \|\mu\|.
\end{equation} 
By definition, $L_{\mu}$ is indeed the restriction of $\mu$ in $H^k(\TT)$. {Thus every finite Borel measure $\mu$ defines a continuous functional on $H^k(\mathbb{T}^d)$. Hence, the embedding
$$\mathcal M \hookrightarrow H^{-k}(\mathbb{T}^d)$$
is continuous (for $k > d/2$).
Therefore, any finite Borel measure $\mu \in \mathcal M$ can be viewed as a member of
$ H^{-k}(\mathbb{T}^d).$ }
Consequently, in our setting, when $d=1$ or $2,$ we have $f\in L^{\infty}(0,T;H^{-2}(\TT))$ and using \eqref{norm} and \eqref{controlcost1} we have 
\begin{equation*}
	\norm{f}_{L^{\infty}(0,T;H^{-2}(\TT))}\leq Ce^{\frac{C}{T}}\norm{u_0}_{L^2}. 
\end{equation*}
The proof of the local control result can be done  in a similar manner as \Cref{res_non}--\Cref{point-1}.
\end{proof}

\section{Further remarks and open problems}

The results obtained in this paper show that the Cahn--Hilliard dynamics can be
controlled globally by combining two different ingredients: a finite-dimensional
control localized in Fourier space, used to obtain global approximate
controllability, and a control localized in the physical domain, used to reach
the origin exactly once the trajectory is sufficiently close to it. This
combination establishes, to the best of our knowledge, the first global
controllability results for the Cahn--Hilliard equation in several space
dimensions, and opens several natural questions. We briefly discuss some of them
below.

\subsection{On the spatial localization of the approximate control}

A first natural question concerns the spatial localization of the
finite-dimensional control used in Theorem~\ref{small_time}. In our global
approximate controllability result, the control takes values in the space
\begin{equation*}
\mathscr H_0
=
\spn\{1,\sin(x\cdot k),\cos(x\cdot k)\}_{k\in\mathcal K},
\end{equation*}
and therefore acts on the whole torus. Thus, the global step of our construction
is localized in frequency, but not in the physical domain. By contrast, the
local null controllability step is performed by a control supported in a
prescribed set $\omega\subset\TT$.

It would be natural to ask whether the approximate controllability step can also
be achieved with controls that are localized in space. More precisely, one may
ask whether the Fourier controls in $\mathscr H_0$ can be replaced by a finite
family of localized profiles, for instance controls of the form
\begin{equation*}
\eta(t,x)
=
\sum_{j=1}^M a_j(t)\zeta_j(t,x),
\qquad
\operatorname{supp}\zeta_j(t,\cdot)\subset\omega.
\end{equation*}
Here the profiles $\zeta_j$ should be fixed in advance, or at least depend only
on the control region $\omega$, while the scalar functions $a_j$ would play the
role of the actual control parameters. Thus, the control would remain
finite-dimensional in the sense that only finitely many time-dependent
coefficients are chosen, but it would also be localized in the physical domain.

Related localization ideas have recently been developed for fluid models,
notably for the two-dimensional Boussinesq system in \cite{NR25b}. There, one
uses transported Fourier modes and exploits in an essential way the transport
and coupling structure of the equation in order to produce finitely
decomposable controls supported in a prescribed physical region. For the
Cahn--Hilliard equation, however, it is not clear how to reproduce such a
construction. The proof of Theorem~\ref{small_time} relies on an algebraic
saturation argument based on the interaction of Fourier modes through the cubic
term $\Delta(u^3)$. Replacing these modes by spatially localized profiles would
destroy the algebraic structure that makes the saturation argument work.

This suggests a genuine open problem: to determine whether the localization in
frequency used in the global approximate step can be made compatible with
localization in physical space. Such a result would bring the global approximate
controllability result closer to the localized controls used in the null
controllability part of the paper, but it seems to require new ideas beyond the
present arguments.

\subsection{Approximate tracking controllability}

The global approximate controllability result proved in this paper is a
controllability result at a prescribed time. Given an initial state $u_0$, a
target state $u_1$, and $\varepsilon>0$, the control is chosen so that the
corresponding solution satisfies
\begin{equation*}
\|u(T)-u_1\|_{H^k(\TT)}<\varepsilon .
\end{equation*}
A stronger question is whether one can control not only the endpoint of the
trajectory, but the trajectory itself. More precisely, given a reference path
$\psi:[0,T]\to H^k(\TT)$, one may ask whether there exists a control $\eta$ such
that the associated solution satisfies
\begin{equation*}
\sup_{t\in[0,T]}\|u(t)-\psi(t)\|_{H^k(\TT)}<\varepsilon .
\end{equation*}
This property is usually referred to as approximate tracking controllability; see,
for instance, \cite{ZZ24,BPZ25,GKS21} and the references therein.

This question has recently appeared in the context of incompressible fluid
equations. In \cite{KNRT25}, approximate tracking controllability is established
for the two-dimensional Euler equations driven by finitely many scalar controls.
This result is particularly close in spirit to the approach of the present paper:
the controls are localized in frequency, the relevant finite-dimensional spaces
are built from Fourier modes, and the proof relies on a saturation argument in
the sense of Agrachev--Sarychev.

In \cite{KNRT25}, the objective is not merely to steer the solution close to a
prescribed state at time $T$, but to keep the whole state trajectory close to a
prescribed space-time function. This is a key ingredient in their construction
of relaxation-enhancing velocity fields as trajectories of a controlled fluid
system. A useful feature of that approach is the use of a weaker trajectory norm.
Namely, for a Banach space $X$, one considers
\begin{equation*}
\vertiii{v}_{T,X}
=
\sup_{t\in[0,T]}
\left\|
\int_0^t v(s)\,ds
\right\|_X .
\end{equation*}
Thus, instead of requiring pointwise-in-time closeness in a strong norm, one
controls the accumulated effect of the difference between two trajectories. Such
a formulation is well adapted to problems where fast oscillations or large
controls on short time intervals naturally appear.

For the Cahn--Hilliard equation, an approximate tracking result would complement
Theorem~\ref{small_time} by controlling the evolution on the whole time interval. In the
notation of the present paper, one could ask whether, for a suitable class of
reference paths $\psi$, it is possible to find an $\mathscr H_0$-valued control
$\eta$ such that
\begin{equation*}
\vertiii{u-\psi}_{T,H^k(\TT)}<\varepsilon.
\end{equation*}
This would correspond to controlling the evolution of the phase-separation
process throughout the whole time interval, rather than only prescribing its
final configuration.

The result of \cite{KNRT25} suggests that a tracking version of our global
approximate controllability result may be accessible. To achieve this, one would
need to refine the present strategy in order to control the accumulated effect of
the intermediate states, rather than only its final outcome.

This leads to a natural open problem: to develop a tracking version of the
saturation method for the Cahn--Hilliard equation. Such a result would provide a
finer form of controllability for phase separation dynamics. In view of the
frequency-localized saturation strategy used here, this seems to be a promising
direction for future work.

\subsection{Stochastic perturbations}

Another possible direction concerns stochastic versions of the Cahn--Hilliard
equation with finite-dimensional noise. Stochastic Cahn--Hilliard equations have
been studied in several works; see, notably, \cite{DPD96}. In the present setting,
a natural stochastic counterpart of the controlled equation would be
\begin{equation}\label{eq:stoch-ch}
\d u+\big(\Delta^2 u-\Delta(u^3-u)\big)\d t
=
\sum_{j=1}^M \eta_j\,\d W_j(t),
\end{equation}
where $\eta_1,\ldots,\eta_M$ are the same Fourier-mode directions used to define
the finite-dimensional control space $\mathscr H_0$, and $W_1,\ldots,W_M$ are
independent Brownian motions.

Thus, \eqref{eq:stoch-ch} amounts to replacing the deterministic
finite-dimensional control by random forcing acting on the same modes. The
question is then whether the nonlinear dynamics can propagate the randomness
injected in these directions to the whole state space. In this sense, the
saturation argument developed in this paper is closely related to the program of
\cite{GHHM18}, where scaling and saturation ideas are used to connect approximate
controllability properties of deterministic control systems with qualitative
properties of the associated random dynamics.

For \eqref{eq:stoch-ch}, the global approximate controllability result obtained
in this paper suggests a possible route toward qualitative properties of the
corresponding stochastic system, such as irreducibility and full support of
transition probabilities or invariant measures, provided the noise acts on a
sufficiently rich finite-dimensional space of Fourier modes. Of course, such
conclusions would not follow from controllability alone. One would still need to
establish the well-posedness of the stochastic equation, suitable Feller
properties, a priori estimates, and, for stronger conclusions such as uniqueness
of invariant measures or positivity of finite-dimensional densities, additional
smoothing or Malliavin-type nondegeneracy properties.

Thus, the present result may be viewed as a deterministic component of a broader
stochastic problem: understanding how finite-dimensional random perturbations are
transmitted by the Cahn--Hilliard nonlinearity and what probabilistic properties
they generate for the corresponding random dynamics.

\subsection{The Cahn--Hilliard--Navier--Stokes system}

Another natural direction is to study whether the ideas developed in this paper
can be extended to coupled phase-field fluid models. Such systems arise in the
description of multi-phase flows, where the phase variable interacts with the
motion of an incompressible fluid. We refer, for instance, to the
Cahn--Hilliard--Navier--Stokes model studied in \cite{Boyer99}.

As a reference point, consider the coupled system
\begin{equation}\label{eq:chns}
\begin{cases}
\partial_t v+(v\cdot\nabla)v-\nu\Delta v+\nabla p
=
\mu\nabla\varphi,\\
\operatorname{div}v=0,\\
\partial_t\varphi+v\cdot\nabla\varphi
=
\Delta\mu,\\
\mu=-\Delta\varphi+\varphi^3-\varphi,
\end{cases}
\end{equation}
where $v$ is the incompressible velocity field, $\varphi$ is the order parameter,
and $\mu$ is the chemical potential.

From the control viewpoint, several different problems may be considered,
depending on where the external action is applied. We may act on the phase
equation, on one or several components of the fluid equation, or on both
equations. Each choice leads to a different way of using the coupling between
the phase variable and the velocity field.

The coupled system \eqref{eq:chns} contains nonlinear interactions that are
absent in the Cahn--Hilliard equation considered in this paper. The phase
equation retains the Cahn--Hilliard structure, but it is now affected by the
transport term $v\cdot\nabla\varphi$. At the same time, the fluid equation is
coupled back to the phase variable through the force $\mu\nabla\varphi$. Thus,
one may ask whether the asymptotic limits and saturation ideas used in this
paper can be combined with these transport and coupling effects in order to
obtain global approximate controllability results for the
Cahn--Hilliard--Navier--Stokes system.

This question is currently under investigation and is the objective of a
forthcoming work.

\appendix

\section{Proof of well-posedness results}\label{app}
Let us recall the linearized Cahn-Hilliard system
\begin{equation}\label{abs2}
	\begin{cases}
	u_t=\mathcal A u+f & \text{ in } (0,T),\\
	u(0)=u_0,
	\end{cases}
\end{equation}
where $\mathcal A$ is defined in \eqref{abstract}. Recall the space $X_T^k$ defined in \eqref{XT}.
We have the following result.
\begin{lemma}\label{est_semigroup}
Let $k\in\mathbb N$. For any $T > 0$, $u_0 \in H^k(\TT)$ and
	$f \in L^2(0, T; H^{k-2}(\TT))$, there exists a constant $C>0$ independent of $T$, $u_0$ and $f$ such that
	\begin{itemize}
		\item $\|\mathcal{S}(\cdot) u_0\|_{X_T^k} \le C e^{CT} \|u_0\|_{H^k}$, 
		\item 
		$\Big\| \int_0^{\cdot} \mathcal{S}(\cdot - s) f(s) \, ds \Big\|_{X_T^k} 
		\le C e^{CT} \|f\|_{L^2(0,T;H^{k-2}(\TT))}$.
	\end{itemize}
	\end{lemma}
\begin{proof}
	It is easy to check that $(\mathcal A, D(\mathcal A))$ is quasi-dissipative and maximal operator in $L^2(\TT)$ and thus thanks to \cite[Theorem 12.22]{RR06}, we have that $(\mathcal A, D(\mathcal A))$ generates strongly continuous semigroup $\{\mathcal S(t)\}_{t\geq 0} $ on $L^2(\TT)$. Consequently, for every $u_0\in L^2(\TT)$, $\mathcal S(\cdot)u_0 \in C([0, T]; L^2(\TT))$  is the unique solution of the following system:
	\begin{equation}\label{eq_semigroup}
		\begin{cases}
			u_t =\mathcal A u  & \text{in }   (0, T) \times \TT, \\
			u(0) = u_0 & \text{in } \TT.
		\end{cases}
	\end{equation}
	Multiplying equation \eqref{eq_semigroup} by $u$ and $ (-\Delta)^k u$ successively, integrating over $\TT$, performing integration by parts and then adding the two resulting identities yields the following identity
	\begin{equation}\label{eq_identity}
	\begin{split}
		\frac{1}{2} &\frac{d}{dt} \left( \|u\|_{L^2}^2 + \|(-\Delta)^{\frac{k}{2}} u\|_{L^2}^2 \right)
		\\
		&+ \|(-\Delta)^{\frac{k+2}{2}} u\|_{L^2}^2+\norm{\Delta u}_{L^2}^2+\int_{\TT}\Delta u u+\int_{\TT}\Delta u (-\Delta)^{{k}}u = 0.
		\end{split}
	\end{equation}
	Let us apply Young's inequality to the last two terms of the above identity and for sufficiently small $\epsilon>0$, we obtain the following:  \begin{align}\label{eq_youngs}
		\notag	&\int_{\TT}\Delta u u \leq \epsilon \norm{\Delta u}^2_{L^2}+ C_\epsilon\norm{ u}^2_{L^2}\\
		&\int_{\TT}\Delta u (-\Delta)^{{k}}u \leq \epsilon \|(-\Delta)^{\frac{k+2}{2}} u\|_{L^2}^2+ C_\epsilon \|(-\Delta)^{\frac{k}{2}} u\|_{L^2}^2.
	\end{align}
	Combining \eqref{eq_youngs} with \eqref{eq_identity}, application of Gr\"onwall’s inequality yields that
	\begin{equation*}
	\|u\|_{X_T^k} \le C e^{CT} \|u_0\|_{H^k}.
	\end{equation*}
	For the second estimate, we note that 
	\begin{equation*}
	\int_0^{\cdot}\mathcal  S(\cdot - s) f(s) \, ds
	\end{equation*}
	is the unique solution of
	\begin{equation}\label{eq_nonhomogeneous_semigroup}
		\begin{cases}
			u_t = \mathcal A u + f & \text{ in } (0, T) \times \TT, \\
			u(0) = 0 & \text{ in } \TT.
		\end{cases}
	\end{equation}
Multiplying equation \eqref{eq_nonhomogeneous_semigroup} by $u + (-\Delta)^k u$, integrating over $\TT$, performing integration by parts, we obtain for any $\epsilon > 0$
	\begin{align*}
		\frac{1}{2} \frac{d}{dt} &\left( \|u\|_{L^2}^2 + \|(-\Delta)^\frac{k}{2} u\|_{L^2}^2 \right)
		+ \|(-\Delta)^{\frac{k+2}{2}} u\|_{L^2}^2+\norm{\Delta u}_{L^2}^2\\
		&\leq \epsilon \left( \|\Delta u\|_{L^2}^2 + \|(-\Delta)^{\frac{k+2}{2}} u\|_{L^2}^2 \right)
		+ C_\epsilon \bigg( \|u\|_{L^2}^2 + \|(-\Delta)^\frac{k}{2} u\|_{L^2}^2+\|f\|_{ {H^{k-2}}}^2 \bigg).
	\end{align*}
	Choosing sufficiently small $\epsilon>0$, we obtain
	\begin{align*}
	\frac{d}{dt} \left( \|u\|_{L^2}^2 + \|(-\Delta)^\frac{k}{2} u\|_{L^2}^2 \right)
	&+ \|(-\Delta)^{\frac{k+2}{2}} u\|_{L^2}^2+ \|(-\Delta)^\frac{k}{2} u\|_{L^2}^2\\
&	\leq C \bigg( \|u\|_{L^2}^2 +\|(-\Delta)^\frac{k}{2} u\|_{L^2}^2+ \|f\|_{ {H^{k-2}}}^2\bigg).
	\end{align*}
	Applying Gr\"onwall's inequality, we deduce that
	$
	\|u\|_{X_T^k} \le C e^{CT} \left( \|f\|_{L^2(0, T; H^{k-2}(\TT))}\right).
	$
\end{proof}
From Lemma \ref{est_semigroup}, the following results follows immediately
\begin{proposition}
Let $k\in\mathbb N$ and $T > 0$. For any $u_0 \in H^k(\TT)$ and
	$f \in L^2(0, T; H^{k-2}(\TT))$, equation \eqref{abs2} has a unique solution in $X_T^k.$
	\end{proposition}
\subsection{Proof of \Cref{wellposed2}}\label{pr_of_1}
\begin{proof}
	To prove this result, we first establish the local existence in time and uniqueness of the solution using a fixed-point argument. We will show that there exists $\theta>0$ such that equation \eqref{eq_extension} has a unique solution $u\in X_\theta^k$. Let $\tau >0$, and $v \in X_{\tau}^k$. For any $t \in [0, \tau]$, we set
	\begin{equation*}
		\Phi_{u_0}(v)(t) = \mathcal S(t) u_0 + \int_0^t \mathcal S(t - \sigma) F(v)(\sigma) \, d\sigma,
	\end{equation*}
	where
	\begin{equation*}
		F(v) = -\Delta^2 \varphi - \Delta \varphi +\Delta(v + \varphi)^3  + \eta.
	\end{equation*}
	Since $H^{k}(\TT)$ is a Banach algebra when $k>d/2$, we estimate the nonlinear term of $F(v)$ as follows
	\begin{equation*}
		\|\Delta(v + \varphi)^3\|_{H^{k-2}} \leq C \|(v + \varphi)^3\|_{H^{k}} 
		\leq C\|(v + \varphi)\|^3_{H^{k}}
		\leq C\left( \|v\|^3_{H^{k}} + \|\varphi\|^3_{H^{k}}\right),
	\end{equation*}
for some $C>0$ independent of $u_0$ and $\tau.$
Consequently, we deduce
	\begin{equation*}
		\norm{	F(v)}_{H^{k-2}}\leq C\big(\norm{\varphi}_{H^{k+2}}+\|v\|^3_{H^{k}} + \|\varphi\|^3_{H^{k}}+\|\eta\|_{ {H^{k-2}}}\big).
	\end{equation*}	
	Using this bound, we have
	\begin{equation*}
		\norm{	F(v)}_{L^2(0,\tau;H^{k-2}(\TT)}\leq C\bigg(\left(\norm{\varphi}_{H^{k+2}}+\norm{\varphi}^3_{H^{k}}\right)+\sqrt{\tau}\|v\|^3_{X_{\tau}^k} +\|\eta\|_{ {L^2(0,\tau;H^{k-2}(\TT))}}\bigg).
	\end{equation*}	
	Finally, by applying the semigroup estimate  Lemma \ref{est_semigroup}, we get a positive constant $C>0$ such that
	\begin{align*}
		\norm{\Phi_{u_0}(v)}_{X_{\tau}^k}&\leq e^{C\tau}\bigg(\norm{u_0}_{H^{k}}+\|F(v)\|_{ {L^2(0,\tau;H^{k-2}(\TT))}}\bigg)\\
		& \leq e^{C\tau}\bigg(\norm{u_0}_{H^{k}(\TT)}+\norm{\varphi}_{H^{k+2}}+\norm{\varphi}^3_{H^{k}}+\|\eta\|_{ {L^2(0,\tau;H^{k-2}(\TT))}}\bigg)+e^{C\tau}\sqrt{\tau}\|v\|^3_{X_{\tau}^k}.
	\end{align*}
	For any $v_1, v_2 \in X_{\tau}^k,$ we obtain the following estimate
	\begin{align*}
		\norm{\Phi_{u_0}(v_1)-\Phi_{u_0}(v_2)}_{X_{\tau}^k}&\leq e^{C\tau} \norm{F(v_1)-F(v_2)}_{{ {L^2(0,\tau;H^{k-2}(\TT))}}}\\
		&\leq e^{C\tau} \norm{v_1-v_2}_{X_{\tau}^k}\sqrt{\tau}\left(\norm{v_1}^2_{X_{\tau}^k}+\norm{v_2}^2_{X_{\tau}^k}+\norm{\varphi}^2_{H^{k}}\right).
	\end{align*}
	Next, let us denote $R:=2e^{C\tau}\bigg(\norm{u_0}_{H^{k}}+\norm{\varphi}_{H^{k+2}}+\norm{\varphi}^3_{H^{k}}+\|\eta\|_{ {L^2(0,\tau;H^{k-2}(\TT))}}\bigg).$
	We now choose $v_1, v_2\in \mathcal{B}_{\tau}^k(R):=\{v\in X_{\tau}^k: \norm{v}_{X_{\tau}^k}\leq R \}$. We can always find $\theta>0$ sufficiently small such that for every $\tau\leq \theta$, the following conditions hold
	\begin{equation*}
		\begin{cases}
			2e^{C\tau}\sqrt{\tau}R^2\leq 1,\\
			e^{C\tau}\sqrt{\tau}\left(2R^2+\norm{\varphi}^2_{H^{k}}\right)\leq \frac{1}{2}.
		\end{cases}
	\end{equation*}
	Under these assumptions, one can easily check that $\Phi_{u_0}$ is a contraction on $\mathcal{B}_{\theta}^k(R)$ and thus has a unique fixed point $u\in \mathcal{B}_{\theta}^k(R)$, which gives a solution of \eqref{eq_extension} for small $\theta>0,$ moreover we have
	\begin{equation*}
		\sup_{0\leq t\leq \theta}\norm{u(t)}_{H^k}\leq R.
	\end{equation*}  
Thus, there exists $\mathcal T(u_0, \varphi,\eta)>0$ such that $[0,\mathcal T)$ be the maximal interval of existence of $u$, solution of  \eqref{eq_extension}. If $\mathcal T < +\infty$, then $\|u(t)\|_{H^{k}} \to +\infty$ as $t \to \mathcal T^{-}$, otherwise 
$u$ could be extended beyond $\mathcal T$, which contradicts the maximality of $\mathcal T$. Thus, for any $0<T<\mathcal T$, equation \eqref{eq_extension} has a unique solution $u\in X_T^k$.
\end{proof}

	\subsection{Proof of \Cref{max}}\label{pr_of_2}
	Let $T<\mathcal T(u_0,\varphi,\eta)$ and $\hat u_0\in H^k(\TT)$, $ \varphi\in H^{k+2}(\TT)$ and $ \eta \in L^2_{\mathrm{loc}}(0,\infty;H^{k-2}(\TT))$. Thanks to \Cref{wellposed2}, solution $\hat u$ with $\hat u_0$, $\varphi$ and $\eta$ exists up to the maximal time $\mathcal T(\hat u_0, \varphi, \eta).$ We will show that $T<\mathcal T(\hat u_0, \varphi, \eta)),$ provided $\delta>0$ is sufficiently small. set 
	$w = u - \hat{u}$. Then $w$ satisfies the following problem when $t< \min\{\mathcal T(u_0,\varphi,\eta), T(\hat u_0, \varphi, \eta)\}$
	\begin{equation}\label{eq_diff1}
		\begin{cases}
			w_t + \Delta^2 w +\Delta w  = \Delta \left(w^3-3w^2(u+\varphi)+3w(u+\varphi)^2\right),  \\
			w(0) = w_0=u_0 - \hat{u}_0.
		\end{cases}
	\end{equation}
Let us multiply equation \eqref{eq_diff1} by $w$ and integrate it over $\TT$. Performing integration by parts and applying Young's inequality, we obtain for sufficiently small $\epsilon>0$ 
	\begin{align}\label{max_est1}
	\frac{1}{2} \frac{d}{dt}  \|w\|_{L^2}^2 
		+\norm{\Delta w}_{L^2}^2
		\leq \epsilon  \|\Delta w\|_{L^2}^2
		+ C_\epsilon \|w\|_{L^2}^2 +\mathcal I,
	\end{align}
	where $
	\mathcal I = \int_{\TT} w\Delta w^3-3\int_{\TT} w\Delta w^2(u+\varphi)+3\int_{\TT} w\Delta w(u+\varphi)^2=:\mathcal I_1+\mathcal I_2+\mathcal I_3.
	$
	Again, performing integrating by parts and using Young's inequality and noting the fact that $H^k(\TT)$ is an algebra when $k>d/2,$ we estimate $\mathcal I$ as follows
	\begin{align}\label{max_est2}
	\notag	&\mathcal I_1 = -3\int_{\TT}w^2|\nabla w|^2\,dx \leq 0,\\
	\notag	&\mathcal I_2\leq 3\norm{\Delta w}_{L^2}\norm{w^2(u+\varphi)}_{L^2}\leq \epsilon \norm{\Delta w}_{L^2}^2+C_\epsilon \norm{w}_{H^k}^4\norm{u+\varphi}_{H^k}^2\\
		&\mathcal I_3\leq 3\norm{\Delta w}_{L^2}\norm{w(u+\varphi)^2}_{L^2}\leq \epsilon \norm{\Delta w}_{L^2}^2+C_\epsilon \norm{w}_{H^k}^2\norm{u+\varphi}_{H^k}^4.
	\end{align}
Thus combining \eqref{max_est1} and \eqref{max_est2}, we obtain
\begin{align}\label{max_est10}
	\frac{d}{dt}  \|w\|_{L^2}^2 
	+\norm{\Delta w}_{L^2}^2
	\leq 
	 C \big(\|w\|_{L^2}^2 +\norm{w}_{H^k}^2\norm{u+\varphi}_{H^k}^2+\norm{w}_{H^k}^4\norm{u+\varphi}_{H^k}^4\big).
\end{align}
Next, multiplying \eqref{eq_diff1} by $(-\Delta)^k w$  when $k>d/2$ and using integration by parts we have:
	\begin{align}\label{max_est5}
	\notag	\frac{1}{2} \frac{d}{dt} \left(\|(-\Delta)^{\frac{k}{2}} w\|_{L^2}^2 \right)
		&+ \|(-\Delta)^{\frac{k+2}{2}} w\|_{L^2}^2
		\leq \epsilon  \|(-\Delta)^{\frac{k+2}{2}} w\|_{L^2}^2 \\
		&
		+ C_\epsilon  \|(-\Delta)^{\frac{k}{2}} w\|_{L^2}^2+\mathcal J.
	\end{align}
where $\mathcal J=\int_{\TT}(-\Delta)^k w \Delta w^3-3\int_{\TT} (-\Delta)^k w\Delta w^2(u+\varphi)+3\int_{\TT}(-\Delta)^k w\Delta w(u+\varphi)^2=:\mathcal J_1+\mathcal J_2+\mathcal J_3.$ 
Let us estimate the term $\mathcal J$ as above.
	\begin{align*}
	\notag|\mathcal J_1|=	\left|\int_{\TT}(-\Delta)^k w \, \Delta(w)^3\right|&=\left|\int_{\TT}(-\Delta)^{\frac{k+2}{2}} w \, (-\Delta)^{\frac{k}{2}}(w^3)\right|\\
	\notag	&\leq 
		\epsilon \|(-\Delta)^{\frac{k+2}{2}} w\|^2_{L^2}+C_{\epsilon} \|(-\Delta)^{\frac{k}{2}} (w)^3\|_{L^2}^2\\
	\notag	&\leq \epsilon \|(-\Delta)^{\frac{k+2}{2}} w\|^2_{L^2}+C_{\epsilon} \| w^3\|_{H^k}^2\\
		&\leq \epsilon \|(-\Delta)^{\frac{k+2}{2}} w\|^2_{L^2}+C_{\epsilon} \| w\|_{H^k}^6.
	\end{align*}
Proceeding with arguments similar to those used in the estimate \eqref{max_est2}, we estimate the remaining terms $\mathcal J_2, \mathcal J_3$ and after combining these with \eqref{max_est5}, obtain the following
	\begin{align}\label{max_est3}
	\notag	 \frac{d}{dt} \left(   \|(-\Delta)^{\frac{k}{2}} w\|_{L^2}^2 \right)
		+ \|(-\Delta)^{\frac{k+2}{2}} w\|_{L^2}^2
	\notag	&	\leq 
		C \bigg(\|w\|_{H^k}^2+ \|w\|_{H^k}^6 +\norm{w}_{H^k}^2\norm{u+\varphi}_{H^k}^2 \\
		&\qquad\quad  +\norm{w}_{H^k}^4\norm{u+\varphi}_{H^k}^4
		 \bigg).
	\end{align}
	Putting together \eqref{max_est10} and \eqref{max_est3}, we obtain a constant $C>0$ depending on $\Lambda$, such that following estimates holds.
	\begin{align*}
		 \frac{d}{dt} \left( \|w\|_{L^2 }^2 + \|(-\Delta)^{\frac{k}{2}} w\|_{L^2 }^2 \right)
		&
		\leq 
		C \left( \norm{w}^2_{H^k}+ \| w\|_{H^k}^6\right), \qquad t<\min\{T, \mathcal T(u_0,\varphi,\eta),\mathcal T(\hat u_0, \varphi, \eta)\}. 
	\end{align*}
Thanks to Gr\"onwall inequality, for some $M>0$, we deduce the following 
\begin{align}\label{grn5}
	 \| w(t)\|_{H^k}^2 
	&
	\leq 
	e^{M T} \left( \delta+\int_{0}^{t}\| w(s)\|_{H^k}^6\right), \qquad t<\min\{T,  \mathcal T(\hat u_0, \varphi, \eta)\}.
\end{align}
Let us denote that $ \tau:=\sup\{t< \min\{T,  \mathcal T(\hat u_0, \varphi, \eta)\}: \norm{w(t)}_{H^k}<1\}.$
We also define
\begin{equation*}
\Phi(t) := \delta + \int_0^t \|w(s)\|_{H^k}^{6}\,ds .
\end{equation*}
Inequality \eqref{grn5} implies that there exists a constant $C>0$ such that the following holds
\begin{equation*}
\frac{\dot\Phi}{\Phi^{\,3}} \le C.
\end{equation*}
Integrating this differential inequality, we obtain
\[
\Phi(t) \le
	\displaystyle 
	\delta\big(1 - 2C\delta^{\,2} t\big)^{-1/2}, 
\qquad t < \min\{T,  \mathcal T(\hat u_0, \varphi, \eta)\}.
\]
Choosing $\delta>0$ sufficiently small, we deduce that
\begin{equation*}
	\Phi(t) < C_1\,\delta, 
	\qquad t < \min\{T, \mathcal  T(\hat u_0, \varphi, \eta)\}.
\end{equation*}
Choosing $\delta $ further small enough, we have $\| w(t)\|_{H^k} <1$ for $t< \min\{T, \mathcal  T(\hat u_0, \varphi, \eta)\}$, which implies that $T < \mathcal T(\hat u_0, \varphi, \eta)$ and \eqref{eq:local_uniform_bound} holds. By \Cref{wellposed2}, we also have $\hat u\in X_T^k$, which completes the proof.

\subsection{Proof of \Cref{Lipschitz type}}\label{pr_of_3}
\begin{proof}
Let $u_0$, $\eta$, $u$, $T$, and $\Lambda$  be as in the statement. Choose $\delta>0$ given by \Cref{max} with $\varphi=0$  and let $\hat u_0\in H^k(\TT)$ satisfy $\|u_0-\hat u_0\|_{H^k}<\delta$. By \Cref{max},
the corresponding solution $\hat u=\mathcal R(\hat u_0,0,\eta)$ belongs to $X_T^k$ and satisfies
\eqref{eq:local_uniform_bound}. Set $w=u-\hat u$. Then $w$ satisfies
	\begin{equation}\label{eq_diff}
		\begin{cases}
			w_t + \Delta^2 w +\Delta w  = \Delta (w(u^2+u\hat u+\hat u^2)) & \text{in }  (0, T) \times \mathbb{T}^d, \\
			w(0) = w_0=u_0 - \hat{u}_0 & \text{in } \mathbb{T}^d.
		\end{cases}
	\end{equation}
	Multiplying equation \eqref{eq_diff} by $w$ and integrating over $\TT$, we obtain for $k>d/2$:
	\begin{align}\label{est15}
		 \frac{d}{dt} \left( \|w\|_{L^2}^2 \right)
		&+\norm{\Delta w}_{L^2}^2
		\leq 
		C\norm{w}^2_{L^2}\left(1+\norm{u}^4_{H^k}+\norm{\hat u}^4_{H^k}\right),
	\end{align}
where the nonlinear terms are estimated as below using the facts: for $k>d/2$, $H^k(\TT)\hookrightarrow L^\infty(\TT)$ and $H^k(\TT)$ is an algebra, more precisely
\begin{align*}
	\left|\int_{\TT} w \, \Delta(w\, v_1\,v_2)\right|&=\left|\int_{\TT}\Delta w \, (w\, v_1\, v_2)\right|\\
	&\leq 
	\epsilon \|\Delta w\|_{L^2}^2+C_{\epsilon}  \| w\, v_1\, v_2\|_{L^2}^2\\
	&\leq \epsilon \|\Delta w\|_{L^2}^2+C_{\epsilon}\| w\|_{L^2}^2\norm{v_1v_2}^2_{L^{\infty}}\\
	&\leq \epsilon \|\Delta w\|_{L^2}^2+C_{\epsilon} \| w\|_{L^2}^2\| v_1\|_{H^k}^2\| v_2\|_{H^k}^2.
\end{align*}
for any $v_1,v_2\in H^k(\TT)$.

		Next multiplying \eqref{eq_diff} by $(-\Delta)^k w$  when $k>d/2$ and using integration by parts we have the following:
	\begin{align}\label{est16}
		 \frac{d}{dt} \left(   \|(-\Delta)^{\frac{k}{2}} w\|_{L^2}^2 \right)
		+ \|(-\Delta)^{\frac{k+2}{2}} w\|_{L^2}^2
		&	\leq 
		C \| w\|_{H^k}^2 \left(1+ \norm{u}^4_{H^k}+\norm{\hat u}^4_{H^k}\right).
	\end{align}
At this point, we have estimated the nonlinear terms in the following manner
	\begin{align*}
	\left|\int_{\TT}(-\Delta)^k w \, \Delta(w\, v_1\,v_2)\right|&=\left|\int_{\TT}(-\Delta)^{\frac{k+2}{2}} w \, (-\Delta)^{\frac{k}{2}}(w\, v_1\, v_2)\right|\\
	&\leq 
	\epsilon \|(-\Delta)^{\frac{k+2}{2}} w\|_{L^2}^2+C_{\epsilon} \|(-\Delta)^{\frac{k}{2}} (w v_1 v_2)\|_{L^2}^2\\
	&\leq \epsilon \|(-\Delta)^{\frac{k+2}{2}} w\|_{L^2}^2+C_{\epsilon} \| (w\, v_1\, v_2)\|_{H^k}^2\\
	&\leq \epsilon \|(-\Delta)^{\frac{k+2}{2}} w\|_{L^2}^2+C_{\epsilon} \| w\|_{H^k}^2\| v_1\|_{H^k}^2\| v_2\|_{H^k}^2,
\end{align*}
for any $v_1,v_2\in H^k(\TT)$.

Adding \eqref{est15} and \eqref{est16}, we have a positive constant $C$ such that
	\begin{align*}
	 \frac{d}{dt} \left( \|w\|_{L^2 }^2 + \|(-\Delta)^{\frac{k}{2}} w\|_{L^2 }^2 \right)
		&	\leq 
		C \| w\|_{H^k}^2 \left(1+ \norm{u}^4_{H^k}+\norm{\hat u}^4_{H^k}\right).
	\end{align*}
By \eqref{eq:local_uniform_bound},
$\|\hat u\|_{C([0,T];H^k)}\leq \|u\|_{C([0,T];H^k)}+1$. Hence, Gr\"onwall's inequality gives
\begin{equation*}
\|w(t)\|_{H^k}^2
\leq \|w_0\|_{H^k}^2
\exp\left(CT\left(1+\bigl(\|u\|_{C([0,T];H^k)}+1\bigr)^4\right)\right),
\qquad t\in[0,T].
\end{equation*}
Since $\|u\|_{C([0,T];H^k)}\leq\|u\|_{X_T^k}\leq\Lambda$, this proves \eqref{stab_est}
with a constant depending only on $T$ and $\Lambda$.
\end{proof}
\subsection{Proof of Lemma~\ref{lem:parabolic_regularization}}\label{pr_of_4}
\begin{proof}
By \Cref{wellposed2},
\begin{equation}\label{eq:regularization_initial_bound}
u\in C([0,T];H^k(\TT))\cap L^2(0,T;H^{k+2}(\TT)).
\end{equation}
Since $u\in H^{k+2}(\TT)\cap L^\infty(\TT)$, it is not difficult to see that
\begin{equation}\label{eq:cubic_regularization}
\|u^3\|_{H^{k+2}}
\leq C\|u\|_{L^\infty}^2\|u\|_{H^{k+2}}
\leq C\|u\|_{H^k}^2\|u\|_{H^{k+2}},
\end{equation}
where the last inequality follows from the embedding 
$H^k(\TT)\hookrightarrow L^\infty(\TT)$.
From \eqref{eq:regularization_initial_bound} and
\eqref{eq:cubic_regularization}, we deduce that
$\Delta(u^3)\in L^2(0,T;H^k(\TT))$.

Fix $\tau\in(0,T)$ and choose $\chi\in C^\infty([0,T])$ such that
$\chi=0$ in a neighborhood of $0$ and $\chi=1$ on $[\tau,T]$. The
function $v=\chi u$ satisfies 
\begin{equation}\label{eq:cutoff_regularization}
v_t=\mathcal Av+F,\qquad v(0)=0,\qquad
F=\chi\Delta(u^3)+\chi\eta+\chi'u.
\end{equation}
Note that $F\in L^2(0,T;H^k(\TT))$. Applying
Lemma~\ref{est_semigroup} to
\eqref{eq:cutoff_regularization} with $k+2$, we obtain
$v\in C([0,T];H^{k+2}(\TT))\cap L^2(0,T;H^{k+4}(\TT))$. Since
$v=u$ on $[\tau,T]$, this proves \eqref{eq:parabolic_regularization}.
\end{proof}
\begin{rmk}
For initial data in $L^2(\TT)$, the regularization from $L^2(\TT)$ to
$H^2(\TT)$ for the uncontrolled Cahn--Hilliard equation is established
in \cite[Chapter III, Theorem 4.2]{RT88}.
Lemma~\ref{lem:parabolic_regularization} provides the corresponding
higher-order regularity estimate for smoother initial data and
nonzero forcing terms.
\end{rmk}

\section{Proof of some technical lemmas}\label{estp3}
\subsection{Estimate of $\mathcal P_3$ in \Cref{prop_asympt}}
The goal of this section is to prove the following. 
\begin{lemma}\label{estimateofp3}
	Assume $v,$ $\varphi,$ and $w$ are as in the proof of \Cref{prop_asympt}. There exists a constant $C>0$  only depending on $k$ and $d$, such that the following estimate holds
	\begin{equation}\label{eq:estimate_p3_lemma}
		\int_{\TT} (-\Delta)^{k}v\Delta((v+w)\varphi^2)\leq C \|\varphi\|_{H^{3k+2}}^2\left( \|v\|_{H^k}^2+\|w\|_{H^{k+2}}^2\right). 
	\end{equation}  
\end{lemma}
\begin{proof} 
We begin by estimating the term
\begin{equation}\label{eq:I_v}
I:=\int_{\TT} (-\Delta)^{k}v\Delta(v\varphi^2).
\end{equation}
This is the most difficult term in the left-hand side of \eqref{eq:estimate_p3_lemma}. Let $\alpha\in\mathbb N^d$ be a multi-index of length $|\alpha|:=\alpha_1+\ldots+\alpha_d$. We denote by $\alpha!$ the number $\alpha_1!\alpha_2!\ldots\alpha_d!$ and we say that for two multi-indices $\alpha$ and $\beta$, $\alpha\leq \beta$ if $\alpha_i\leq \beta_i$ for all $i\in\inter{1,d}$ and write $\alpha<\beta$ if we have strict inequality entry-wise. 

For sufficiently regular functions $u$ and $v$, we recall the usual Leibniz rule
\begin{equation}\label{leibniz_formula}
\dif{\alpha}(uv)=\sum_{\beta\leq \alpha}\binom{\alpha}{\beta}\dif{\beta}u\dif{\alpha-\beta}v
\end{equation}
where $\binom{\alpha}{\beta}:=\prod_{i=1}^{d}\binom{\alpha_i}{\beta_i}=\frac{\alpha!}{\beta!(\alpha-\beta)!}$. Let $k\in\mathbb N^*$. The operator $(-\Delta)^k$ of order $2k$ is defined as
\begin{equation*}
u \mapsto (-\Delta)^{k}u:=(-1)^k\sum_{|\alpha|=k} \frac{k!}{\alpha!}\dif{2\alpha}u.
\end{equation*}
To estimate $I$, we integrate by parts and note that
\begin{equation}\label{I_initial}
I=-\int_{\TT}v(-\Delta)^{k+1}(v\varphi^2).
\end{equation}
So, by Leibniz formula \eqref{leibniz_formula}
\begin{equation*}
(-\Delta)^{k+1}(v\varphi^2)=(-1)^{k+1}\sum_{|\alpha|=k+1}\frac{(k+1)!}{\alpha!}\sum_{\gamma\leq 2\alpha}\binom{2\alpha}{\gamma}\dif{\gamma}v\,\dif{2\alpha-\gamma}(\varphi^2).
\end{equation*}
Hence, replacing the above identity on \eqref{I_initial} we have
\begin{equation}\label{eq:I_decomposition}
\begin{split}
I&=(-1)^{k+2}\sum_{|\alpha|=k+1}\frac{(k+1)!}{\alpha!}\int_{\TT}v\dif{2\alpha}v \varphi^2 \\
&\quad +(-1)^{k+2}\sum_{|\alpha|=k+1}\frac{(k+1)!}{\alpha!}\sum_{\gamma\leq 2\alpha \atop \gamma\neq 2\alpha}\binom{2\alpha}{\gamma}\int_{\TT}v\dif{\gamma}v\,\dif{2\alpha-\gamma}(\varphi^2)=: I_{1}+I_{2}
 .\end{split}
\end{equation}
Our task is reduced to estimate the terms $I_1$ and $I_2$. To this end, we present the following auxiliary lemmas.

\begin{lemma}\label{lem_I_1}
Let $k\in\mathbb N^*$. We have that
\begin{align*}
I_1=&-\sum_{|\alpha|=k+1}\frac{(k+1)!}{\alpha!}\int_{\TT}|\dif{\alpha}v|^2\varphi^2-\sum_{|\alpha|=k+1}\frac{(k+1)!}{\alpha!}\sum_{\beta\leq \alpha \atop \beta\neq \alpha}\int_{\TT}\binom{\alpha}{\beta}\dif{\alpha}v\dif{\beta}v\dif{\alpha-\beta}(\varphi^2).
\end{align*}
\end{lemma}
\begin{proof}
To prove it, it is enough to perform integration by parts and apply Leibniz formula to the term $\int_{\TT}u\dif{2\alpha}u\varphi^2$. More precisely,
\begin{equation}\label{I1_1}
\int_{\TT}v\dif{2\alpha}v\varphi^2=(-1)^{k+1}\int_{\TT}\dif{\alpha}v\dif{\alpha}(v\varphi^2)
\end{equation}
whence
\begin{equation}\label{I1_2}
\dif{\alpha}(v\varphi^2)=\varphi^2\dif{\alpha}v+\sum_{\beta\leq \alpha \atop \beta\neq \alpha}\binom{\alpha}{\beta}\dif{\beta}v\dif{\alpha-\beta}(\varphi^2).
\end{equation}
Putting together \eqref{I1_1} and \eqref{I1_2} and replacing in $I_1$ yields the desired result. 
\end{proof}

\begin{lemma}\label{lem_I_1_remainder}
Let $\alpha\in\mathbb N^d$ be such that $|\alpha|=k+1$. For any multi-index $\beta\leq \alpha$, $\beta\neq\alpha$, we have
\begin{equation*}
\left|\int_{\TT}\dif{\alpha}v\dif{\beta}v\dif{\alpha-\beta}(\varphi^2)\right|\leq C\|\varphi\|_{H^{2k+2}}^2\|v\|_{H^k}^2.
\end{equation*}
\end{lemma}
\begin{proof}
Since $|\alpha|=k+1$, we note that $0\leq |\beta|\leq k$. In this way, the most difficult case is to estimate the one that corresponds to $|\alpha|=k+1$ and $|\beta|=k$. We note that in this case we can write $\alpha=\beta+e_i$ where $e_i$ is a canonical vector from the basis of $\mathbb R^d$. Then
\begin{align*}
{\int_{\TT}\dif{\alpha}v\dif{\beta}v\dif{\alpha-\beta}(\varphi^2)}=\int_{\TT}\dif{}_i\dif{\beta}v\dif{\beta}v\dif{}_i(\varphi^2)=-\frac{1}{2}\int_{\TT}|\dif{\beta}v|^2\dif{}_{ii}(\varphi^2).
\end{align*}
Therefore
\begin{equation*}
\left|{\int_{\TT}\dif{\alpha}v\dif{\beta}v\dif{\alpha-\beta}(\varphi^2)}\right|\leq C \|\partial_{ii}\varphi^2\|_{L^\infty}\|v\|_{H^k}^2
\end{equation*}
for all $|\beta|=k$. 

Next, we focus in the case when $|\beta|\leq k-1$. Note that $\alpha$ can be written as $\alpha=\alpha^\prime+e_i$ for some canonical vector from the basis of $\mathbb R^{d}$ and such that $|\alpha^\prime|=k$. Arguing as above
\begin{align*}
{\int_{\TT}\dif{\alpha}v\dif{\beta}v\dif{\alpha-\beta}(\varphi^2)}&=\int_{\TT}\dif{}_i\dif{\alpha^\prime}v\dif{\beta}v\dif{\alpha-\beta}(\varphi^2) \\
&=-\int_{\TT}\dif{\alpha^\prime}v\dif{\beta+e_{i}}v\dif{\alpha-\beta}(\varphi^2)-\int_{\TT}\dif{\alpha^\prime}v\dif{\beta}v\dif{\alpha-\beta+e_i}(\varphi^2).
\end{align*}
So
\begin{align*}
\notag&\left|{\int_{\TT}\dif{\alpha}v\dif{\beta}v\dif{\alpha-\beta}(\varphi^2)}\right| \\
&\quad \leq \|\dif{\alpha-\beta}\varphi^2\|_{L^\infty}\|\dif{\alpha^\prime}v\|_{L^2}\|\dif{\beta+e_i}v\|_{L^2}
+\|\dif{\alpha-\beta+e_i}\varphi^2\|_{L^\infty}\|\dif{\alpha^\prime}v\|_{L^2}\|\dif{\beta}v\|_{L^2}.
\end{align*}
for all multi-index $\beta$ such that $|\beta|\leq k-1$. Thus, $|\beta+e_i|\leq k$ and $|\alpha-\beta+e_i|\leq k+2$ and therefore
\begin{align*}
\left|{\int_{\TT}\dif{\alpha}v\dif{\beta}v\dif{\alpha-\beta}(\varphi^2)}\right| \leq C\|\varphi\|^2_{H^{2k+2}}\|v\|_{H^k}^2.
\end{align*}
This ends the proof. 
\end{proof}

\begin{lemma}\label{lem_I_2}
Let $\alpha\in\mathbb N^d$ be such that $|\alpha|=k+1$. For any multi-index $\gamma\leq 2\alpha$, $\gamma\neq2\alpha$, we have
\begin{equation*}
\left|\int_{\TT} v\dif{\gamma}v\dif{2\alpha-\gamma}(\varphi^2)\right|\leq C\|\varphi\|_{H^{3k+2}}^2\|v\|_{H^k}^2.
\end{equation*}
\end{lemma}

\begin{proof}

We separate in several cases according to the length of the multi-index $\gamma$. 

\smallskip
-- \textit{Case 1: $|\gamma|\leq k$.} First, we note that for any $\gamma\in\mathbb N^d$ such that $|\gamma|\leq k$, we have by H\"older and Cauchy-Schwarz inequalities  
\begin{equation*}
\left|\int_{\TT} v\dif{\gamma}v\dif{2\alpha-\gamma}(\varphi^2)\right|\leq \|\dif{2\alpha-\gamma}\varphi^2\|_{L^\infty} \|v\|_{L^2}\|\dif{\gamma}v\|_{L^2}\leq C\|\varphi\|_{H^{3k+2}}^2\|v\|_{H^k}^2.
\end{equation*}

\smallskip

-- \textit{Case 2: $k+1\leq|\gamma|\leq 2k$.} Let $m=|\gamma|$, we can split $\gamma=\gamma_1+\gamma_2$ with $|\gamma_1|=k$ and $|\gamma_2|=m-k\geq 1$, so by Leibniz formula
\begin{align*}
\int_{\TT} v\dif{\gamma}v\dif{2\alpha-\gamma}(\varphi^2) & = \int_{\TT} v\dif{\gamma_1}(\dif{\gamma_2} v)\dif{2\alpha-\gamma}(\varphi^2)=(-1)^{|\gamma_1|}\int_{\TT}\dif{\gamma_2}v\,\dif{\gamma_1}\left(v\dif{2\alpha-\gamma}(\varphi^2)\right) \\
&=(-1)^{k}\sum_{\beta\leq \gamma_1}\binom{\gamma_1}{\beta}\int_{\TT} \dif{\gamma_2}v\, \dif{\beta}v \dif{2\alpha-\gamma+\gamma_1-\beta}(\varphi^2).
\end{align*}
We note that $|\gamma_2|\leq k$, $|\beta|\leq |\gamma_1|=k$ and $|2\alpha-\gamma+\gamma_1-\beta|\leq 2k+1$, so by direct computations using Cauchy-Schwarz inequality we get
\begin{align*}
\left|\int_{\TT} v\dif{\gamma}v\dif{2\alpha-\gamma}(\varphi^2)\right|&\leq \sum_{\beta\leq \gamma_1}\binom{\gamma_1}{\beta} \|\dif{2\alpha-\gamma+\gamma_1-\beta}\varphi^2\|_{L^\infty} \|\dif{\gamma_2}v\|_{L^2}\|\dif{\beta}v\|_{L^2}\\
&\leq C\|\varphi\|_{H^{3k+1}}^2\|v\|_{H^k}^2.
\end{align*}

\smallskip

-- \textit{Case 3: $\gamma\leq 2\alpha$, $\gamma\neq 2\alpha$, and $|\gamma|=2k+1$}. We first observe that if $\gamma\leq 2\alpha$ and $|\gamma|=2k+2$, then necessarily $\gamma=2\alpha$, so the condition $\gamma\neq 2\alpha$ excludes that case. Now, since $\gamma\leq 2\alpha$ and $|\gamma|=2k+1$, it follows that $|2\alpha-\gamma|=1$ and hence $\gamma=2\alpha-e_i$ for some $i\in\inter{1,d}$ with $\alpha_i\geq 1$, where $e_i$ is the $i$-th canonical basis vector of $\mathbb{R}^d$. Thus, we may write $\gamma = \zeta + \xi$ with $|\zeta|=k+1$ and $|\xi|=k$ (for instance, $\zeta = \alpha$ and $\xi = \alpha - e_i$). The specific decomposition does not matter in what follows.

By Leibniz formula
\begin{align}\notag
\int_{\TT} v\dif{\gamma}v\dif{2\alpha-\gamma}(\varphi^2)&=\int_{\TT} v\dif{\zeta}(\dif{\xi}v)\dif{2\alpha-\zeta-\xi}(\varphi^2) = (-1)^{|\zeta|} \int_{\TT}\dif{\xi}v\dif{\zeta}\left(v\dif{2\alpha-\zeta-\xi}(\varphi^2)\right) \\ \label{iden_zeta_xi}
&= (-1)^{k+1} \sum_{\beta\leq \zeta} \binom{\zeta}{\beta} \int_{\TT}\dif{\xi}v\dif{\beta}v\dif{2\alpha-\xi-\beta}(\varphi^2).
\end{align}
On the other hand
\begin{align}\notag
\int_{\TT} v\dif{\gamma}v\dif{2\alpha-\gamma}(\varphi^2)&=\int_{\TT} v\dif{\xi}(\dif{\zeta}v)\dif{2\alpha-\zeta-\xi}(\varphi^2) = (-1)^{|\xi|} \int_{\TT}\dif{\zeta}v\dif{\xi}\left(v\dif{2\alpha-\zeta-\xi}(\varphi^2)\right) \\ \label{iden_xi_zeta}
&= (-1)^{k} \sum_{\beta\leq \xi} \binom{\xi}{\beta} \int_{\TT}\dif{\zeta}v\dif{\beta}v\dif{2\alpha-\zeta-\beta}(\varphi^2).
\end{align}
Adding up \eqref{iden_zeta_xi} and \eqref{iden_xi_zeta}, we see that the terms corresponding to $\beta=\zeta$ and $\beta=\xi$ cancel out, leading us to 
\begin{align}\label{eq:iden_dupl}
&\int_{\TT} v\dif{\gamma}v\dif{2\alpha-\gamma}(\varphi^2) = \frac{(-1)^k}{2}\left(-\sum_{\beta\leq \zeta \atop \beta\neq \zeta}I_{\beta}(\zeta,\xi)+ \sum_{\beta\leq \xi \atop\beta\neq \xi}I_{\beta}(\xi,\zeta)\right),
\end{align}
where we have used the notation
\begin{equation*}
I_{\beta}(\Psi,\Upsilon):= \binom{\Psi}{\beta}\int_{\TT}\dif{\Upsilon}v\dif{\beta}v\dif{2\alpha-\Upsilon-\beta}(\varphi^2).
\end{equation*}
We observe that the first term in the right-hand side of \eqref{eq:iden_dupl} (namely, the one containing $I_{\beta}(\zeta,\xi)$) can be bounded easily by following Case 1 above since $|\xi|\leq k$ and $\beta\leq \zeta$ with $\beta\neq \zeta$ implies that $|\beta|\leq k$, so
\begin{equation}\label{eq:est_i_zeta_xi}
\left|\sum_{\beta\leq \zeta \atop \beta \neq \zeta}I_{\beta}(\zeta,\xi)\right|\leq C \|\varphi\|^2_{H^{2k+2}}\|v\|^2_{H^k}. 
\end{equation}
For the last term in \eqref{eq:iden_dupl}, we observe that $|\zeta|=k+1$ but $\beta\leq \xi$ with $\beta\neq \xi$ imply that $|\beta|\leq k-1$, which falls into Case 2 above where we can pass first-order derivatives to the lower order terms. With some direct computations and arguing as in Case 2, we can prove that
\begin{equation}\label{eq:est_i_xi_zeta}
\left|\sum_{\beta\leq \xi \atop \beta \neq \xi}I_{\beta}(\xi,\zeta)\right|\leq C\|\varphi\|^2_{H^{2k+2}}\|v\|_{H^k}^2.
\end{equation}
Thus, using \eqref{eq:est_i_zeta_xi}--\eqref{eq:est_i_xi_zeta} in \eqref{eq:iden_dupl} yield that 
\begin{equation*}
\left|\int_{\TT} v\dif{\gamma}v\dif{2\alpha-\gamma}(\varphi^2)\right|\leq C\|\varphi\|^2_{H^{2k+2}}\|v\|^2_{H^k}.
\end{equation*}
Lastly, combining Cases 1 to 3 yields the desired estimate. 
\end{proof}

\noindent\emph{Continuation of the proof of Lemma \ref{estimateofp3}.}
Combining Lemmas \ref{lem_I_1}--\ref{lem_I_2} with the decomposition \eqref{eq:I_decomposition}, we obtain
\begin{equation}\label{eq:I_bound}
I\leq C\|\varphi\|_{H^{3k+2}}^2\|v\|_{H^k}^2.
\end{equation}
We now return to the term on the left-hand side of \eqref{eq:estimate_p3_lemma}. We have
\begin{equation}\label{eq:vw_decomposition}
\int_{\TT}(-\Delta)^k v\,\Delta((v+w)\varphi^2)
=I+\int_{\TT}(-\Delta)^k v\,\Delta(w\varphi^2),
\end{equation}
Thus, it only remains to estimate the last term in \eqref{eq:vw_decomposition}. By the algebra property of $H^{k+2}(\mathbb T^d)$, we obtain
\begin{align}\label{eq:w_term_bound}
	\left|\int_{\mathbb T^d}(-\Delta)^k v\,\Delta(w\varphi^2)\,dx\right|
	&\leq \|v\|_{H^k}\|\Delta(w\varphi^2)\|_{H^k}\\
	&\leq C\|v\|_{H^k}\|w\|_{H^{k+2}}\|\varphi\|_{H^{k+2}}^2.
\end{align}
Combining \eqref{eq:I_bound} and \eqref{eq:w_term_bound} in \eqref{eq:vw_decomposition}, and applying Young's inequality, proves \eqref{eq:estimate_p3_lemma}.
\end{proof}
%
%
%

\section*{Acknowledgments} Subrata Majumdar expresses his sincere gratitude to Shirshendu Chowdhury for introducing him to the geometric control approach, drawing his attention to the related literature and for fruitful discussions. He also thanks  Debanjit Mondal for fruitful discussions on this topic. Finally, he is grateful to Debayan Maity for many fruitful discussions and for his valuable comments and suggestions. The authors also thank the Editor for helpful comments and suggestions that improved the presentation of the manuscript.

\section*{Declarations}
\noindent{\bf Conflict of interest.} The authors have no conflict of interest to declare that are relevant to the content of this article. 

\medskip

\noindent{\bf Data Availability.} No data were used or generated in the preparation of this manuscript.	

\bibliographystyle{alpha}
\small{\bibliography{bib_nonlinear_coupled}}


\medskip
\medskip
\medskip
\medskip

\begin{flushleft}

\bigskip

\textbf{Víctor Hernández-Santamaría}\\
Departamento de Matem\'aticas, Facultad de Ciencias\\
Universidad Nacional Autónoma de México \\
Circuito Exterior, C.U.\\
04510, Coyoacán, CDMX, Mexico\\
\texttt{victor.santamaria@ciencias.unam.mx}

\bigskip
\bigskip
	
	\textbf{Subrata Majumdar and Luz de Teresa}\\
	Instituto de Matemáticas\\
	Universidad Nacional Autónoma de México \\
	Circuito Exterior, Ciudad Universitaria\\
	04510 Coyoacán, Ciudad de México, México\\
	\texttt{subrata.majumdar@im.unam.mx}\\
	\texttt{ldeteresa@im.unam.mx}
\end{flushleft}

\end{document}